\documentclass[11pt]{article}
\usepackage[letterpaper]{geometry}
\usepackage[parfill]{parskip}
\usepackage{amsmath,amsthm,amssymb,bbm,bm}
\usepackage{mathtools}
\usepackage{cases}
\usepackage{graphicx}
\usepackage{tabularx}
\usepackage{microtype}
\usepackage{enumitem}
\usepackage{dsfont}
\usepackage[normalem]{ulem}

\usepackage{authblk}

\usepackage{url}
\usepackage[colorlinks,citecolor=blue,urlcolor=blue,linkcolor=blue,linktocpage=true]{hyperref}
\usepackage{cleveref}
\crefformat{equation}{(#2#1#3)}
\crefrangeformat{equation}{(#3#1#4) to~(#5#2#6)}
\crefname{equation}{}{}
\Crefname{equation}{}{}

\usepackage[round]{natbib}
\usepackage{multirow}

\newtheoremstyle{mythmstyle}
  {8 pt} 
  {3 pt} 
  {} 
  {} 
  {\bfseries} 
  {.} 
  {.5em} 
  {} 

\theoremstyle{plain}

\makeatletter
\def\thm@space@setup{%
  \thm@preskip=6pt plus 1pt minus 1pt
  \thm@postskip=\thm@preskip 
}
\makeatother

\newtheorem{theorem}{Theorem}[section]
\newtheorem{lemma}[theorem]{Lemma}
\newtheorem{corollary}[theorem]{Corollary}
\newtheorem{proposition}[theorem]{Proposition}

\newtheorem{remark}[theorem]{Remark}
\newtheorem{example}[theorem]{Example}
\newtheorem*{example*}{Example}
\newtheorem{definition}[theorem]{Definition}
\newtheorem*{definition*}{Definition}
\newtheorem{assumption}[theorem]{Assumption}
\newtheorem*{remark*}{Remark}

\crefname{definition}{\textbf{definition}}{definitions}
\Crefname{definition}{Definition}{Definitions}
\crefname{assumption}{\textbf{assumption}}{assumptions}
\Crefname{assumption}{Assumption}{Assumptions}

\newcommand{\cv}{{\mathrm{cv}}}
\newcommand{\ssp}{{\mathrm{ss}}}
\newcommand{\aux}{{\mathrm{aux}}}
\newcommand{\tr}{{\mathrm{tr}}}
\newcommand{\te}{{\mathrm{te}}}
\newcommand{\D}{{\mathcal{D}}}

\newcommand{\E}{{\mathcal{E}}}

\renewcommand{\S}{{\mathcal{S}}}
\newcommand{\X}{{\mathcal{X}}}
\newcommand{\var}{\mathrm{Var}}

\begin{document}
\allowdisplaybreaks
\title{A Modern Theory of Cross-Validation through the Lens of Stability}

\author[1]{Jing Lei}
\affil[1]{Carnegie Mellon University}

\maketitle

\begin{abstract}
Modern data analysis and statistical learning are characterized by two defining features: complex data structures and black-box algorithms. The complexity of data structures arises from advanced data collection technologies and data-sharing infrastructures, such as imaging, remote sensing, wearable devices, and genomic sequencing. In parallel, black-box algorithms—particularly those stemming from advances in deep neural networks—have demonstrated remarkable success on modern datasets. This confluence of complex data and opaque models introduces new challenges for uncertainty quantification and statistical inference, a problem we refer to as ``black-box inference''.

The difficulty of black-box inference lies in the absence of traditional parametric or nonparametric modeling assumptions, as well as the intractability of the algorithmic behavior underlying many modern estimators. These factors make it difficult to precisely characterize the sampling distribution of estimation errors. A common approach to address this issue is post-hoc randomization, which includes permutation, resampling, sample splitting, cross-validation, and noise injection. When combined with mild assumptions, such as exchangeability in the data-generating process, these methods can yield valid inference and uncertainty quantification.

Post-hoc randomization methods have a rich history, ranging from classical techniques like permutation tests, the jackknife, and the bootstrap, to more recent developments such as conformal inference. These approaches typically require minimal knowledge about the underlying data distribution or the inner workings of the estimation procedure. While originally designed for varied purposes, many of these techniques rely, either implicitly or explicitly, on the assumption that the estimation procedure behaves similarly under small perturbations to the data. This idea, now formalized under the concept of \emph{stability}, has become a foundational principle in modern data science. Over the past few decades, stability has emerged as a central research focus in both statistics and machine learning, playing critical roles in areas such as generalization error, data privacy, and adaptive inference.

In this article, we investigate one of the most widely used resampling techniques for model comparison and evaluation---cross-validation (CV)---through the lens of stability. We begin by reviewing recent theoretical developments in CV for generalization error estimation and model selection under stability assumptions. We then explore more refined results concerning uncertainty quantification for CV-based risk estimates. By integrating these research directions, we uncover new theoretical insights and methodological tools. Finally, we illustrate their utility across both classical and contemporary topics, including model selection, selective inference, and conformal prediction.
\end{abstract}

\paragraph{Acknowledgements} This monograph grew out of a short course I taught at Seoul National University in Fall 2024. I am deeply grateful to the series editors, Rina Barber and Ryan Tibshirani, for inviting me to write this monograph and for their continued support throughout the process. I also thank the anonymous reviewers for their careful reading and constructive feedback. Finally, I am indebted to Yuchen Chen, Hao Lee, Zhixin Lai, Tianyu Zhang, Lorenzo Testa, and Julian Braganza for their help in proofreading the draft. Of course, I am solely responsible for any remaining errors. This monograph was written while I was supported by NSF grants DMS-2515687 and DMS-2310764.

\tableofcontents

\newpage

\section{Introduction}

The primary goal of data science and statistics is to uncover and understand the underlying
structures in observed data. Mathematically, these structures are often represented by a
parameter that belongs to a set of possible values known as the parameter space. 
Because parameter estimates are derived from data subject to random noise, it is crucial
to assess their quality and quantify the associated uncertainty.

There are two widely used approaches for accuracy assessment and uncertainty
quantification. The first focuses on measuring the distance between the estimated parameter
and the true value, and on characterizing the variability of the estimate around this target.
Although conceptually straightforward, this approach typically relies on correct specification
and parameterization of the data-generating process, making it sensitive to model
misspecification.

An alternative approach emphasizes evaluating the predictive performance of the estimated parameter on future, unseen data. In his influential paper, \cite{breiman2001statistical}, Leo Breiman strongly advocated for this perspective, arguing that prediction should be the central focus of modeling in data science. Within this framework, a key task is to assess the predictive accuracy of a fitted model or parameter on new data points. A fitted model is commonly obtained by minimizing the empirical risk over a class of candidate parameters using training data. However, the model’s performance on future data---referred to as the out-of-sample risk or test error---is typically worse than its performance on the training data, known as the in-sample risk or training error. This gap arises because the noise in training data influences both the fitting procedure and the evaluation of in-sample performance. The discrepancy between training and test error, referred to as optimism by \cite{efron1986biased}, has become a central topic in the study of model selection and generalization error in both statistics and machine learning. This article focuses on this predictive performance evaluation approach, as it relies on weaker modeling assumptions and has broader applicability.

\subsection{Cross-validation: A Brief History}
The concept of cross-validation has existed for nearly a century. By definition, the test error can be estimated without bias if an independent dataset is available for evaluating the risk of a fitted model. In practice, this is commonly achieved by splitting the data into two parts: one for model fitting and the other for risk evaluation. This idea of sample splitting is the foundation of cross-validation and can be traced back as early as the 1930s. Larson \citep{larson1931shrinkage} was among the first to formally study this approach, noting the inherent optimism in regression analysis and emphasizing the importance of out-of-sample accuracy. As he observed:
\begin{quote}
   \emph{\ldots\ ordinarily the practical employment of a regression equation involves its use with data other than those from which it was derived.}
\end{quote}

Since Larson’s initial proposal of sample splitting for out-of-sample risk estimation, the concept of cross-validation (CV) has emerged, evolved, and found application in various contexts. An early and clear formulation of leave-one-out cross-validation was provided by Mosteller and Tukey \citep{mosteller1968data}. Building on these developments, CV was later formalized as a general method for model assessment and selection in the works of \cite{Stone74,Geisser75}, which also provide additional historical insights into the use of CV during the 1950s and 1960s.  
Interestingly, cross-validation originated with an emphasis on \emph{predictive} accuracy rather than parameter estimation, as captured by Geisser’s observation \cite{Geisser75}:
\begin{quote}
\emph{\ldots\ the prediction of observables or potential observables is of much greater relevance than the estimation of what are often artificial constructs—parameters.}
\end{quote}

Since the mid-1970s, cross-validation (CV) has continued to find applications across nearly all areas of statistical inference and machine learning. Because CV requires only an estimator and a loss function, it is arguably the most broadly applicable method for assessing and comparing model quality. In contrast, alternative approaches often require specifying a likelihood function and/or a complexity measure for the candidate models. This simplicity has made CV a subject of study in its own right, with significant efforts devoted to understanding and enhancing its theoretical foundations and practical performance. The body of literature on the application, theory, and refinement of CV is extensive; therefore, we do not attempt a comprehensive review in this article. For a thorough survey of the CV literature from the 1970s through the 2010s, readers are encouraged to consult \cite{arlot2010survey}.

\subsection{Risk Estimation and Model Selection Properties of CV}\label{subsec:1.2-estimation-selection}
In the theoretical study of cross-validation (CV), two of the most extensively examined properties are risk estimation and model selection.

\emph{Risk estimation} refers to the ability to quantify the predictive risk of a fitted model on future, unseen data. Alternatively, one might consider the estimation risk relative to a true underlying parameter, which requires assuming the existence of a true model for the data-generating process. In this article, we focus on predictive risk estimation, also known as generalization error estimation in the machine learning literature. 

General results of CV risk consistency mostly rely on empirical process theory \citep{devroye1988automatic,vaart2006oracle,lecue2012oracle}.
Results in specific regression settings include nonparametric regression and density estimation \citep{devroye2001combinatorial,gyorfi2006distribution} and the Lasso \citep{homrighausen2017risk}. 

\emph{Model selection} concerns the ability to identify the best, or nearly the best, model from a set of candidates. This task is often divided into two variants: (i) identifying the true model that generated the data, and (ii) selecting the model with the best predictive performance, regardless of whether a true model exists within the candidate set. These are referred to as \emph{model selection for identification} and \emph{model selection for estimation}, respectively, in \cite{arlot2010survey}. Our focus in this article is primarily on the latter, as it requires fewer assumptions and subsumes the identification problem as a special case when the true model yields the smallest predictive risk.

The model selection properties of CV are generally more challenging to analyze. Early theoretical results indicate that conventional versions of CV---such as leave-one-out and $K$-fold CV---can be inconsistent even in simple linear regression settings \citep{Zhang93,Shao93}. These inconsistencies were later extended to more general frameworks in \cite{yang2007consistency}. The underlying reason is intuitive: standard CV procedures struggle to distinguish between the best parametric model and a slightly overfitting alternative, as both models tend to have estimation errors on the order of $1/\sqrt{n}$. Evaluating their performance on a sample of size $O(n)$ leads to a constant-order probability of selection error. To address this limitation, CV must be applied in a nonstandard way using a vanishing fraction of the data for model fitting and the remainder for risk evaluation.

Additionally, it has been widely observed that CV often favors models with slight under-smoothing. While such models may yield near-optimal predictive accuracy, they frequently suffer from reduced interpretability. To counteract this under-regularization tendency, several heuristic adjustments have been proposed, including the “.632 bootstrap” \citep{efron1997improvements} and the “1-standard-deviation rule” \citep{tibshirani2009bias}. Other examples and variants include \cite{yu2014modified,feng2019restricted}.

\subsection{Stability and Cross-validation}
In a seminal work, \cite{bousquet2002stability} established the risk estimation consistency of cross-validation (CV) under the assumption of algorithmic stability. Broadly speaking, stability refers to the property that the output of a learning algorithm remains relatively unchanged when small perturbations are made to the training data. This notion is closely related to robustness \citep{huber2011robust} and aligns with the concept of sensitivity in differential privacy \citep{dwork2009differential,dwork2014algorithmic}. As a technical condition, stability is natural and often satisfied by reasonable estimators and learning procedures.

The significance of \cite{bousquet2002stability} lies in demonstrating that stability alone can yield meaningful guarantees for CV-based risk estimates---without the need for complexity measures or empirical process theory. Subsequent work has further cemented stability as a central concept in modern statistical learning. Notably, it has been linked to learnability \citep{shalev2010learnability,wang2016learning} and plays an important role in broader statistical methodology \citep{yu2013stability,yu2020veridical}. Leveraging stability assumptions, \cite{zhang2023online} extended the model selection consistency results of \cite{yang2007consistency} to an online variant of cross-validation.

Another major theoretical advancement was contributed by \cite{bayle2020cross,austern2020}, who established central limit theorems for CV-based risk estimates under general stability conditions. These results enable the construction of confidence sets for the best candidate model, as further developed in \cite{lei2020cross,kissel2022high}, connecting cross-validation to the model confidence set literature. In the context of model confidence sets, the goal is not to identify a single best model, but to find the subset of models whose predictive risks are statistically indistinguishable from the minimum. This inferential target has a rich history in econometrics and statistics \citep{jiang2008fence,hansen2011model,gunes2012confidence,ferrari2015confidence}.

Model confidence sets are especially useful when model selection consistency is difficult to achieve, such as in settings with a limited signal-to-noise ratio. They offer a flexible and uncertainty-aware alternative to point selection and can be viewed as a rigorous counterpart to heuristic CV adjustments mentioned above. For instance, \cite{lei2020cross} showed that selecting the most parsimonious model from the CV-based confidence set can yield consistent model selection in classical problems---such as best subset selection in linear regression—where standard $K$-fold CV fails.

\subsection{Cross-validation Beyond Model Selection and Evaluation}
Beyond model comparison and selection, the cross-validation principle---estimating a parameter and evaluating a functional on disjoint subsamples---has played a pivotal role in nonparametric functional estimation and semiparametric inference.

Consider a general setup where the goal is to estimate the expected value of $g(X;\gamma)$---denoted as $\theta = \mathbb{E}[g(X; \gamma)]$---from independent and identically distributed (i.i.d.) observations $(X_1, \ldots, X_n)$, where $g(\cdot, \cdot)$ is a known function and $\gamma$ is a (possibly infinite-dimensional) parameter of the data-generating distribution. For example, if $X_i$ is a one-dimensional variable with unknown density $f(x)$, and we define $\gamma = f$, $g(x, \gamma) = -\log \gamma(x)$, then $\theta = -\mathbb{E}[\log f(X)]$ corresponds to the entropy of the distribution.

A common strategy for estimating such a functional is through influence function-based methods. This involves identifying an influence function $\varphi(x)$ associated with the functional $\theta$, and estimating it from data. The final estimator takes the form of a sample average: $\frac{1}{n} \sum_{i=1}^n \hat\varphi(X_i)$, where $\hat\varphi$ is a data-driven approximation of $\varphi$. Since $\hat\varphi$ depends on the data, inference must account for this randomness. A widely adopted solution is cross-fitting, which, like CV, uses disjoint subsamples to separately estimate $\hat\varphi$ and evaluate its expectation. This idea dates back to early work in nonparametric and semiparametric statistics \citep{hasminskii1979nonparametric,pfanzagl1985contributions,schick1986asymptotically,bickel1988estimating}, and has been extensively developed in recent literature \citep{chernozhukov2019inference,kennedy2024semiparametric}.

In cross-fitting, a crucial requirement is that the estimated influence function $\hat\varphi$ converges sufficiently fast to a well-defined target $\varphi$. However, with the advent of algorithmic stability, it has been shown that explicit sample-splitting may not be necessary if $\hat\varphi$ is stable \citep{chen2022debiased}. More recently, \cite{zhang2024winners} demonstrated that even when $\hat\varphi$ does not converge to a fixed parameter, stability can still suffice to ensure valid one-sided inference for the target functional.

\subsection{Outline of This Article}
In \Cref{sec:bousquet}, we introduce the foundational setup of cross-validation along with the relevant stability conditions, and present the risk consistency and concentration results established by \cite{bousquet2002stability}. In addition, we provide a collection of general probabilistic tools concerning sub-Weibull random variables and their concentration properties. These tools can be used to extend the results of \cite{bousquet2002stability} to settings involving more general, potentially unbounded loss functions. As such, they will serve as a key component in the analysis presented in subsequent sections.

In \Cref{sec:3-model-selection}, we present results on the model selection consistency of cross-validation. We begin with batch mode, revisiting the consistency result of \cite{yang2007consistency}, and provide a simplified proof.  As an extension, we introduce an online model selection framework and propose an online variant of cross-validation \citep{zhang2023online}. While this online method resembles classical leave-one-out cross-validation---which is known to be inconsistent in the batch setting---it can, somewhat unexpectedly, achieve consistent model selection under appropriate stability conditions.

\Cref{sec:4-clt} focuses on central limit theorems (CLTs) for cross-validation risk estimates. We begin by presenting the CLT results under two different centering schemes, originally developed by \cite{bayle2020cross} and \cite{austern2020}, while streamlining and simplifying some of the proofs. In our treatment of the deterministic centering results, several technical conditions from the original work are either weakened or removed. We also unify the two centering approaches by identifying a condition under which they become equivalent. Finally, we discuss the extensions in \cite{kissel2022high}, which develop high-dimensional Gaussian approximation results in the spirit of \cite{chernozhukov2013gaussian,chernozhukov2022improved}.

\Cref{sec:5-applications} presents several important applications of the theoretical developments in \Cref{sec:bousquet} and \Cref{sec:4-clt} across various contexts. The first application introduces a model confidence set based on Gaussian comparisons of CV risk estimates, representing a methodological advancement that emphasizes constructing confidence intervals for the \emph{differences} in model performance, rather than for each model individually. Building on this framework, we establish a model selection consistency result for selecting the most parsimonious model within the $K$-fold CV model confidence set, offering a principled correction for the well-documented overfitting tendency of CV. In \Cref{subsec:5.3-argmin}, we leverage the CLT results to propose a novel method for high-dimensional mean testing, providing a simple alternative to selective inference techniques. Extending this line of inquiry, \Cref{subsec:5.4-argmin-conf} presents a procedure for constructing confidence sets for the location of the minimum entry in a vector observed with noise. Finally, in \Cref{subsec:5.5-cross-conf}, we establish an asymptotic coverage guarantee for the cross-conformal prediction method under appropriate stability conditions.

\Cref{sec:6-stability} presents several examples of stable estimators. \Cref{subsec:6.1-stab-sgd-1,subsec:6.2-stab-sgd-2} focus on the stochastic gradient descent (SGD) algorithm under general regularity conditions, providing batch-mode stability results needed for theorems in \Cref{sec:bousquet,sec:4-clt}. \Cref{subsec:6.3-stab-diff} examines the stability of the difference between two competing estimators in a prototypical nonparametric regression setting. Finally, \Cref{subsec:6.4-stab-online} explores the stability properties of various SGD variants in the online setting, as required for the online model selection consistency results in \Cref{subsec:3.3-rv}.

\subsection{Other Related Topics}
 This article focuses on recent developments in cross-validation theory, particularly through the lens of stability conditions. While cross-validation is a widely used tool for model selection and risk estimation, it is by no means the only approach. Classical alternatives include regularized likelihood and data compression-based methods, such as AIC, BIC, Mallow's \( C_p \), and the Minimum Description Length (MDL) principle. For comprehensive reviews and discussions on these methods and their relationships to cross-validation, see \cite{shao1997asymptotic,ding2018model}. Another important line of work is penalized empirical risk minimization, developed under the framework of uniform convergence and concentration inequalities \citep{bartlett2002model}.

In this article, we focus on model comparison and risk estimation in the setting of independent and identically distributed (i.i.d.) data. However, cross-validation has also been adapted to settings involving dependent data, such as time series and spatio-temporal data. A brief survey and references to foundational works in this area can be found in Section~8.1 of \cite{arlot2010survey}. Moreover, cross-validation techniques have been extended to data matrices with dependent entries---for instance, in determining the number of principal components in singular value decomposition \citep{owen2009bi,hoff2007modeling}, and in model selection for network data \citep{chen2018network,li2020network,qin2022consistent}.

Stability plays a vital role across nearly all areas of data science. It is one of the three foundational principles of the PCS framework---\textbf{P}redictability, \textbf{C}omputability, and \textbf{S}tability---introduced in \citet{yu2020veridical}. For practical examples of how stability impacts real-world data analyses, as well as systematic procedures for achieving it through perturbation analysis, see \citet{yu2024veridical}. In the theory and method literature, stability has also proven useful in various aspects of predictive inference beyond cross-validation. These include analyses of generalization error \citep{buhlmann2002analyzing,shalev2010learnability,wang2016learning}, variable selection \citep{meinshausen2010stability,liu2010stability}, jackknife methods \citep{barber2021predictive,steinberger2023conditional}, and conformal prediction \citep{lei2018distribution,angelopoulos2024theoretical}.

\subsection{Notation}
The mathematical notation used in this article is mostly standard, transparent, and self-contained.  Most of the symbols have explanations near their appearances.   Here we collect some of the most commonly used notation and symbols for the reader's convenience. The notation specifically related to the sample-splitting scheme of cross-validation is introduced in \Cref{subsec:prelim}.
\begin{itemize}
\item We use $(X_i:i\ge 0)$ to denote an infinite sequence of i.i.d. (independent and identically distributed) sample points, and $\D_n=(X_i:1\le i\le n)$ the dataset consisting of the first $n$ sample points.  
\item $\mathcal X$ denotes the space of a single sample point, and $\mathcal X^{\otimes n}$ the $n$-fold product space, so that $\D_n\in \mathcal X^{\otimes n}$.
\item For a positive integer $n$, $[n]\coloneqq \{1,2,...,n\}$.
    \item $\mathbb E_X(\cdot)$ means taking expectation over the randomness of $X$, conditioning on everything else.
    \item $\mathbb R^d$ denotes the $d$-dimensional Euclidean space. 
    \item $\|\cdot\|_q$ denotes the $L_q$-norm of a random variable. More concretely, $\|X\|_q=(\mathbb E|X|^q)^{1/q}$ if $1\le q<\infty$, and $\|\cdot\|_\infty$ is the essential supremum of a random variable.  For a function $h$ acting on a random variable $X$, $\|h\|_{q,X}=\|h(X)\|_q$.
    \item $\|\cdot\|$ denotes the Euclidean, Hilbert, or Banach-space norm, depending on the context, and $\|\cdot\|_{\rm op}$ denotes the operator norm of a matrix or linear operator.
    \item For two non-random sequences $a_n$ and $b_n$ such that $b_n\ge 0$,
    \begin{itemize}
        \item   $a_n=o(b_n)$ means $|a_n|/b_n\rightarrow 0$;
        \item $a_n=O(b_n)$ means $|a_n|/b_n\le c$ for a positive constant $c$. This is sometimes written as $a_n\lesssim b_n$;
        \item $a_n=\omega(b_n)$ means $b_n=o(|a_n|)$;
        \item $a_n=\Omega(b_n)$ means $b_n=O(|a_n|)$, or equivalently $|a_n|\gtrsim b_n$;
        \item $a_n\asymp b_n$ means $|a_n|\lesssim b_n \lesssim |a_n|$;
        \item $a_n=\tilde o(b_n)$ means $|a_n|=o(b_n\cdot c_n)$, where $c_n$ is a finite-degree polynomial of $\log n$ and perhaps the logarithm of other model quantities such as the dimensionality of the data and/or the number of models being compared.
    \end{itemize}
    \item $\stackrel{P}{\rightarrow}$ denotes convergence in probability.
    \item $\rightsquigarrow$ denotes convergence in distribution.
    \item For two random variables $U$, $V$ defined on the same space and $q\ge 1$, $d_{W_q}(U,V)$ denotes the Wasserstein-$q$ distance between $U$ and $V$.
\item For two sequences of possibly random variables $U_n$ and $V_n$ such that $V_n\ge 0$,
\begin{itemize}
    \item $U_n=o_P(V_n)$ means $U_n/|V_n|\stackrel{P}{\rightarrow}0$;
    \item $U_n=O_P(V_n)$ means that for each $\epsilon>0$ there exists $M>0$ such that $\sup_n \mathbb P(|U_n/V_n|\ge M)\le \epsilon$.
\end{itemize}
    \item For a positive integer $m$: $\mathbb S_{m-1}$ denotes the $(m-1)$-dimensional simplex in $\mathbb R^m$: $\{(w_1,...,w_m)\in \mathbb R^m:~w_r\ge 0,~\sum_{r=1}^m w_r=1\}$.
    \item For a smooth function $h$ defined on a Euclidean or Hilbert space, $\dot h(\cdot)$ and $\ddot h(\cdot)$ denote the gradient and Hessian of $h$, respectively.
\end{itemize}

\newpage


\section{Risk Consistency of CV by Stability}\label{sec:bousquet}
In this section, we introduce the relevant notations for cross-validation (\Cref{subsec:prelim}) and present results from \cite{bousquet2002stability} on the risk consistency (\Cref{subsec:bousquet_consistency}) and concentration (\Cref{subsec:bousquet_concentration}) properties of cross-validation under stability conditions. In addition, we prove several useful concentration inequalities that will be employed in later analysis (\Cref{subsec:concentration_ineq}).

\subsection{The Basic Setup}\label{subsec:prelim}
In this article, we focus on the abstract setting of independent, identically distributed (i.i.d.) data.
Let $\D_n=(X_1,...,X_n)$ be a dataset consisting of $n$ i.i.d. sample points, where each $X_i\in\X$ is generated from a common probability measure $P_X$.

In a statistical learning task, the goal is to use the data $\D$ to learn a parameter $f$ 
within a parameter space $\mathcal{F}$. In classical parametric inference, 
$\mathcal{F}$ is typically a finite-dimensional Euclidean space. In contrast, 
nonparametric inference allows $\mathcal{F}$ to be an infinite-dimensional functional space. 
The general theory developed in this article encompasses both settings; however, 
the results are often more relevant and yield richer insights in high-dimensional 
or infinite-dimensional cases, where the bias–variance trade-off becomes more subtle.

In applications of cross-validation, the estimation of $f$ may be carried out on datasets with different sizes. To keep track of the potentially varying fitting sample size, we consider
a collection of estimators indexed by different sample sizes. Let $\hat f:\bigcup_{n=1}^\infty \X^{\otimes n}\mapsto \mathcal F$ be an estimation procedure that maps a sample of any size to the parameter space.  We assume that $\hat f$ is permutation invariant (or symmetric).
\begin{definition}[Symmetric Estimator]\label{asm:symmetric-f}
We say an estimator $\hat f:\bigcup_{n=1}^\infty \X^{\otimes n}\mapsto \mathcal F$ is symmetric if
$\hat f(x_1,...x_n)=\hat f(x_{\sigma(1)},...,x_{\sigma(n)})$ for all $n\ge 1$, $(x_1,...,x_n)\in \X^{\otimes n}$, and permutations $\sigma:[n]\mapsto[n]$.    
\end{definition}

Let $\ell:\mathcal \X\times\mathcal F\mapsto [-\infty,\infty]$ be a loss function used to evaluate the quality of model fit.  This notation covers both supervised learning and unsupervised learning.

Examples in supervised learning include regression and classification.
\begin{itemize}
    \item Regression: $X=(Y,Z)\in\mathbb R^1\times \mathbb R^p$, and $f$ is a function mapping $Z$ to $f(Z)\in\mathbb R^1$ that predicts the value of $Y$ using $Z$. A common choice of $\ell$ is the squared error: $\ell((y,z),f)=(y-f(z))^2$.
    \item Classification: $X=(Y,Z)\in\{0,1\}\times \mathbb R^p$, and $f$ is a function mapping $Z$ to $[0,1]$ that predicts the conditional probability of $Y=1$ given $Z$. A common choice of $\ell$ in classification is the negative logistic likelihood  $\ell((y,z),f)=-y\log(f(z)/(1-f(z)))+\log(1-f(z))$.
\end{itemize}

Examples in unsupervised learning include principal components analysis and mixture models.
\begin{itemize}
    \item Principal components analysis (PCA): $X\in\mathbb R^p$ and $f:\mathbb R^p\mapsto\mathbb R^p$ is a $d$-dimensional linear projection that maps $x\in\mathbb R^p$ to $f(x)$ in a $d$-dimensional linear subspace of $\mathbb R^p$, for some $1\le d<p$. The loss function is the residual squared norm $\ell(x,f)=\|x-f(x)\|^2$, where $\|\cdot\|$ denotes the Euclidean norm.
    \item Mixture models: $X\in\mathbb R^p$, $f$ is a density function (e.g., in a reproducing kernel Hilbert space or a Gaussian mixture), and $\ell(x,f)=-\log f(x)$.
\end{itemize}

The overall quality of a fitted model $f$ under the loss function $\ell(\cdot,\cdot)$ is the risk\footnote{Here we adopt the machine-learning convention, in which ``risk'' denotes the expected loss of a fixed predictor \(f\) with respect to the randomness of the evaluation (test) observation. This contrasts with decision theory, where ``risk'' refers to the overall expected loss of a procedure, averaging over both estimation (training) and evaluation (testing) randomness.
} $R(f)$:
$$R(f)=\mathbb E_X \ell(X,f)\,.$$
Here the notation $\mathbb E_{X}(\cdot)$ means taking conditional expectation over the randomness of $X$ only.
We will often consider the risk of a fitted model $\hat f=\hat f(\D_n)$, which is obtained by applying the fitting procedure $\hat f$ to the input dataset $\D_n$. To simplify notation, we
write $R(\hat f(\D_n))$ as $R(\D_n)$ and $\ell(\cdot,\hat f(\D_n))$ as $\ell(\cdot,\D_n)$ when there is no ambiguity on the fitting procedure.

Under such a setting, there are two natural quantities of interest.
The first is $R(\D_n)$, the risk of the estimated parameter obtained by applying a particular fitting procedure $\hat f$ to the fitting dataset $\D_n$. This is a random quantity, because it depends on the random fitting dataset $\D_n$.
Another quantity of interest is the population version of $R(\D_n)$:
\begin{equation}\label{eq:2.1-mu_n}
\mu_n=\mathbb E R(\D_n)\,,    
\end{equation}
which is a deterministic quantity that depends only on the fitting procedure $\hat f$, the data distribution $P_X$, and the sample size $n$.

Neither $R(\D_n)$ nor $\mu_n$ can be straightforwardly estimated without additional assumptions, such as uniform convergence over the function class or, as we will see later this article, stability of $\hat f$.  

A common challenge in estimating the risk $R(\D_n)$ or the average risk $\mu_n$ is to handle simultaneously the randomness in $\hat f(\D_n)$ (estimation) and in approximating the risk of $\hat f(\D_n)$ (evaluation). The cross-validation technique does this by splitting the dataset to create a holdout sample for empirical evaluation of the risk, and further reducing the estimation variability by repeated sample splitting.  In this article, we focus on the commonly used
$K$-fold cross-validation estimate, described as follows.
\begin{enumerate}
    \item Fix $K$ such that $n_\te\coloneqq n/K$, the test subsample size, is an integer\footnote{We assume $n/K$ to be an integer throughout this article to simplify the notation.  The scenario in which $n/K$ is not an integer does not make any qualitative difference to the methods and results.}.  Let $n_\tr=n-n_\te$ be the training subsample size.
   \item  For $k=1,...,K$, let $I_{n,k}=\{n_\te(k-1)+1,...,n_\te k\}$ be the indices of the $k$th fold of the data, $\D_{n,k}=\{X_i:i\in I_{n,k}\}$ be the corresponding $k$th fold of the data, and $\D_{n,-k}=\D_n\setminus \D_{n,k}$.   Define $I_{n,-k}=[n]\setminus I_{n,k}$.
\item Define the $K$-fold cross-validation risk estimate
\begin{equation}\label{eq:2.1-cv-risk-est}
\hat R_{\cv,n,K} = n^{-1}\sum_{k=1}^K \sum_{i\in I_{n,k}}\ell(X_i,\D_{n,-k})\,.    
\end{equation}
\end{enumerate}
We may drop the dependence on $n$ and $K$ and simply write $\hat R_{\cv}$.  We will consider the following two cases.
\begin{itemize}
    \item $K=O(1)$. This corresponds to the commonly used 5-fold or 10-fold CV.
    \item $K=K_n\rightarrow\infty$. The most typical case is $K=n$, which corresponds to the leave-one-out (LOO) CV.
\end{itemize}

While we might intuitively expect $\hat R_{\cv,n,K}$ to approximate $R(\D_n)$ or $\mu_n$, it is important to note that, by construction, $\hat R_{\cv}$ represents the risk of models trained on samples of size $n_\tr$ rather than $n$, since each term in \eqref{eq:2.1-cv-risk-est} is a loss evaluated at a parameter estimate fitted on a dataset of size $n_\tr$.

With a sample size $n_\tr$, the counterpart of $\mu_n$ is simply $\mu_{n_\tr}$.  But for the risk of the fitted model $R(\D)$, the CV procedure creates a total of $K$ versions at the sample size $n_\tr$: $\{R(\D_{n,-k}):1\le k\le K\}$.  Hence the counterpart of $R(\D_n)$ at  sample size $n_\tr$ targeted by the CV risk estimate $\hat R_{\cv,n,K}$ in \eqref{eq:2.1-cv-risk-est} should be their average:
\begin{equation}\label{eq:2.1-R-bar}
\bar R_{\cv,n,K} \coloneqq \frac{1}{K}\sum_{k=1}^K R(\D_{n,-k})\,.
\end{equation}
It is straightforward to check that $\hat R_{\cv,n,K}$ is unbiased for $\bar R_{\cv, n,K}$ and $\mu_{n_\tr}$, in the sense that $\mathbb E(\hat R_{\cv,n,K}) = \mathbb E \bar R_{\cv, n,K}=\mu_{n_\tr}$.  
The following picture describes the chain of approximation among the quantities of interest.
\begin{equation}
    \label{eq:cv_graph}
    \begin{array}{ccccc}
       \hat R_{\cv,n,K}  & \xrightarrow{(a)} & \bar R_{\cv, n,K} & \xrightarrow{(b)} & \mu_{n_\tr}\\
       &  &  & & \downarrow\text{\scriptsize $(c)$} \\
       & & R(\D_n) &\xrightarrow{(b')} & \mu_n
    \end{array}
\end{equation}
where the approximation in each arrow has its own interpretation.
\begin{enumerate}
    \item [(a)]: the sampling variability in the evaluation of empirical risk;
    \item [(b,b$'$)]: the concentration of the risk of estimated models around the expected value;
    \item [(c)]: the continuity of the average risk as the sample size changes.
\end{enumerate}
As we will see in this article, all three approximations can be established using appropriate stability conditions.

\subsection{Stability Conditions}\label{subsec:stability_overview}
Now we describe the specific notions of stability to be considered in this article.  Intuitively, stability refers to a form of continuity of the fitting procedure $\hat f$, such that $\hat f(\D)\approx \hat f(\D')$ whenever $\mathcal D\approx \mathcal D'$.  
Because $\hat f$ is random, stability must be stated in a stochastic manner, such as via a norm or tail-probability bound on a random variable. 

We will consider two particular types of pairs $(\D,\D')$ that are close to each other. 
In the first, $\D'$ is obtained by removing one element in $\D$. Recall that $\D_n=\{X_1,...,X_n\}$ consists of a collection of i.i.d. sample points. We can use $\D_{n-1}$ to denote the dataset obtained by removing one sample point from $\D_n$. This leads to the \emph{leave-one-out stability}.
\begin{definition}[Leave-One-Out Stability in $L_q$-Norm]\label{def:2.2-loo-stab-lq}
    Let $h:[0,1]\times\bigcup_{n=1}^\infty \X^{\otimes n}\mapsto\mathbb R^1$ be a function that maps a dataset of any size and an auxiliary input in $[0,1]$ to a real number. 
For $q\in[1,\infty]$, define the Leave-One-Out $L_q$ stability (LOO-$L_q$) of $h$ at sample size $n$ as
    $$
  \mathcal S_{q,n}^{\rm loo}(h)\coloneqq  \|h(U,\D_n)-h(U,\D_{n-1})\|_q\,,
    $$
    where $U\sim {\rm Unif}(0,1)$ is independent of $\mathcal D_n$.
\end{definition}

In this definition, we allow $h$ to involve additional randomness besides the dataset $\D_n$, through $U\sim {\rm Unif}(0,1)$. There is no loss of generality by assuming $U\sim{\rm Unif}(0,1)$ since we can use $U$ to generate random vectors of arbitrary distribution.  In particular, this definition covers the loss function $\ell(X_0,\D_n)$, which involves the random evaluation point $X_0$.  For example, in \Cref{thm:bousquet-consistency} below, we will consider
$\mathcal S_{q,n}^{\rm loo}(\ell(X_0,\cdot))$, which equals
$\|\ell(X_0,\mathcal D_n)-\ell(X_0,\mathcal D_{n-1})\|_q$, where the $\ell_q$ norm is with respect to the joint randomness of the $n+1$ independent sample points $X_0,...,X_n$.

In another type of neighboring pairs $(\D,\D')$, $\D'$ is obtained by replacing one entry of $\D$ by an i.i.d. copy.  
For each $i\in\{0,1,...,n\}$ we let $X_i'$ be an i.i.d. copy of $X_i$, independent of everything else.  For any $i\in[n]$ and function $h:\bigcup_{n=1}^\infty \X^{\otimes n}\mapsto\mathbb R^1$ which may involve additional randomness, we define the perturb-one difference operator
\begin{equation}\label{eq:2.2-nabla}
\nabla_i h(\D_n) = h(\D_n) - h(\D_n^i)
\end{equation}
where $\D_n^i$ is obtained by replacing the $i$th entry of $\D_n$ by $X_i'$.

The perturb-one stability is defined as follows.
\begin{definition}[Perturb-One Stability in $L_q$-Norm]\label{def:2.2-po-stab-lq}
    Let $h:[0,1]\times \bigcup_{n=1}^\infty \X^{\otimes n}\mapsto\mathbb R^1$ be a  function that maps a dataset of any size and an auxiliary input in $[0,1]$ to a real number. 
For $q\in[1,\infty]$, define the Perturb-One $L_q$ (PO-$L_q$) stability of $h$ at sample size $n$ as
    $$
\S_{q,n}^{\rm po}   \coloneqq \|\nabla_i h(U,\D_n)\|_q\,,
    $$
    where $U\sim {\rm Unif}(0,1)$ is independent of $\D_n$.
\end{definition}

Although from a practical point of view the difference is not substantial, mathematically the Perturb-One stability is weaker than the Leave-One-Out stability. Roughly speaking, Perturb-One Stability only requires that each individual sample point contributes a small proportion to the output, while Leave-One-Out stability also requires continuity with respect to sample size.  In particular, a simple application of triangle inequality implies that the PO-$L_q$ stability is no larger than twice the LOO-$L_q$ stability .
\begin{proposition}[LOO Stability $\Rightarrow$ PO Stability]\label{pro:2.2-loo-imply-po}
     $\S_{q,n}^{\rm po}(h)\le 2\S_{q,n}^{\rm loo}(h)$.
\end{proposition}
Conversely, PO-$L_q$ stability does not necessarily imply LOO-$L_q$ stability.  For  example, consider the following $h$ which does not use the auxiliary input: 
\begin{align*}
    h(\D_n) = \left\{\begin{array}{cc}
        \frac{1}{n}\sum_{i=1}^n X_i & \text{ if $n$ is even,} \\
        1+ \frac{1}{n}\sum_{i=1}^n X_i & \text{ if $n$ is odd.} 
    \end{array}\right.
\end{align*}
Clearly, $h$ satisfies PO-$L_q$ stability at rate $O(1/n)$ whenever $\|X_1\|_{q}<\infty$ but $h$ does not satisfy LOO-$L_q$ stability at the same rate.  




\subsection{Risk Consistency of LOOCV}\label{subsec:bousquet_consistency}
Now we present a risk consistency result for the leave-one-out cross-validation under the LOO-$L_q$ stability condition of the loss function.

\begin{theorem}\label{thm:bousquet-consistency}
Let $1\le q\le r\le \infty$ be such that $1/q+1/r=1$.
Assume the estimator $\hat f$ is symmetric, the loss function has bounded $L_r$-norm: $\|\ell(X_0,\D_n)\|_r\le B<\infty$ uniformly for all $n$, and
    $\S_{q,n}^{\rm loo}[\ell(X_0,\cdot)]\le \epsilon_n$. Then
$$\mathbb E(\hat R_{\cv,n,n}-R(\D_n))^2\le n^{-1}B^2+6 B \epsilon_n\,.$$
\end{theorem}

The proof of \Cref{thm:bousquet-consistency} is quite elementary and elegant. The key step is a ``label-swapping trick'' that makes use of the symmetry of data, which further leads to the application of the stability condition.

\begin{proof}
The proof is based on a direct expansion of the squared difference between $\hat R_{\cv,n,n}$ and $R(\D_n)$:
\begin{align}
    \mathbb E (\hat R_\cv-R(\D_n))^2 = \mathbb E R(\D_n)^2- 2\mathbb E \left[R(\D_n)\hat R_{\cv,n,n}\right] + \mathbb E \hat R_{\cv,n,n}^2\,. \label{eq:loo-risk-expansion}
\end{align}
Next, we expand each of the three terms.

For the first term, we have
    \begin{align}
    \mathbb E R(\D_n)^2=\mathbb E_{\D_n}\left[\mathbb E_{X_0}\ell(X_0,\D_n)\right]^2 = \mathbb E \left[\ell(X_0,\D_n)\ell(X_0',\D_n)\right]\,.\label{eq:loo-risk-expansion1}
\end{align}

For the second term, we have, by definition of $\hat R_{\cv,n,n}$, symmetry of $\hat f$ and exchangeability of sample points,
\begin{align}
    \mathbb E \left[R(\D_n) \hat R_{\cv,n,n}\right] = & \mathbb E \left\{\mathbb E_X \ell(X,\D_n)\left[n^{-1}\sum_{i=1}^n\ell(X_i,\D_{n,-i})\right]\right\}\nonumber\\
    = &  \mathbb E \left\{\ell(X_0,\D_n)\ell(X_1, \D_{n,-1})\right\}\,.\label{eq:loo-risk-expansion2}
\end{align}

For the third term, note that the assumptions $1\le q\le r\le\infty$ and $1/q+1/r=1$ imply that $r\ge 2$, so
\begin{align}
    \mathbb E \hat R_{\cv,n,n}^2 = & n^{-2}\sum_{i=1}^n\mathbb E\ell^2(X_i,\D_{n,-i})+n^{-2}\sum_{1\le i\neq j\le n}\mathbb E \ell(X_i,\D_{n,-i})\ell(X_j,\D_{n,-j})\nonumber\\
    \le & n^{-1}B^2+(1-n^{-1})\mathbb E\ell(X_1,\D_{n,-1})\ell(X_2,\D_{n,-2})\nonumber\\
    \le & \mathbb E[\ell(X_1,\D_{n,-1})\ell(X_2,\D_{n,-2})]+n^{-1}B^2\,,\label{eq:loo-risk-expansion3}
\end{align}
where the first inequality uses the fact that $\|\ell(X_i,\D_{n,-i})\|_2^2\le \|\ell(X_i,\D_{n,-i})\|_r^2\le B^2$, and symmetry of $\hat f$ and the data points.

Putting \eqref{eq:loo-risk-expansion1}, \eqref{eq:loo-risk-expansion2}, and \eqref{eq:loo-risk-expansion3} in \eqref{eq:loo-risk-expansion}, we obtain
\begin{align}
    &\mathbb E (\hat R_{\cv,n,n}-R(\D_n))^2-n^{-1}B^2\nonumber\\ 
    = &\mathbb E R(\D_n)^2- 2\mathbb E R(\D_n)\hat R_{\cv,n,n} + \mathbb E \hat R_{\cv,n,n}^2-n^{-1}B^2\nonumber\\
    \le & \mathbb E\ell(X_0,\D_n)\ell(X_0',\D_n)-2\mathbb E \ell(X_0,\D_n)\ell(X_1, \D_{n,-1})+\mathbb E\ell(X_1,\D_{n,-1})\ell(X_2,\D_{n,-2})\nonumber\\
    = & \mathbb E\left[\ell(X_0,\D_n)\ell(X_0',\D_n)-\ell(X_0,\D_n)\ell(X_1,\D_{n,-1})\right]\nonumber\\
     &~~+ \mathbb E \left[\ell(X_1,\D_{n,-1})\ell(X_2,\D_{n,-2})-\ell(X_1,\D_{n,-1})\ell(X_0,\D_n)\right]\,.\label{eq:roo-risk-exp-simple}
\end{align}

Recall that $\D_n^i$ is the perturb-one dataset obtained by replacing the $i$th entry of $\D_n$ by an i.i.d. copy $X_i$. We have the following key equation (the ``label-swapping trick'').
\begin{equation}\label{eq:2.3-label-swapping}\mathbb E\ell(X_0,\D_n)\ell(X_0',\D_n)=\mathbb E\ell(X_0,\D_{n}^1)\ell(X_0',\D_{n}^1)=\mathbb E \ell(X_0,\D_{n}^1)\ell(X_1,\D_{n}^1)\,,\end{equation}
where the first equality holds due to exchangeability between $X_1$ and $X_1'$, and the second equality holds due to exchangeability between $X_0'$ and $X_1$.

As a result
\begin{align*}
   & \mathbb E\left[\ell(X_0,\D_n)\ell(X_0',\D_n)-\ell(X_0,\D_n)\ell(X_1,\D_{n,-1})\right]\\
   = & \mathbb E\left[\ell(X_0,\D_n^1)\ell(X_1,\D_n^1)-\ell(X_0,\D_n)\ell(X_1,\D_{n,-1})\right]\\
   = &\mathbb E\big[\ell(X_0,\D_n^1)\ell(X_1,\D_n^1)-\ell(X_0,\D_n)\ell(X_1,\D_n^1)\\
   &~~+\ell(X_0,\D_n)\ell(X_1,\D_n^1)-\ell(X_0,\D_n)\ell(X_1,\D_{n,-1})\big]\\
   = & \mathbb E\left[\ell(X_0,\D_n^1)-\ell(X_0,\D_n)\right]\ell(X_1,\D_n^1)+\mathbb E\ell(X_0,\D_n)\left[\ell(X_1,\D_n^1)-\ell(X_1,\D_{n,-1})\right]\,.
\end{align*}
By the LOO-$L_q$ stability of $\ell(X_0,\D_n)$ we have
$$\|\ell(X_1,\D_n^1)-\ell(X_1,\D_{n,-1})\|_q\le \epsilon_n\,$$
and by triangle inequality
$$
\|\ell(X_0,\D_n^1)-\ell(X_0,\D_n)\|_q\le 2\epsilon_n\,.
$$
Thus H\"{o}lder's inequality implies that
$$\left|\mathbb E\left[\ell(X_0,\D_n)\ell(X_0',\D_n)-\ell(X_0,\D_n)\ell(X_1,\D_{n,-1})\right]\right|\le 3 B\epsilon_n\,.$$

A similar treatment of swapping $X_2$, $X_2'$, and $X_0$ in the term 
$$\mathbb E \left[\ell(X_1,\D_{n,-1})\ell(X_2,\D_{n,-2})-\ell(X_1,\D_{n,-1})\ell(X_0,\D_n)\right]$$
leads to the same upper bound of $3B\epsilon_n$. Therefore the RHS of \eqref{eq:roo-risk-exp-simple} is bounded by $6B\epsilon_n$, which leads to the claimed result.
\end{proof}

The ``label-swapping trick'' \eqref{eq:2.3-label-swapping} is a special case of a more general technique called the ``free nabla trick'' (\Cref{lem:4.4-free-nabla}).  Besides its crucial role in proving CLTs for CV risk estimates, the free nabla trick also allows us to extend \Cref{thm:bousquet-consistency} by providing upper bounds on $|\hat R_{\cv,n,K}-\bar R_{\cv,n,K}|$ for general values of $K$ under a Perturb-One stability condition.  An application of this upper bound is presented in \Cref{subsec:5.5-cross-conf}, in the context of conformal prediction.

\subsection{Concentration of CV Risk Estimate}\label{subsec:bousquet_concentration}
\Cref{thm:bousquet-consistency} gives a general consistency result for the LOO-CV risk estimates. In this subsection, we study the concentration bounds for such estimates.  The main challenge is that many mean-zero terms that disappear under expectation or are combined in the proof of \Cref{thm:bousquet-consistency}, such as those in \eqref{eq:loo-risk-expansion1} and \eqref{eq:loo-risk-expansion2}, have complicated correlations that contribute to the random fluctuation of $\hat R_{\cv,n,n}$ around $R(\D_n)$.  Therefore, we will need stronger stability conditions to establish exponential concentration inequalities for CV risk estimates.

The starting point is the following well-known bounded-difference concentration result.
\begin{theorem}[McDiarmid's Inequality]\label{thm:mcdiarmid}
For any function $h$ that maps $\D_n$ to a real number, if 
there exist constants $c_1,...,c_n$ such that $|\nabla_i h|\le c_i$ almost surely for each $i\in[n]$, then
$$\mathbb P(h-\mathbb E h\ge t)\le \exp\left(-\frac{2t^2}{\sum_{i=1}^n c_i^2}\right)\,,$$
for all $t>0$.
\end{theorem}

The bounded difference condition in McDiarmid's inequality corresponds to the PO-$L_\infty$ bound.
Therefore, we consider the following perturb-one stability in the $L_\infty$-norm condition on the loss function of the fitted model:
\begin{equation}\label{eq:PO_infty_bound}
\S_{\infty,n-1}^{\rm po}(\ell(X_0,\hat f(\cdot))) \le \epsilon_n\,.    
\end{equation}

Again, in this subsection, we focus on the LOO-CV, where $K=n$ and $\D_{n,-i}$ denotes the dataset obtained by removing the $i$th entry.

Under condition \eqref{eq:PO_infty_bound}, for any $1\le i\neq j\le n$
\begin{align*}
    |\nabla_i\ell(X_j,\D_{n,-j})| \le \epsilon_n\,. 
\end{align*}
Therefore,
\begin{align*}
\nabla_i \hat R_{\cv,n,n} = n^{-1}\left[\ell(X_i,\D_{n,-i})-\ell(X_i',\D_{n,-i})\right]
+n^{-1}\sum_{j\neq i} \nabla_i \ell(X_j,\D_{n,-j})\,.
\end{align*}
If $\|\ell(X_i,\mathcal D_{n,-i})\|_\infty=B<\infty$, then
the triangle inequality implies that $\hat R_{\cv,n,n}$ satisfies the bounded difference assumption in McDiarmid's inequality with $c_i=2Bn^{-1}+\epsilon_n$. Therefore, we have the following result.
\begin{proposition}\label{pro:bousquest_concentration}
Assume $\|\ell(X_0,\D_{n-1})\|_\infty=B<\infty$.
\begin{enumerate}
    \item If the loss function $\ell(\cdot,\cdot)$ and fitting procedure $\hat f$  satisfy \eqref{eq:PO_infty_bound}, we have, for any $t>0$,
    $$
    \mathbb P\left(|\hat R_{\cv,n,n}-\mu_{n-1}|\ge t\right)\le 2\exp\left(-\frac{2t^2}{n\left(\frac{2B}{n}+\epsilon_n\right)^2}\right)\,.
    $$
    \item If the loss function $\ell(\cdot,\cdot)$ and fitting procedure $\hat f$  satisfy $\S_{\infty,n}^{\rm loo}(\ell(X_0,\hat f(\cdot)))\le \epsilon_n$, then we have, for any $t>0$,
    $$
    \mathbb P\left(|\hat R_{\cv,n,n}-\mu_{n-1}|\ge t\right)\le 2\exp\left(-\frac{2t^2}{n\left(\frac{2B}{n}+2\epsilon_n\right)^2}\right)\,,
    $$
    and
    $$
    \mathbb P\left(|\hat R_{\cv,n,n}-\mu_{n}|\ge t+\epsilon_n\right)\le 2\exp\left(-\frac{2t^2}{n\left(\frac{2B}{n}+2\epsilon_n\right)^2}\right)\,.
    $$
\end{enumerate}
\end{proposition}

In other words, $\hat R_{\cv,n,n}$ concentrates around $\mu_n$ at an $O(n^{-1/2}+n^{1/2}\epsilon_n)$ scale and a sub-Gaussian tail.  This result can be extended to general $K$-fold cross-validation to bound $|\hat R_{\cv,n,K}-\mu_{n_\tr}|$.

\begin{remark}\label{rem:2.4-epsilon-o(sqrt(n))}
The concentration bounds based on McDiarmid's inequality presented in \Cref{pro:bousquest_concentration} have a much stronger requirement on the stability bound $\epsilon_n$ than in the consistency result \Cref{thm:bousquet-consistency}.
In the consistency result, \Cref{thm:bousquet-consistency} implies that $\hat R_{\cv,n,n}-R(\D_n)=o_P(1)$ as long as $\epsilon_n=o(1)$.
In contrast, the concentration results in \Cref{pro:bousquest_concentration} require
$\epsilon_n=o(n^{-1/2})$ in order to have a vanishing bound.  Such a difference reflects the subtle nature of the dependence between the cross-validated losses, which is a major challenge in uncertainty quantification for CV model evaluation and comparison.
\end{remark} 

\subsection{Extensions to Sub-Weibull Stability}\label{subsec:concentration_ineq}
The concentration result \Cref{pro:bousquest_concentration} requires the perturb-one 
loss stability to be uniformly bounded, which may be too restrictive, as it precludes 
many common scenarios involving unbounded support, such as linear regression with 
Gaussian noise under the squared loss.
Unlike the consistency result, which can avoid the uniform bound condition by leveraging 
different combinations of conjugate indices in H\"{o}lder's inequality, relaxing the 
$L_\infty$-norm stability assumption in the concentration result is less straightforward, 
as it typically requires extending the bounded-difference inequality to the case of 
infinite support.
Here, we present an approach based on sub-Weibull tail conditions. Some results and 
properties of sub-Weibull random variables may be of more general interest.

\begin{definition}[Sub-Weibull Random Variables] For a pair of fixed positive real numbers $(\kappa,\alpha)$, we say a random variable $Z$ is $(\kappa,\alpha)$-sub-Weibull (or $(\kappa,\alpha)$-SW) if
\begin{equation}\label{eq:2.5-sw-def-prob}
\mathbb P\left(\frac{|Z|}{\kappa}\ge t\right)\le 2\exp(-t^{1/\alpha})\,,~~\forall~t>0\,.
\end{equation}
\end{definition}
The definition of sub-Weibull random variables is concerned with the tail behavior. It involves two parameters: a scaling parameter $\kappa$ and an exponent parameter $\alpha$.  A larger $\alpha$ corresponds to a heavier tail. In particular, a random variable $Z$ is $(\kappa,\alpha)$-SW if and only if $(Z/\kappa)^{1/\alpha}$ has a sub-exponential tail.   The sub-Gaussian tail condition with standard deviation $\kappa$ corresponds to $(\kappa,1/2)$-SW, and the sub-exponential condition corresponds to $(\kappa,1)$-SW.   It may also be helpful to interpret the extreme case of $\alpha=0$ as uniformly bounded.

Like sub-Gaussian and sub-exponential tail conditions, we can bound the $L_q$-norms of a sub-Weibull random variable proportionally to the scale parameter.
\begin{proposition}[Sub-Weibull Condition in $L_q$-Norm]\label{pro:2.5-sw-equiv}
    If a random variable $Z$ satisfies \eqref{eq:2.5-sw-def-prob} then
    \begin{equation}\label{eq:2.5-sw-def-lq}
        \|Z\|_q\le \kappa' q^\alpha\,,~~\forall q\ge 1\,,
    \end{equation}
    with a constant $\kappa'\le 6 e^\alpha(1+\alpha)^\alpha \kappa$.  Conversely, if a random variable $Z$ satisfies \eqref{eq:2.5-sw-def-lq} for some $(\kappa',\alpha)$ then it satisfies
    \eqref{eq:2.5-sw-def-prob} with $\kappa \le c_\alpha \kappa'$ for a positive constant $c_\alpha$ depending only on $\alpha$.
\end{proposition}
In most cases, we will only focus on the rates of the scale parameters of sub-Weibull residual terms and usually do not distinguish two scaling parameters differing by a constant factor. Therefore, we will use the two sub-Weibull conditions \eqref{eq:2.5-sw-def-prob} and \eqref{eq:2.5-sw-def-lq} interchangeably.  

\begin{proof}[Proof of \Cref{pro:2.5-sw-equiv}]
For $q\ge 1$,
\begin{align*}
\|Z\|_q^q = & \mathbb E |Z|^q =\int_0^\infty \mathbb P(|Z|^q>t)dt = \int_0^\infty \mathbb P(|Z|/\kappa>t^{1/q}/\kappa)dt\\
\le & \int_0^\infty 2\exp\left(-t^{1/(\alpha q)}\kappa^{-1/\alpha}\right)dt\\
(t^{1/(\alpha q)}\kappa^{-1/\alpha}=u)~~~\le & 2\kappa^q(\alpha q)\int_0^\infty e^{-u} u^{\alpha q-1}du\\
=& 2\kappa^q(\alpha q)\Gamma(\alpha q)\\
=&2\kappa^q\Gamma(1+\alpha q)\\
\le & 6\kappa^q (1+\alpha q)^{1+\alpha q}\,,
\end{align*}
where $\Gamma(z)=\int_0^\infty t^{z-1}e^{-t}dt$ is the gamma function.
So for $q\ge 1$ we have
\begin{align*}
    \|Z\|_q \le 6 \kappa(1+\alpha q)^{1/q}(1+\alpha q)^\alpha \le 6 e^\alpha (1+\alpha)^\alpha \kappa q^\alpha\,,
\end{align*}
because $\sup_{x>0}(1+x)^{1/x}=e$.

Now we prove the converse. If $Z$ satisfies \eqref{eq:2.5-sw-def-lq} with constants $(\kappa,\alpha)$, then for $k\ge \alpha$
\begin{align*}
    \mathbb E |Z|^{k/\alpha} =\|Z\|_{k/\alpha}^{k/\alpha}\le \left( \kappa(k/\alpha)^\alpha\right)^{k/\alpha}= \kappa^{k/\alpha}(k/\alpha)^k\,.
\end{align*}
If $1\le k<\alpha$, then by Jensen's inequality
\begin{align*}
    \mathbb E |Z|^{k/\alpha} \le \left(\mathbb E |Z|\right)^{k/\alpha}= \|Z\|_1^{k/\alpha}\le \kappa^{k/\alpha}=\kappa^{k/\alpha}(k/\alpha)^k (\alpha/k)^k\le \alpha^\alpha\kappa^{k/\alpha}(k/\alpha)^k\,.
\end{align*}
So in either case we have
$$\mathbb E |Z|^{k/\alpha}\le c_\alpha\kappa^{k/\alpha}(k/\alpha)^k$$
with $c_\alpha = 1\vee \alpha^\alpha$.

Now using $k!\ge (k/e)^k$ for all $k\ge 1$ we have
\begin{align*}
    \mathbb E e^{(\lambda|Z|)^{1/\alpha}} = & 1+\sum_{k=1}^\infty \frac{\lambda^{k/\alpha} \mathbb E|Z|^{k/\alpha}}{k!}\\
    \le & 1+c_\alpha\sum_{k=1}^\infty\frac{\lambda^{k/\alpha} \kappa^{k/\alpha} (k/\alpha)^k}{(k/e)^k}\\
     \le & 1+\frac{c_\alpha e(\lambda \kappa)^{1/\alpha}/\alpha}{1-e(\lambda \kappa)^{1/\alpha}/\alpha}
     \le 2\,,
\end{align*}
with the choice $\lambda =\kappa^{-1}\left[\frac{\alpha}{e(1+c_\alpha)}\right]^\alpha$.

Let $\kappa'=1/\lambda=\kappa \left[\frac{e(1+c_\alpha)}{\alpha}\right]^\alpha$.
We have by Markov's inequality applied to $e^{(\lambda |Z|)^{1/\alpha}}$ for any $t>0$
\begin{align*}
    &\mathbb P\left(\frac{|Z|}{\kappa'}\ge t\right)=\mathbb P\left(\lambda |Z|\ge t\right)
    =\mathbb P\left((\lambda |Z|)^{1/\alpha}\ge t^{1/\alpha}\right)
    \le 2\exp(-t^{1/\alpha})\,. \qedhere
\end{align*}
\end{proof}

The sum and product of sub-Weibull random variables are also sub-Weibull, with the scale parameter preserved through the sum and product, and a slightly different exponent parameter.
\begin{proposition}[Basic Properties of SW Random Variables]\label{pro:2.5-sum-prod-sw}
    If $Z_i$ is $(\kappa_i,\alpha_i)$-SW, for $i=1,2$, then
    \begin{enumerate}
        \item $Z_1Z_2$ is $(c_{\alpha_1,\alpha_2}\kappa_1\kappa_2,\alpha_1+\alpha_2)$-SW, where $c_{\alpha_1,\alpha_2}$ is a constant depending on $(\alpha_1,\alpha_2)$ only.
        \item $Z_1+Z_2$ is $(\kappa_1+\kappa_2, \alpha_1\vee\alpha_2)$-SW. 
        \item $\mathbb E(Z_1|\mathcal C)$ is $(\kappa_1,\alpha_1)$-SW for any sub-$\sigma$-field $\mathcal C$. 
        \item For $\beta>0$, $|Z_1|^\beta$ is $[(1\vee\beta^{\alpha_1\beta})\kappa_1^\beta,\alpha_1\beta]$-SW.
    \end{enumerate}
\end{proposition}

\begin{proof}
    For the first part, take $r,s>1$ such that $1/r+1/s=1$
    \begin{align*}
        \|Z_1 Z_2\|_q =& \left[\mathbb E |Z_1|^q |Z_2|^q\right]^{1/q}\\
        \le& \left[\left(\mathbb E |Z_1|^{qr}\right)^{1/r} \left(\mathbb E|Z_2|^{qs}\right)^{1/s}\right]^{1/q}\\
        =& \left[\|Z_1\|_{qr}^q\|Z_2\|_{qs}^q\right]^{1/q}\\
        =& \|Z_1\|_{qr}\|Z_2\|_{qs}\\
        \le & \kappa_1 (qr)^{\alpha_1}\kappa_2(qs)^{\alpha_2}\\
         = &\kappa_1\kappa_2 q^{\alpha_1+\alpha_2} r^{\alpha_1}s^{\alpha_2}\\
         \le & \frac{(\alpha_1+\alpha_2)^{\alpha_1+\alpha_2}}{\alpha_1^{\alpha_1}\alpha_2^{\alpha_2}}\kappa_1\kappa_2 q^{\alpha_1+\alpha_2}
    \end{align*}
    where the last inequality follows from minimizing $r^{\alpha_1}s^{\alpha_2}$ over all $(r,s)$ such that $1/r+1/s=1$, which is achieved at $r=1+\alpha_2/\alpha_1$, $s=1+\alpha_1/\alpha_2$.

    For the second part,
    \begin{align*}
        \|Z_1+Z_2\|_q\le \|Z_1\|_q+\|Z_2\|_q\le & \kappa_1 q^{\alpha_1}+\kappa_2 q^{\alpha_2}\le (\kappa_1+\kappa_2)q^{\alpha_1\vee\alpha_2}\,.
    \end{align*}

    The third part is a direct consequence of Jensen's inequality.

    To prove the fourth part, without loss of generality we assume $Z_1\ge 0$ almost surely. Let $q\ge 1$ be a positive number. Consider two cases.  First, if $1\le q < 1/\beta$, then
    $$\| Z_1^\beta\|_q\le \|Z_1^\beta\|_{\frac{1}{\beta}} = \|Z_1\|_1^\beta\le \kappa_1^\beta\le \kappa_1^\beta q^{\alpha_1\beta}\,.$$
    Second, if $q\ge 1/\beta$ then
    $$
    \|Z_1^\beta\|_q = \|Z_1\|_{\beta q}^\beta\le \kappa_1^\beta (\beta q)^{\alpha_1\beta}\,.
    $$
    So in either case we have $\|Z_1^\beta\|_q\le (1\vee\beta^{\alpha_1\beta} )\kappa_1^\beta q^{\alpha_1\beta}$.
\end{proof}

The following result will be useful in providing tail probability bounds for a collection of sub-Weibull random variables.
\begin{lemma}
    [Maximal Inequality]
    \label{lem:2.5-sw-maximal}
    Let $Z_1,...,Z_m$ be $(\kappa,\alpha)$-SW random variables.
    Then there exists a constant $c_\alpha$ depending only on $\alpha$, such that $\max_{1\le i\le m} |Z_i|$ is  $(c_\alpha (\log m)^\alpha \kappa,\alpha)$-SW.

    More generally, if each $Z_i$ is $(\kappa_i,\alpha_i)$-SW, then the result holds with $\kappa=\max_i \kappa_i$ and $\alpha=\max_i\alpha_i$.
\end{lemma}
\begin{proof}
    [Proof of \Cref{lem:2.5-sw-maximal}]
    Let $\tilde Z_i=|Z_i/\kappa|^{1/\alpha}$. Then $\tilde Z_i$ is $(1/\alpha,1)$-SW by Part (4) of \Cref{pro:2.5-sum-prod-sw}. Hence $\mathbb E e^{\tilde Z_i/c_\alpha}-1=1$ for a constant $c_\alpha>0$ depending only on $\alpha$.  Lemma 2.2.2 of \cite{van1996weak} implies that, for another $c_\alpha'$, $\max_i \tilde Z_i / (c_\alpha'\log m)$ is sub-exponential, i.e., $\max_i \tilde Z_i$ is $(c_\alpha''\log m,1)$-SW for another constant $c_\alpha''$. Therefore, the claimed result follows from
    Part (4) of \Cref{pro:2.5-sum-prod-sw}.
\end{proof}

Now we return to the bounded-difference concentration inequality. For a function $h$ acting on a dataset $\D_n$ consisting of i.i.d. sample points, we want to generalize McDiarmid's inequality by establishing concentration of $h(\D_n)$ around its mean value under sub-Weibull tail conditions of the perturb-one differences $\nabla_i h(\D_n)$.

We begin with the standard martingale decomposition of $h(\D_n)$. 
With the convention $\D_0=\emptyset$, we have
\begin{align}
h(\D_n)- \mathbb E h(\D_n)  = & \sum_{i=1}^n \mathbb E\left[h(\D_n) |\D_{i}\right]-\mathbb E\left[h(\D_n)|\D_{i-1}\right]\nonumber\\
 =&\sum_{i=1}^n \mathbb E\left[\nabla_i h(\D_n)|\D_i\right]\,.\label{eq:2.5-nabla-telescope}
\end{align}

Similar to the proof of McDiarmid's inequality, the summand $\{\mathbb E(\nabla_i h(\D_n)|\D_i):1\le i\le n\}$ is a martingale increment sequence adapted to the filtration of $\sigma$-fields generated by $\{\D_i:1\le i\le n\}$.  On the other hand, Part (3) of \Cref{pro:2.5-sum-prod-sw} implies that if all the differences $\nabla_i h(\D_n)$ are sub-Weibull then the summand $\mathbb E[\nabla_i h(\D_n)|\D_i]$ are also sub-Weibull. However, a direct application of Part (2) of \Cref{pro:2.5-sum-prod-sw} may not yield a useful upper bound, since the resulting scaling parameter may involve $n$ in an undesirable way.

The following theorem offers a good upper bound for sums of martingale increments with sub-Weibull tails.
\begin{theorem}[Rio's Inequality]\label{thm:rio}
    Let $q\ge 2$ and $(S_n)_{n\ge 0}$ be a sequence of random variables with finite $L_q$-norm.
    Let $M_k=S_k-S_{k-1}$. If $\mathbb E(M_k|S_{k-1})=0$ a.s. for all $k$ then
    $$
    \|S_n\|_q^2\le \|S_0\|_q^2+(q-1)\sum_{i=1}^n\|M_i\|_q^2\,.
    $$
\end{theorem}
We refer readers to \cite{rio2009moment} for the proof of \Cref{thm:rio}. Using Rio's inequality we can prove a bounded difference inequality under sub-Weibull conditions.
\begin{theorem}[Sub-Weibull Difference Concentration]\label{thm:2.5-sw-mcdiarmid}
Let $h$ be a real-valued function acting on $\X^{\otimes n}$ and $\D_n$ a set of $n$ independent sample points in $\X$.
    If $\nabla_i h(\D_n)$ is $(\kappa_i,\alpha_i)$-SW in the sense of \eqref{eq:2.5-sw-def-lq}, then 
    $\mathbb E[h(\D_n)]-h(\D_n)$ is $(\kappa,\alpha)$-SW with
    $$\kappa=2^{\bar\alpha+1/2}\left(\sum_{i=1}^n\kappa_i^2\right)^{1/2}\,,~~~\alpha = \bar\alpha+1/2\,,$$
    where $\bar\alpha=\max_{1\le i\le n}\alpha_i$.
\end{theorem}
\begin{proof}[Proof of \Cref{thm:2.5-sw-mcdiarmid}]
    Let $S_0=0$, $S_i=\sum_{j=1}^i M_j$ with $M_i=\mathbb E[ h(\D_n) |\D_{i-1}]-\mathbb E[h(\D_n)|\D_i]=\mathbb E[\nabla_i h(\D_n)|\D_i]$.
    Then $S_n=\mathbb E[h(\D_n)]-h(\D_n)$.

    Using Rio's inequality and the sub-Weibull condition of $\nabla_i h(\D_n)$ we have, for $q>2$,
    \begin{align*}
        \|\mathbb E[h(\D_n)]-h(\D_n)\|_q\le & (q-1)^{1/2}\left(\sum_{i=1}^n \kappa_i^2 q^{2\alpha_i}\right)^{1/2}\\
        \le & q^{\bar\alpha+1/2} \left(\sum_{i=1}^n \kappa_i^2\right)^{1/2}\,.
    \end{align*}
    When $1\le q\le 2$, we have
    \begin{align*}
        \|\mathbb E[h(\D_n)]-h(\D_n)\|_q\le& \|\mathbb E[h(\D_n)]-h(\D_n)\|_2\le 2^{\bar\alpha+1/2} \left(\sum_{i=1}^n \kappa_i^2\right)^{1/2}\\
        =&q^{\bar\alpha+1/2}(2/q)^{\bar\alpha+1/2} \left(\sum_{i=1}^n \kappa_i^2\right)^{1/2}\,.
    \end{align*}
    The claim holds since $(2/q)\le 2$.
\end{proof}

This result shows that $h(\D_n)$ deviates from the mean value on the correct order of $\sqrt{\kappa_1^2+...+\kappa_n^2}$.   The original McDiarmid's inequality can be stated as a special case of \Cref{thm:2.5-sw-mcdiarmid} with $\alpha_i=0$ and $c_i=\kappa_i$.

Now we have a concentration inequality for cross-validation risk estimates under a relaxed version of perturb-one stability.
\begin{definition}[Perturb-One Sub-Weibull Stability]\label{def:2.5-po-sw-stab}
       Let $h:[0,1]\times \bigcup_{n=1}^\infty \X^{\otimes n}\mapsto\mathbb R^1$ be a  function that maps a dataset of any size and an auxiliary input in $[0,1]$ to a real number.  
We say $h$ satisfies the Perturb-One Sub-Weibull (PO-SW) stability with parameter $(\kappa_n,\alpha)$ at sample size $n$ if $\nabla_i h(U,\D_n)$ is $(\kappa_n,\alpha)$-SW in the sense of \eqref{eq:2.5-sw-def-lq} for $U\sim {\rm Unif}(0,1)$ independent of $\D_n$.
\end{definition}

\begin{theorem}[Concentration of LOO-CV Risk Under Sub-Weibull Stability]\label{thm:2.5-loocv-concentration-sw}
Assuming i.i.d. data, 
if $\ell(X_0,\cdot)$ satisfies PO-SW stability with parameter $(\epsilon_n,\alpha)$ at sample size $n-1$, and $\ell(X_0,\D_{n-1})$ is $(B,\alpha)$-SW, then  
$\hat R_{\cv,n,n}-\mu_{n-1}$ is $(2^{\alpha+1/2}(2B/\sqrt{n}+\sqrt{n}\epsilon_n),~\alpha+1/2)$-SW.
\end{theorem}
\begin{proof}[Proof of \Cref{thm:2.5-loocv-concentration-sw}]
    By definition
    $$\nabla_i \hat R_{\cv,n,n} = n^{-1}\left[\ell(X_i',\D_{n,-i})-\ell(X_i,\D_{n,-i})\right]+\frac{1}{n}\sum_{j\neq i}\nabla_i\ell(X_j,\D_{n,-j})\,.$$
    Using the triangle inequality and the sub-Weibull moment bound we have for $q\ge 1$
    $$\|\nabla_i\hat R_{\cv,n,n}\|_q\le (2/n)B q^\alpha+(1-n^{-1})\epsilon_n q^\alpha\le (2B/n+\epsilon_n)q^\alpha\,,$$
    which implies that $\nabla_i \hat R_{\cv,n,n}$ is $(2B/n+\epsilon_n,~\alpha)$-SW.
    Then the claimed result follows from \Cref{thm:2.5-sw-mcdiarmid}.
\end{proof}
\Cref{thm:2.5-loocv-concentration-sw} shows that the LOO-CV risk estimate concentrates around the mean value at essentially the same rate as in \Cref{pro:bousquest_concentration}, with the same dependence, up to constant factors, on $(B,\epsilon_n)$. We will revisit sub-Weibull stability in \Cref{subsec:4.3-hd-clt} when proving high-dimensional Gaussian comparison for cross-validated risks.

Similar to \Cref{thm:bousquet-consistency} and \Cref{pro:bousquest_concentration}, the concentration result under sub-Weibull stability in \Cref{thm:2.5-loocv-concentration-sw} can be extended to bound $|\hat R_{\cv,n,K}-\mu_{n_\tr}|$ in general $K$-fold CV.

\subsection{Bibliographic Notes}
The point estimation target of cross-validation is studied in \cite{bates2021cross}. The random target $\bar R_{\cv,n,K}$ also appeared as the ``random centering'' in central limit theorems of cross-validation \citep{austern2020,bayle2020cross}, and was used as an intermediate inference target for the construction of model confidence sets using cross-validation \citep{lei2020cross,kissel2022high}.

The main consistency and concentration results, \Cref{thm:bousquet-consistency} and \Cref{pro:bousquest_concentration}, are based on \cite{bousquet2002stability}.  Other important contributions to the study of stability and generalization error include \cite{kutin2002almost}, who developed similar generalization error bounds in parallel to \cite{bousquet2002stability}, and \cite{shalev2010learnability,wang2016learning}, who established the connection between stability and generalization error in a broader context of empirical risk minimization.  More recently, \cite{bousquet2020sharper} obtained sharper generalization error bounds under uniform stability.

The stability conditions introduced in \Cref{def:2.2-loo-stab-lq,def:2.2-po-stab-lq,def:2.5-po-sw-stab} are most directly relevant to the theoretical study of cross-validation.  Alternative stability conditions have been considered in the study of other predictive inference tools. For example, \cite{barber2021predictive} and \cite{liang2023algorithmic} used a notion of ``out-of-sample stability'' and a high-probability version of leave-one-out stability to analyze several variants of conformal prediction methods.

More discussion and details about sub-Weibull random variables can be found in \cite{vladimirova2020sub,kuchibhotla2018moving}.  The combination of $L_q$-norm upper bound for martingale increments (\Cref{thm:rio}) and sub-Weibull difference (\Cref{thm:2.5-sw-mcdiarmid}) originated from personal communication with Arun Kumar Kuchibhotla, and a similar form has appeared in Appendix~B of \cite{kuchibhotla2023uniform}. Alternative approaches to proving extensions of McDiarmid's inequality using high-probability bounded differences can be found in \cite{kutin2002extensions,combes2024extension}.

In this article, we consider generic cross-validation methods applied to an arbitrary learning algorithm $\hat f$ under stability conditions.  Parameter estimation and risk consistency results for cross-validation have been established for specific estimators under more specific conditions, such as the lasso \citep{homrighausen2017risk,chetverikov2021cross}, ridge regression \citep{patil2021uniform}, and gradient-descent least squares regression \citep{patil2024failures}.

\newpage


\section{Regression Model Selection Under Squared Loss}\label{sec:3-model-selection}
In this section, we consider another key application of cross-validation: model and tuning parameter selection. The overarching goal of inference in this context is to choose, from a finite collection of candidate models, one that best fits the data.

We adopt the perspective of \emph{model selection for estimation} (see \Cref{subsec:1.2-estimation-selection} and references therein), in which we neither assume that the true model is included in the candidate set nor require the existence of a true model at all. Instead, each candidate model is viewed as a specific procedure for producing a function that predicts future data points, and a good model is one that yields accurate predictions.  In other words, in our context of model selection, each ``model'' simply corresponds to a particular estimator of an underlying parameter of the data-generating distribution, which further leads to a procedure for predicting future, unseen data.

This framework imposes minimal assumptions on both the data-generating process and the model-fitting procedures, capturing the flexibility and broad applicability of cross-validation. All that is required is an i.i.d.\ sample, a loss function, and a collection of fitting procedures. Notably, this setup also encompasses tuning parameter selection, and for simplicity, we will refer to both tasks under the umbrella of ``model selection.''

Using the notation and terminology in \Cref{subsec:prelim}, assume we have a collection of symmetric estimators $\{\hat f^{(r)}:1\le r\le m\}$, where for each $r\in\mathcal M=\{1,...,m\}$, $\hat f^{(r)}:\bigcup_{n=1}^\infty \X^{\otimes n}\mapsto \mathcal F$ maps a dataset $\D_n$ to a parameter $\hat f^{(r)}(\D_n)$. We would like to pick the $r$ that gives the best predictive risk $R[\hat f^{(r)}(\D_n)]$ or $\mu^{(r)}_n=\mathbb E\{R[\hat f^{(r)}(\D_n)]\}$.  

 For each $r\in\mathcal M$, we use the $K$-fold CV risk estimate $\hat R^{(r)}_{\cv,n,K}$ defined in \eqref{eq:2.1-cv-risk-est} with $\ell^{(r)}(x,\D)\coloneqq \ell(x,\hat f^{(r)}(\D))$ to approximate its out-of-sample prediction risk. The $K$-fold CV model selection output is
 $$
 \hat r_{\cv, n,K} =\arg\min_{r\in\mathcal M} \hat R^{(r)}_{\cv,n,K}\,.
 $$
The model selection consistency question we want to address is: When does $\hat r_{\cv,n,K}$ coincide with the ``best candidate model''?

\subsection{Pairwise Selection Consistency}\label{subsec:3.1-yang07}
Intuitively, for the cross-validation model selection procedure described above to successfully identify the ``best model,'' two conditions must be met. First, the best model should exhibit superior predictive accuracy at the population level—a quality gap must exist. Second, the cross-validation risk estimates must concentrate around their corresponding population risks, ensuring that the variability in the estimates does not obscure the true differences between models.

In this section, we follow the setting in \cite{yang2007consistency}, where two further simplifications are made.  The first simplification is to consider only two models: $\mathcal M=\{1,2\}$. This pairwise setting facilitates the articulation of some key intuitions, and the results can be easily extended to any finite candidate set such that $m=|\mathcal M|=O(1)$. The second is a regression setting with the squared loss:
\begin{itemize}
    \item Each sample point consists of a pair $X=(Y,Z)$ such that 
    \begin{equation}\label{eq:3.1-reg-model}
    Y=f^*(Z)+\epsilon\,,\end{equation} 
    where the noise $\epsilon$ has a finite second moment and $\mathbb E(\epsilon|Z)=0$.
    \item Squared loss: $\ell(X,f)=\ell((Y,Z),f)=(Y-f(Z))^2$\,.
\end{itemize}

Under the regression model \eqref{eq:3.1-reg-model} and the squared loss, it is often more convenient to study the \emph{excess risk} of each fitted candidate model. For each $r\in\mathcal M$ define the excess risk
\begin{equation}
    \label{eq:3.1-excess-risk}
    \Psi_{r,n}\coloneqq R[\hat f^{(r)}(\D_n)] - R(f^*)=\mathbb E_{Z_0} \{[\delta_n^{(r)}(Z_0)]^2\}\,.
\end{equation}
where
$$
\delta^{(r)}_n\coloneqq \hat f^{(r)}(\D_n)-f^*
$$
is the estimation error of the regression function.

Without loss of generality, assume that the first model ($r=1$) is better than the other. Intuitively, this requires $\Psi_{2,n}$ to be larger than $\Psi_{1,n}$. In order to give a precise mathematical notion of ``model $1$ is better than model $2$'', we need to take into account the following issues:
\begin{enumerate}
    \item [(i)] The quantities $\Psi_{1,n}$ and $\Psi_{2,n}$ are random, so the comparison between the models must be stated in a stochastic way.
    \item [(ii)] The separation between $\Psi_{2,n}$ and $\Psi_{1,n}$ needs to be large enough to ensure consistency.
\end{enumerate}
With these issues in mind, the following ``risk dominance'' condition seems natural.
\begin{assumption}[Risk Dominance]\label{asm:3.1-stochastic_domination}
The two sequences of random variables $(\Psi_{r,n}:n\ge 1)$ satisfy $\Psi_{2,n}>0$ almost surely and
  \begin{equation}\label{eq:3.1-risk-dom}\frac{\Psi_{1,n}}{(\Psi_{2,n}-\Psi_{1,n})_+}=O_p(1)\,,\end{equation}
  where $(x)_+\coloneqq \max(x,0)$, with the convention $0/0=\infty$.
\end{assumption}
In other words, the risk dominance condition requires that, with high probability, $\Psi_{2,n}$ exceeds $\Psi_{1,n}$ by at least a constant proportion of $\Psi_{1,n}$.  This definition precludes the scenario of $\Psi_{2,n}/\Psi_{1,n}\rightarrow 1$ in probability, where the two risk sequences become asymptotically equivalent.

We can derive some useful consequences from the risk dominance condition.
\begin{proposition}\label{pro:3.1-conseq-risk-dom}
\begin{enumerate}
    \item [(a)] An equivalent statement of the risk dominance condition \eqref{eq:3.1-risk-dom} is that for any $\varepsilon>0$ there exists $c_\varepsilon>0$ such that
  $$\lim_{n\rightarrow\infty}\inf_{t\ge n}\mathbb P[\Psi_{2,t}>(1+c_\varepsilon)\Psi_{1,t}]\ge 1-\varepsilon\,.$$
  \item [(b)] The risk dominance condition \eqref{eq:3.1-risk-dom} implies $\mathbb P(\Psi_{2,n}\le \Psi_{1,n})=o(1)$.
  \item [(c)] Condition \eqref{eq:3.1-risk-dom} also implies that (with the convention $(0/0)\times 0\equiv 0$)
  $$
  \frac{\Psi_{2,n}}{(\Psi_{2,n}-\Psi_{1,n})_+}\mathds{1}(\Psi_{2,n}>\Psi_{1,n})=O_P(1)\,.
  $$
\end{enumerate}
\end{proposition}
\begin{proof}[Proof of \Cref{pro:3.1-conseq-risk-dom}]
    Part (a) follows from a direct rearrangement of the tail probability of $\frac{\Psi_{1,n}}{(\Psi_{2,n}-\Psi_{1,n})_+}$:
\begin{align*}
    &\mathbb P\left[\frac{\Psi_{1,n}}{(\Psi_{2,n}-\Psi_{1,n})_+}\ge c\right]\\
    =&\mathbb P\left[\frac{(\Psi_{2,n}-\Psi_{1,n})_+}{\Psi_{1,n}}\le c^{-1}\right]\\
    =&\mathbb P\left[\frac{\Psi_{2,n}-\Psi_{1,n}}{\Psi_{1,n}}\le c^{-1}\right]\\
    =&\mathbb P\left[\Psi_{2,n}\le (1+c^{-1})\Psi_{1,n}\right]\,.
\end{align*}
The claim follows by picking $c=c_\varepsilon$ large enough in the above equation so that the probability is no larger than $\varepsilon$ for all sufficiently large $n$. 

Part (b) follows from the fact that $\Psi_{2,n}\le \Psi_{1,n}$ implies $\frac{\Psi_{1,n}}{(\Psi_{2,n}-\Psi_{1,n})_+}=\infty$.

Part (c) follows from the following:
\begin{align*}
   & \frac{\Psi_{2,n}}{(\Psi_{2,n}-\Psi_{1,n})_+}\mathds{1}(\Psi_{2,n}>\Psi_{1,n})\\
   =&\frac{\Psi_{2,n}-\Psi_{1,n}}{(\Psi_{2,n}-\Psi_{1,n})_+}\mathds{1}(\Psi_{2,n}>\Psi_{1,n})+\frac{\Psi_{1,n}}{(\Psi_{2,n}-\Psi_{1,n})_+}\mathds{1}(\Psi_{2,n}>\Psi_{1,n})\\
   =& \mathds{1}(\Psi_{2,n}>\Psi_{1,n})+\frac{\Psi_{1,n}}{(\Psi_{2,n}-\Psi_{1,n})_+}\mathds{1}(\Psi_{2,n}>\Psi_{1,n})\\
   =&O_P(1)\,.\qedhere
\end{align*}
\end{proof}

Now we state the following pairwise comparison consistency result for $K$-fold cross-validation, originally due to \cite{yang2007consistency}. Recall that in $K$-fold cross-validation, the training sample size is $n_\tr=n(1-1/K)$ and the test sample size is $n_\te=n-n_\tr=n/K$.  
\begin{theorem}\label{thm:3.1-yang-consistency}
Assume $\D_n$ consists of i.i.d. data from the regression model \eqref{eq:3.1-reg-model} with $\epsilon$ independent of $Z$, $\mathbb E\epsilon=0$, and ${\var}(\epsilon)=\sigma^2<\infty$.
    For two estimation methods $\hat f^{(r)}$, $r=1,2$, assume $\|\hat f^{(r)}(\D_n)-f^*\|_\infty\le B<\infty$ for $r=1,2$ and all sample sizes $n$. If the excess risks $\Psi_{r,n}$ $(r=1,2)$ satisfy \Cref{asm:3.1-stochastic_domination} and 
   \begin{equation}\label{eq:3.1-psi2-lower}
   n_\te \Psi_{2,n_\tr}\stackrel{P}{\rightarrow}\infty\,,~~n_\tr\rightarrow\infty\,,\end{equation} 
  then, for $K$-fold CV risk estimates with $K=O(1)$, we have $$\lim_{n\rightarrow\infty}\mathbb P(\hat R^{(1)}_{\cv,n,K}<\hat R^{(2)}_{\cv,n,K})=1\,.$$
\end{theorem}

Before proving \Cref{thm:3.1-yang-consistency}, we briefly discuss the additional conditions stated above.  The boundedness condition $\|\hat f^{(r)}(\D_n)-f^*\|_\infty\le B$ is more or less standard, and is needed to establish rates of convergence for empirical approximations of the excess risks using the test sample. It can be relaxed to $\|\hat f^{(r)}(\D_n)-f^*\|_{4,Z}^4\le B^2 \|\hat f^{(r)}(\D_n)-f^*\|_{2,Z}^2$, where $\|f\|_{q,Z}^q=\mathbb E_Z |f(Z)|^q$, which is plausible whenever $\hat f^{(r)}(\D_n)(Z_0)-f^*(Z_0)=O_P(1)$ and is not too heavy-tailed. The second condition $n_{\te}\Psi_{2,n_\tr}\stackrel{P}{\rightarrow}\infty$ is more nontrivial and specifies a requirement on the test and training split ratio determined by the convergence rate of the inferior model. Intuitively, the test sample evaluates the empirical risk with a standard error rate $\sqrt{\Psi_{2,n_\tr}}/\sqrt{n_\te}$. Then we need $\Psi_{2,n_\tr}-\Psi_{1,n_\tr}$, which is asymptotically the same order as $\Psi_{2,n_\tr}$, to dominate this risk evaluation error.  We illustrate the effect of split ratio in a simple example below, which also reflects the tightness of the condition \eqref{eq:3.1-psi2-lower}.  Further extensions, such as other sample splitting schemes and diverging $K$, and the implications of \Cref{thm:3.1-yang-consistency} in traditional parametric and nonparametric regression settings, are discussed in Remarks \ref{rem:3.1-arbitrary-split} and \ref{rem:3.1-diverging-K}.

\begin{example}\label{exa:3.1-simple-example}
    Consider the following comparison of two models: $Y_i\sim N(0,1)$,
    $\hat f^{(1)}\equiv 0$, and $\hat f^{(2)}(\D_n) \equiv n^{-1}\sum_{i=1}^n Y_i$. So both $\hat f^{(1)}$ and $\hat f^{(2)}$ are constant functions, where $\hat f^{(1)}$ uses the constant $0$ and $\hat f^{(2)}$ uses the sample average in $\D_n$. In this case, $f^*=\hat f^{(1)}$.  By construction, $\Psi_{1,n}=0$ for all $n$ and $\Psi_{2,n}=(n^{-1}\sum_{i=1}^n Y_i)^2\asymp 1/n$. Let $\hat R_{\ssp,r}$ ($r=1,2$) be the contribution to the CV risk estimate from a single test sample fold, defined in \eqref{eq:3.1-r_ss} below.  It is straightforward to verify that
    $$n_\te(\hat R_{\ssp,2}-\hat R_{\ssp,1})=-2\bar Y_\te\bar Y_\tr + \bar Y_\tr^2\,,$$
    and $\mathbb P(\hat R_{\ssp,2}-\hat R_{\ssp,1}\le 0)\rightarrow 0$ if and only if \eqref{eq:3.1-psi2-lower} holds.
\end{example}

\begin{proof}[Proof of \Cref{thm:3.1-yang-consistency}]
We drop the dependence on $n$ whenever there is no confusion.
    We focus on such a single sample split with $\D_\tr$ and $\D_\te$ denoting the fitting and validation data, with corresponding index sets $I_\tr$, $I_\te$.   

For $r=1,2$, use simplified notation
$$\hat f_{\tr}^{(r)}=\hat f^{(r)}(\D_\tr)\,,~~~ \delta_r=\hat f_{\tr}^{(r)}-f^*\,.$$
Observe that $\mathbb E_Z \delta_r^2(Z)=\|\delta_r\|_{2,Z}^2=\Psi_{r,n}$.
For $r=1,2$,
consider the sample-split risk estimate $\hat R_{\ssp,r}$ defined as
\begin{equation}\label{eq:3.1-r_ss}\hat R_{\ssp, r}\coloneqq n_\te^{-1}\sum_{i\in I_\te}\ell(X_i,\hat f_{\tr}^{(r)})= n_\te^{-1}\sum_{i\in I_\te} \left[\epsilon_i^2 + 2\epsilon_i \delta_r(Z_i) +\delta_r^2(Z_i)\right]\,.
\end{equation}
Further define,
$$
\xi_i = 2\epsilon_i[\delta_1(Z_i)-\delta_2(Z_i)]+\delta_1^2(Z_i)-\delta_2^2(Z_i)-\Psi_{1,n_\tr} + \Psi_{2,n_\tr}\,.
$$
Our argument is based on the following two facts about $\xi_i$.

Fact (i): Conditioning on $\D_\tr$, $(\xi_i:i\in I_\te)$ are i.i.d. with mean zero.

Fact (ii), $\hat R_{\ssp,1}\ge \hat R_{\ssp,2}$ if and only if
\begin{equation}\label{eq:3.1-violation}
    \sum_{i\in I_\te} \xi_i \ge n_\te (\Psi_{2,n_\tr}-\Psi_{1,n_\tr})\,.
\end{equation}

The plan is to use Tchebyshev's inequality to bound the probability of the event in \eqref{eq:3.1-violation}. We first need to control the conditional variance of $\xi_i$ given $\D_\tr$.  
\begin{align}
    {\rm Var}(\xi_i|\D_\tr) = & {\rm Var}\left\{2\epsilon_i[\delta_1(Z_i)-\delta_2(Z_i)]+\delta_1^2(Z_i)-\delta_2^2(Z_i)\big|\D_\tr\right\}\nonumber\\
    = & 4 {\rm Var}\left\{\epsilon_i[\delta_1(Z_i)-\delta_2(Z_i)]\big|\D_\tr\right\}+{\rm Var}\left\{\delta_1^2(Z_i)-\delta_2^2(Z_i)\big|\D_\tr\right\}\nonumber\\
    \le & (4 \sigma^2+4B^2)\mathbb E\left\{[\delta_2(Z_i)-\delta_1(Z_i)]^2|\D_\tr\right\}\nonumber\\
    \le & 8(\sigma^2+B^2) (\Psi_{2,n_\tr}+\Psi_{1,n_\tr})\,,\label{eq:3.1-cond-var-bound}
\end{align}
where the first inequality uses $\mathbb E(\epsilon_i|Z_i)=0$ and the boundedness of $\delta_1+\delta_2$, and the second inequality follows from Cauchy–Schwartz.

Let $\mathcal E$ be the event that $\Psi_{2,n_\tr}-\Psi_{1,n_\tr}>0$.
Then
\begin{align}
\mathbb P(\hat R_{\ssp,1}\ge\hat R_{\ssp,2})\le & \mathbb P(\hat R_{\ssp,1}>\hat R_{\ssp,2}\cap\mathcal E)+\mathbb P(\mathcal E^c)\nonumber\\
=&\mathbb E\left[\mathbb P\left(\sum_{i\in I_\te} \xi_i \ge n_\te (\Psi_{2,n_\tr}-\Psi_{1,n_\tr})\bigg|\D_\tr\right)\mathds{1}(\E)\right]+\mathbb P(\E^c)\,,\label{eq:3.1-err-prob-upper}
\end{align}
because $\mathcal E$ is $\D_\tr$-measurable.

Using Tchebyshev's inequality conditionally given $\D_\tr$ we have
\begin{align*}
   & \mathbb P\left(\sum_{i\in I_\te} \xi_i \ge n_\te (\Psi_{2,n_\tr}-\Psi_{1,n_\tr})\bigg|\D_\tr\right)\mathds{1}(\E)\le  \frac{n_\te{\var}(\xi_i|\D_\tr)}{n_\te^2(\Psi_{2,n_\tr}-\Psi_{1,n_\tr})^2}\mathds{1}(\E)\\
    \le & \frac{8(\sigma^2+B^2)(\Psi_{1,n_\tr}+\Psi_{2,n_\tr})}{n_\te(\Psi_{2,n_\tr}-\Psi_{1,n_\tr})^2}\mathds{1}(\E)\\
    =& O_P\left(\frac{1}{n_\te\Psi_{2,n_\tr}}\right)\\
    =& o_P(1)\,,
\end{align*}
where the second inequality follows from the conditional variance bound \eqref{eq:3.1-cond-var-bound}, the $O_P(\cdot)$ term follows from part~(c) of \Cref{pro:3.1-conseq-risk-dom}, and the $o_P(\cdot)$ term follows from the theorem assumption \eqref{eq:3.1-psi2-lower}.

As a result,
\begin{equation}
\label{eq:3.1-err-prob-upper-1}\mathbb E\left[\mathbb P(\hat R_{\ssp,1}>\hat R_{\ssp,2}|\D_\tr)\mathds{1}_{\E}\right]=o(1)\,,\end{equation}
where the exchange of expectation and limit is guaranteed by the boundedness of $\mathbb P(\hat R_{\ssp,1}>\hat R_{\ssp,2}|\D_\tr)\mathds{1}_{\E}$.

Finally, part (b) of \Cref{pro:3.1-conseq-risk-dom} guarantees that $\mathbb P(\E^c)=o(1)$. Combining this with \eqref{eq:3.1-err-prob-upper-1} we conclude that the RHS of \eqref{eq:3.1-err-prob-upper} is $o(1)$. 
\end{proof}

\begin{remark}[Other variants of $K$-fold CV]\label{rem:3.1-arbitrary-split}
    \Cref{thm:3.1-yang-consistency} also applies to other CV variants than the standard $K$-fold CV, such as repeated random split, where $I_\tr$ is repeatedly generated by uniformly sampling $n_\tr$ indices from $[n]$ without replacement, and $I_\te=[n]\setminus I_\tr$.  Another example is ``reversed $K$-fold CV'', which uses the same $K$-fold split but $I_\tr=I_{n,k}$ and $I_\te=I_{n,-k}$ for $k\in[K]$.
\end{remark}

\begin{remark}[Large values of $K$]\label{rem:3.1-diverging-K}
    The only reason to require $K=O(1)$ in \Cref{thm:3.1-yang-consistency} is to allow a simple union bound so that the consistency of a single split carries over to the $K$-split case.  If specific rates of convergence are available in conditions~\eqref{eq:3.1-risk-dom} and \eqref{eq:3.1-psi2-lower}, then we can allow $K$ to grow slowly, as permitted by the union bound argument.
\end{remark}

\paragraph{Implications and practical concerns.}
We briefly discuss more concrete versions of the sufficient conditions implied by \Cref{thm:3.1-yang-consistency} in standard parametric and nonparametric regression models.   
\begin{enumerate}
    \item Parametric case. If both models are parametric and correctly specified, then $\Psi_{2,n}\asymp 1/n$ and \Cref{thm:3.1-yang-consistency} requires $n_\te/n_\tr\rightarrow\infty$, which turns out to be tight, as seen in \Cref{exa:3.1-simple-example}. Such a split ratio is not achievable by any $K$-fold split but is achievable by the reversed $K$-fold split with a diverging $K$. 
    \item Nonparametric case. When $\Psi_{2,n}\asymp n^{-\alpha}$ for $\alpha\in[0,1)$, \Cref{thm:3.1-yang-consistency} implies that $n_\te=\omega( n^{\alpha})$ is sufficient for model selection consistency using sample splitting and hence $K$-fold CV with a constant $K$.
\end{enumerate}
In practice, reversed $K$-fold cross-validation (CV) is rarely used or recommended because it severely reduces the training sample size. The principal reason is that such an extreme split ratio is only required to guarantee model selection consistency when squared loss risks differ on the order of $O(n^{-1})$---a regime of limited practical relevance when prediction accuracy, rather than model identification, is the primary objective. CV is more commonly used to tune nonparametric and high-dimensional procedures, for which it is important to keep the training sample size close to the full sample. In the next two subsections, we introduce a variant of CV that achieves strong model selection properties under conventional split ratios. In \Cref{sec:5-applications}, we present methods based on model confidence sets that accommodate weak separation among candidate models. We further discuss the choice of split ratio in \Cref{subsec:7.2-choose-k}.

\subsection{Online Model Selection}\label{subsec:3.2-online-model-selection}
In this subsection, we introduce the online model selection problem, setting the stage for a variant of cross-validation tailored for the online setting.

So far, we have only considered the batch setting, where the fitting algorithm $\hat f$ takes in a static dataset $\D_n$ of an arbitrary sample size $n$ and outputs $\hat f(\D_n)\in\mathcal F$.  The online setting, which becomes increasingly popular in large-scale machine learning, differs from the batch setting with profound algorithmic and computational consequences.
\begin{definition}[Online Estimator]\label{def:3.2-online-alg}
Let $\X$ be the sample space of a single data point, and $\mathcal F$ a parameter space.
An online algorithm $\hat f=(\hat f_n:n\ge 1)$ is an infinite sequence of mappings such that
\begin{enumerate}
    \item [(a)] $\hat f_1:\X \mapsto \mathcal F$\,,
    \item [(b)] For $i\ge 2$, $\hat f_i:\mathcal F\times \X \mapsto \mathcal F$\,,
    \item [(c)] After obtaining $\hat f_i$, the sample points $\{X_j:j\le i\}$ and previous estimates $\{\hat f_j:j<i\}$ are not used again by the algorithm in the future.
\end{enumerate}
\end{definition}
An online estimator $\hat f=(\hat f_i:i\ge 1)$ takes in a (possibly infinite) stream of sample points $(X_i:i\ge 1)$ in $\X$, with an initial estimate $\hat f_0$, and updates the estimate from $\hat f_i$ to $\hat f_{i+1}=\hat f_{i+1}(\hat f_i, X_{i+1})$ after seeing the $(i+1)$th sample point ($i\ge 0$).  Part (c) ensures that the online algorithm satisfies the ``single-pass'' (or constant memory) property, meaning that the sample points and estimates are only processed at the time of sampling and will not be stored.  This does not exclude the popular Polyak averaging estimate, which outputs the final estimate as the average of $(\hat f_i:i\ge 1)$, because we only need to store the current average estimate, which costs a constant amount of memory.

Online algorithms often come with great computational efficiency and play a crucial role in modern machine learning.
Perhaps the most important example of such online estimators is stochastic gradient descent.
\begin{example}[Online Stochastic Gradient Descent]\label{exa:3.2-sgd}
    Suppose the estimation target is $f^*=\arg\min \mathbb E_X\varphi(f,X)$, for an objective function $\varphi(\cdot,\cdot)$. The stochastic gradient descent estimate $\hat f_n$ estimates $f^*$ by minimizing the empirical average objective function $\sum_{i=1}^n \varphi(f,X_i)$ using an online update scheme.
    \begin{enumerate}
        \item [(i)] Initialize $\hat f_0=0$
        \item [(ii)] For each $i\ge 1$, $\hat f_i = \hat f_{i-1}-\alpha_i\dot\varphi(\hat f_{i-1},X_i)$, where $\dot\varphi(\cdot,\cdot)$ is the partial gradient of $\varphi(\cdot,\cdot)$ with respect to the first argument.
    \end{enumerate}
    The sequence $(\alpha_i:i\ge 1)$ specifies the learning rate at each step.
\end{example}

To provide a more concrete example, and better connect to recent literature, we consider the following variant of ``sieve-SGD'' tailored for nonparametric regression under squared loss, where each sample point $X=(Y,Z)$, with $Y\in\mathbb R^1$ and $\mathcal F$ being a class of functions mapping from $\mathcal Z$ to $\mathbb R^1$. Let $(\phi_j:j\ge 1)$ be a set of orthonormal basis functions in $L_{2,Z}$. The sieve-SGD initializes $\hat f_0=0$ and the update from $\hat f_{i-1}$ to $\hat f_{i}$ is:
\begin{equation}\label{eq:3.2-sieve-sgd}
    \hat f_{i} = \hat f_{i-1}+\alpha_{i}\left[Y_i - \hat f_{i-1}(Z_i)\right]\sum_{j=1}^{J_i}j^{-w}\phi_j(Z_i)\phi_j\,.
\end{equation}
Here $\alpha_i$ is the learning rate, $J_i\in \mathbb N$ is the number of basis functions used at sample size $i$, and $w$ is a shrinkage exponent.  It has been shown in \cite{zhang2022sieve} that with suitable choices of $(\alpha_i,J_i)$ and $w$, the sieve-SGD estimate nearly achieves the minimax optimal error rate over Sobolev classes of functions.

When model selection is of interest, a fundamental difference emerges between the batch setting and the online setting.  In the batch setting, the evaluation criterion for model selection is often tied to a particular sample size, which is typically $n$ but could also be $n_\tr$ such as in $K$-fold cross-validation.  The batch cross-validation can approximate the quantities $\mu_{n_\tr}$ and $\bar R_{\cv,n,K}$ because the algorithm can access versions of $\hat f(\D_{n_\tr})$ through $\hat f(\D_{n,-k})$, and repeatedly evaluate their losses at many independent sample points.   This is not possible in the online setting, where, as specified in \Cref{def:3.2-online-alg}, for each sample size $i$, the loss function of the estimate $\hat f_i$ at sample size $i$ can only be evaluated once using $\ell(X_{i+1}, \hat f_i)$. As a result, the model selection criterion in the online setting must have the following two distinct features.
\begin{enumerate}
    \item The model selection criterion in the online setting needs to be asymptotic, as it must describe the behavior of an infinite sequence.
    \item As a consequence, in the online setting, a candidate in the model selection problem must be a sequence itself, which specifies the update rule for each sample size.
\end{enumerate}
Given these thoughts, suppose we have $m$ online algorithms, $(\hat f_i^{(r)}:i\ge 1)$ for $r=1,...,m$, we can define the online model selection target as the $r^*\in[m]$ such that, assuming its existence, 
\begin{equation}\label{eq:3.2-online-criterion}
\max_{r\neq r^*}\limsup_{i\rightarrow\infty}\frac{\mathbb E \left[\Psi_{r^*,i}\right]}{\mathbb E \left[\Psi_{r,i}\right]} < 1\,,
\end{equation}
where $\Psi_{r,i}$ is the excess risk of the fitted model $\hat f_i^{(r)}$, as defined in \eqref{eq:3.1-excess-risk} in the squared loss regression setting.

Such an optimal $r^*$ naturally exists in many nonparametric regression settings. Consider the sieve-SGD example in \eqref{eq:3.2-sieve-sgd}. For simplicity, suppose we only need to choose the number of basis $(J_i:i\ge 1)$. Existing theory \cite{zhang2022sieve} suggests the optimal choice is $J_i=B i^\tau$ for a pair of positive constants $(B,\tau)$ that depends on the function class.  Thus, the candidate set of $m$ online algorithms correspond to $m$ distinct combinations $\{(B^{(r)},\tau^{(r)}):r=1,...,m\}$, where for each $r\in[m]$, the corresponding online algorithm uses the update rule \eqref{eq:3.2-sieve-sgd} with $J_i=B^{(r)}i^{\tau^{(r)}}$.

In this example, an online model selection procedure should point us to a choice of sequence $\{J_i^{(r)}:i\ge 1\}$ that will eventually become the best for each $i$ when the sample size keeps increasing. In contrast, standard batch model selection methods only focus on a single sample size and do not suggest a good choice of model for a different sample size. For example, if batch model selection returns a choice of $J=10$ at 
sample size $100$, it is unclear whether this $J$ comes from a $(B,\tau)$ combination of $(1,~1/2)$ or $(\sqrt{10},~1/4)$, which would lead to different choices of $J$ for other values of sample size.

\subsection{Rolling Validation}\label{subsec:3.3-rv}
In this subsection, we present an online version of the LOOCV model selection method, and establish consistency results under the same linear regression setting with the squared loss as in \Cref{subsec:3.1-yang07}.

Recall that in the regression model, we have $X_i=(Y_i,Z_i)$ such that $Y_i=f^*(Z_i)+\epsilon_i$ with zero-mean noise $\epsilon_i$ independent of $Z_i$.  The loss function is $\ell(x,f) = (y-f(z))^2$.

Suppose we have two sequences of online estimates of $f^*$: $(\hat f_i^{(r)}:i\ge 1)$ for $r\in\{1,2\}$ satisfying \Cref{def:3.2-online-alg}.
  Again, the excess risk $\Psi_{r,n}$, defined in \eqref{eq:3.1-excess-risk}, will play an important role.   
Unlike in the batch setting, where we used the stochastic version of model quality gap (\Cref{asm:3.1-stochastic_domination}), in the online setting we will use a more transparent approach, where the model quality gap is stated in the expected excess risks
\begin{equation}\label{eq:3.3-expected-excess-risk}\psi_{r,n}\coloneqq\mathbb E \Psi_{r,n}\,.\end{equation}
We will assume that the expected excess risks converge at a certain rate, serving the purpose of model quality comparison.
\begin{assumption}[Model quality]\label{asm:3.3-rolling_model_quality} There exist $0\le a_2\le a_1\le 1$ such that
$$\limsup_{n\rightarrow\infty} n^{a_1}\psi_{1,n}<
\liminf_{n\rightarrow\infty} n^{a_2}\psi_{2,n}\le \limsup_{n\rightarrow\infty}n^{a_2}\psi_{2,n}<\infty\,.$$
\end{assumption}
\Cref{asm:3.3-rolling_model_quality} involves a chain of three inequalities. The most crucial one is the first inequality, which implies that the expected excess risk of model $1$ is less than that of model $2$ with a sufficient gap.  When the two models have the same rates of convergence ($a_1=a_2$), the strict inequality ensures that model $1$ is better than model $2$ by at least a constant fraction, similar to \Cref{asm:3.1-stochastic_domination}. A notable difference is that \Cref{asm:3.3-rolling_model_quality} makes assumptions about the expected excess risks, which are deterministic quantities. In contrast, \Cref{asm:3.1-stochastic_domination} is about the risks of fitted models, which are random quantities.  

While it is more convenient to describe the model quality gap in terms of the expected excess risk $\psi_{r,n}$, the randomness of online model fittings needs to be controlled in order to guarantee concentration of $\Psi_{r,n}$ around $\psi_{r,n}$.  This will be achieved through stability conditions in \Cref{asm:3.3-rolling_stability} below.
\begin{assumption}[Estimator stability]\label{asm:3.3-rolling_stability} There exist positive constants $c$, $b_1$, and $b_2$ such that, for $j\le i$, $r=1,2$,
\begin{equation}\label{eq:3.3-wrv-stability}
    \mathbb E\left\{\left[\nabla_j \hat f^{(r)}_i(X_{i+1})\right]^2\bigg|\D_j\right\}\le c i^{-2 b_r}\,,~~\text{almost surely}\,.
\end{equation}
\end{assumption}
Recall that $\nabla_j$ is the perturb-$j$th-entry operator defined in \eqref{eq:2.2-nabla}. If we ignore the conditional expectation, which makes the $L_2$-norm smaller, \eqref{eq:3.3-wrv-stability} roughly requires that the value of $\hat f_i$ at a future sample point change no more than $i^{-b_r}$ when one of the fitting sample points is replaced by an i.i.d.\ copy.  Intuitively, if each fitting sample point contributes equally to the estimate, then the influence of a single sample point on the outcome should be of order $i^{-1}$ and such a stability condition can be plausible if $b_r\le 1$. The almost-sure bound condition can be relaxed to $L_q$ norm bounds with different choices of conjugate indices in H\"{o}lder's inequality, together with stronger moment conditions on the risk function.  We provide examples satisfying \eqref{eq:3.3-wrv-stability} in \Cref{subsec:6.4-stab-online}.

We consider the following online \emph{rolling validation}  risk estimate.
\begin{equation}\label{eq:3.3-rv-stat}
\Xi_{r,n} \coloneqq \sum_{i=1}^n \ell(X_{i+1},\hat f^{(r)}_{i})\,.
\end{equation}
The RV risk estimate has some promising features.
\begin{enumerate}
    \item The RV risk estimate $\Xi_{r,n}$ resembles the leave-one-out cross-validation, as each fitted model $\hat f_i^{(r)}$---fitted using all sample points up to time $i$---is evaluated at one independent sample point $X_{i+1}$.
    \item The RV risk estimate $\Xi_{r,n}$ is fully compatible with the online learning protocol described in \Cref{def:3.2-online-alg}, as one can obtain $\Xi_{r,n+1}$ from $(\Xi_{r,n},X_{n+1},\hat f_n^{(r)})$ without requiring access to any $\Xi_{r,i}$, $X_i$, or $\hat f_i^{(r)}$ for $i<n$.
    \item When $n$ becomes large, $\Xi_{r,n}$ is increasingly representative of the asymptotic behavior of $\Psi_{r,n}$ and $\psi_{r,n}$, thereby serving the purpose of online model selection objective in \eqref{eq:3.2-online-criterion}.
\end{enumerate}
The intuition is that, if the fitting procedure does not too heavily rely on any single sample point, as implied by the stability condition \eqref{eq:3.3-wrv-stability}, then the contribution of each single sample point to $\Psi_{r,i}$ will vanish as $i$ increases.  Thus each single sample point still only contributes a small proportion of the sum $\Xi_{r,n}$, making the sum concentrate around its expected value.  This is formally stated in the following theorem.
\begin{theorem}[Model Selection Consistency of Rolling Validation]\label{thm:3.3-wrv-consistency}
Under the regression model \eqref{eq:3.1-reg-model} and the squared loss function,
assume that the function space is uniformly bounded: $\|f(Z)\|_\infty\le B<\infty$ for all $f\in\mathcal F$.
   If \Cref{asm:3.3-rolling_model_quality} and \Cref{asm:3.3-rolling_stability} hold with \begin{align*}
         b_1> &\frac{1}{2}+a_2-\frac{a_1}{2}\,,  \\
         b_2> & \frac{1}{2}+\frac{a_2}{2} \,,
\end{align*}
then the rolling validation risks $\Xi_{r,n}$ defined in \eqref{eq:3.3-rv-stat} satisfy
$$
\lim_{n\rightarrow \infty}\mathbb P(\Xi_{1,n}<\Xi_{2,n})= 1\,.
$$
\end{theorem}

\begin{proof}[Proof of \Cref{thm:3.3-wrv-consistency}]
Denote $\delta_{r,n}=(\hat f^{(r)}_n-f^*)(Z_{n+1})$, $U_n=\delta_{1,n}^2-\delta_{2,n}^2$. We have
\begin{align*}
&\Xi_{1,n}-\Xi_{2,n}\\=&\sum_{i=1}^n\left[ U_i-2\epsilon_{i+1}(\delta_{1,i}-\delta_{2,i})\right]\\
=& \underbrace{-2\sum_{i=1}^n \epsilon_{i+1}(\delta_{1,i}-\delta_{2,i})}_{(I)}+\underbrace{\sum_{i=1}^n\left[ U_i-\mathbb E(U_i|\D_{i})\right]}_{(II)}+\underbrace{\sum_{i=1}^n\left[\mathbb E(U_i|\D_{i}) - \mathbb E U_i\right]}_{(III)}+\sum_{i=1}^n (\psi_{1,i}-\psi_{2,i})\,,
\end{align*}
where, recall the notation $\D_i=(X_j:1\le j\le i)$.
The first three terms are mean zero noises, and the last term is the signal, which will be negative for large values of $n$.  The plan is to derive upper bounds for each of the noise terms so that their sum is unlikely to override the signal.  In this proof, since the claimed result is consistency, we will not keep track of all constant factors. To simplify notation, we will use ``$\lesssim$'' to denote that the inequality holds up to a constant factor.

\textbf{Controlling term (I):} Term (I) has zero mean and martingale increments, and can be bounded similarly as in the proof of \Cref{thm:3.1-yang-consistency}:
\begin{align*}
    \var\left(\sum_{i=1}^n\epsilon_{i+1}(\delta_{1,i}-\delta_{2,i})\right)=&\sigma^2\sum_{i=1}^n\mathbb E (\delta_{1,i}-\delta_{2,i})^2\\
    \le & 2\sigma^2 \sum_{i=1}^n (\psi_{1,i}+\psi_{2,i})\\
     \lesssim & \sum_{i=1}^n i^{-a_1}+i^{-a_2}\\
     \lesssim & n^{1-a_2}\vee\log n\,,
\end{align*}
since $a_2\le a_1$ by \Cref{asm:3.3-rolling_model_quality}.

\textbf{Controlling term (II):} Term (II) also has zero mean and martingale increments.
\begin{align*}
\var\left(\sum_{i=1}^n\left[ U_i-\mathbb E(U_i|\D_{i})\right]\right)=&\sum_{i=1}^n \mathbb E\left[ U_i-\mathbb E(U_i|\D_{i})\right]^2\\
\le & \sum_{i=1}^n \mathbb E U_i^2\le 2B^2\sum_{i=1}^n(\psi_{1,i}+\psi_{2,i})\\
\lesssim & n^{1-a_2}\vee \log n\,.
\end{align*}

\textbf{Controlling term (III):} Term (III) does not have martingale increments.  Each term $\mathbb E(U_i|\D_{i})-\mathbb E U_i$ is a function of $\D_i$, which involves all sample points $(X_j:1\le j\le i)$.
We will use a further martingale decomposition of this term:
\begin{align*}
&  \sum_{i=1}^n  \left[\mathbb E(U_i|\D_i)-\mathbb E U_i\right]\\
=& \sum_{i=1}^n \left[\sum_{j=1}^i \left\{\mathbb E(U_i|\D_j)-\mathbb E(U_i|\D_{j-1})\right\}\right]\\
=&\sum_{j=1}^n \sum_{i=j}^n\left[\mathbb E(U_i|\D_j)-\mathbb E(U_i|\D_{j-1})\right]\\
\coloneqq& \sum_{j=1}^n M_{n,j}
\end{align*}
where $M_{n,j}=\sum_{i=j}^n\left[\mathbb E(U_i|\D_j)-\mathbb E(U_i|\D_{j-1})\right]$, is a martingale increment sequence with respect to the filtration $(\D_j:j\ge 1)$.  Now we control the $L_2$ norm of $M_{n,j}$. Observe that
$$
\mathbb E(U_i|\D_j)-\mathbb E(U_i|\D_{j-1}) =\mathbb E (\nabla_j U_i |\D_{j})\,,
$$
where the $\nabla$ notation is defined in \eqref{eq:2.2-nabla}.

Recall that $\D_n^j=(X_1,...,X_j',...,X_n)$ is the perturb-one dataset obtained by replacing the $j$th entry of $\D_n$ by an i.i.d. copy $X_j'$.  Let $\delta_{r,i}^j$ be the counterpart of $\delta_{r,i}$ obtained as if the input dataset were $\D_n^j$.
By construction, 
$$\nabla_j U_i = (\nabla_j \delta_{1,i})(\delta_{1,i}+\delta_{1,i}^j)-(\nabla_j \delta_{2,i})(\delta_{2,i}+\delta_{2,i}^j).$$
We have
\begin{align*}
    &\left\|\mathbb E(\nabla_j U_i|\D_j)\right\|_2\\
    \le& \left\|\mathbb E\left[(\nabla_j \delta_{1,i})\delta_{1,i}\big|\D_j\right]\right\|_2+ \left\|\mathbb E\left[(\nabla_j \delta_{1,i})\delta_{1,i}^j\Big|\D_j\right]\right\|_2\\
    &\quad+\left\|\mathbb E\left[(\nabla_j \delta_{2,i})\delta_{2,i}\big|\D_j\right]\right\|_2+ \left\|\mathbb E\left[(\nabla_j \delta_{2,i})\delta_{2,i}^j\Big|\D_j\right]\right\|_2\\
    \lesssim & i^{-b_1}\psi_{1,i}^{1/2}+i^{-b_2}\psi_{2,i}^{1/2}
\end{align*}
where the last step follows from conditional H\"{o}lder's inequality and conditional Jensen's inequality.
By \Cref{asm:3.3-rolling_model_quality}, $\psi_{r,i}\lesssim i^{-a_r}$, and hence
\begin{align*}
    \|M_{n,j}\|_2 = & \left\|\sum_{i=j}^n \mathbb E(\nabla_j U_i|\D_{j})\right\|_2\\
    \lesssim &\sum_{i=j}^n i^{-b_1}\psi_{1,i}^{1/2}+i^{-b_2}\psi_{2,i}^{1/2}\\
    \lesssim& \sum_{i=j}^n i^{-b_1-a_1/2} + i^{-b_2-a_2/2}\\
    \lesssim&\sum_{i=j}^n i^{-c}\\
    \lesssim&\left\{\begin{array}{ll}
        j^{1-c}-(n+1)^{1-c} & \text{if }c>1\,, \\
        \log (n+1)-\log j & \text{if } c=1\,,\\
        (n+1)^{1-c}-j^{1-c} & \text{if }c<1\,,
    \end{array}\right.
\end{align*}
where $c=\min(b_1+a_1/2,~b_2+a_2/2)$.

Using the martingale property, Term (III) has variance
\begin{align*}
    \left\|\sum_{i=1}^n  \left[\mathbb E(U_i|\D_i)-\mathbb E U_i\right]\right\|_2^2 =  \left\|\sum_{j=1}^n M_{n,j}\right\|_2^2=\sum_{j=1}^n \|M_{n,j}\|_2^2\,,
\end{align*}
whose upper bound depends on the value of $c$.
\begin{itemize}
    \item When $c>1$,
    \begin{align*}
        \sum_{j=1}^n(j^{1-c}-(n+1)^{1-c})^2
        \lesssim \left\{\begin{array}{ll}
            n^{3-2c} & \text{if } c<3/2 \\
            \log n & \text{if } c\ge 3/2\,.
        \end{array}\right. 
    \end{align*}
    \item When $c=1$, using Stirling's approximation,
    \begin{align*}
        &\sum_{j=1}^n (\log(n+1)-\log j)^2\\
        =&n\log^2(n+1) - 2\log(n+1)\sum_{j=1}^n\log j+\sum_{j=1}^n \log^2j\\
        \le & n \log^2 n - 2\log n(n\log n-n)+n\log^2 n\\
        = & 2n\log n\,.
    \end{align*}
    \item When $c<1$,
    \begin{align*}
        \sum_{j=1}^n((n+1)^{1-c}-j^{1-c})^2\lesssim n^{3-2c}\,.
    \end{align*}
\end{itemize}
So in all cases we have
$$\left\|\sum_{i=1}^n  \left[\mathbb E(U_i|\D_i)-\mathbb E U_i\right]\right\|_2\lesssim n^{3/2-c}\sqrt{\log n}\,.$$

Putting together, we have the noise terms bounded by
\begin{align*}
    (I)+(II)+(III) = O_P(n^{1/2-a_2/2}\vee\sqrt{\log n})+O_P(n^{3/2-c}\sqrt{\log n})\,.
\end{align*}

On the other hand, the signal term is
$$
\sum_{i=1}^n (\psi_{2,i}-\psi_{1,i})\gtrsim n^{1-a_2}\vee\log  n\,.
$$
Thus a sufficient condition for $\mathbb P(\Xi_{1,n}<\Xi_{2,n})\rightarrow 1$ is $c>a_2+1/2$, which is equivalent to
\begin{align*}
    b_1> &\frac{1}{2}+a_2-\frac{a_1}{2}\,,  \\
    b_2> & \frac{1}{2}+\frac{a_2}{2} \,.\qedhere
\end{align*}
\end{proof}


\begin{remark}[Growing number of candidate models]\label{rem:3.3-growing-M}
 If we strengthen the stability condition to $\nabla_j\hat f^{(r)}_i(X_{i+1})$ being $(i^{-b_r},\alpha)$-sub-Weibull, then we can derive exponential probability bounds for correct model selection. This would allow the number of candidate models to diverge sub-exponentially as the sample size grows.  Online learning with a growing number of candidate models is naturally motivated by the need to gradually refine the tuning grid as the training sample size grows, which can be viewed as a sieve method on the regularity parameter space.    However, such stability conditions will be harder to rigorously verify.  In \Cref{sec:6-stability} we provide further discussions and examples of stability bounds for some commonly used estimators such as stochastic gradient descent.
\end{remark}

\subsection{Weighted Rolling Validation}\label{subsec:3.4-wrv}
In the online model selection criterion \eqref{eq:3.2-online-criterion} and the model quality assumption \Cref{asm:3.3-rolling_model_quality}, it is clear that the better model is defined through its superior asymptotic prediction risk. When reflected in the rolling validation risk $\Xi_{r,n}$, the expected model quality difference is 
$$
\mathbb E(\Xi_{1,n}-\Xi_{2,n}) = \sum_{i=1}^n (\psi_{1,i}-\psi_{2,i})\,.
$$
Under \Cref{asm:3.3-rolling_model_quality}, this quantity diverges as $n$ increases, and is driven by the contributions from large values of $i$.  In other words, these assumptions suggest that the model quality is better reflected by the performance of $\hat f_{i}^{(r)}$ at larger values of $i$. This intuition leads to the following weighted version of rolling validation (wRV).
\begin{equation}\label{eq:3.4-wrv}
    \Xi_{r,n,\xi} \coloneqq \sum_{i=1}^n i^\xi \ell(X_{i+1},\hat f_i^{(r)})\,,
\end{equation}
where $\xi\ge 0$ is a positive constant controlling the fraction of weight tilted towards the larger values of $i$.

The original rolling validation risk  $\Xi_{r,n}$ defined in \eqref{eq:3.3-rv-stat} corresponds to the special case of $\Xi_{r,n,0}$, and it is straightforward to check that the wRV risks are also fully compatible with the online learning protocol in \Cref{def:3.2-online-alg}.  More generally, we can write wRV as
$$\Xi_{r,n}=\sum_{i=1}^n w_i \ell(X_{i+1},\hat f_i^{(r)})\,,$$
where $(w_i:i\ge 1)$ is a non-decreasing sequence of positive weights such that
$$\lim_{i\rightarrow \infty}\frac{w_i}{\sum_{j\le i} w_j}=0\,.$$
The specific choice of $w_i=i^\xi$ is quite natural, as a constant fraction of all sample points will have a constant fraction of the maximum weight. For example, at time $i$, a $(1-2^{-1/\xi})$ proportion of sample points will have weights at least half as much as the maximum weight $i^\xi$. These sample points are also the most recent and thus more representative of the asymptotic model quality than the earlier sample points.  In contrast, the exponential weighting $w_i=e^{\xi i}$ will only assign non-trivial weights to a constant number of sample points, leading to an overly large variance in the wRV statistic.

\paragraph{Consistency of wRV.} Interestingly, weighted rolling validation enjoys the same consistency properties, regardless of the value of $\xi$.
\begin{corollary}
    \label{cor:3.4-wrv-consistency}
    Let $\Xi_{r,n,\xi}$ be the wRV statistic defined in \eqref{eq:3.4-wrv}. Then, for any $\xi\ge 0$, under the same conditions as in \Cref{thm:3.3-wrv-consistency}, we have
    $$\lim_{n\rightarrow\infty}\mathbb P(\Xi_{1,n,\xi}<\Xi_{2,n,\xi})=1\,.$$
\end{corollary}
Intuitively, the weights $i^\xi$ affect the variance and model quality gap in the same manner and will cancel out in the final signal-to-noise ratio calculation.  The proof follows exactly the same procedure as that of \Cref{thm:3.3-wrv-consistency}.

\paragraph{The effect and practical choice of $\xi$.}
Although \Cref{cor:3.4-wrv-consistency} shows that wRV can be consistent for any nonnegative value of $\xi$, it does not reflect the benefit of using a positive value of $\xi$.
In practice, the choice of $\xi$ may still impact the finite sample bias-variance trade-off,  and wRV with a positive $\xi$ generally outperforms the original RV with $\xi=0$.  In general, a small $\xi$ leads to smaller variance but also larger bias (one way to interpret the bias is the 
representativeness of the wRV statistic for the asymptotic model quality).  A theoretically rigorous understanding of the effect of $\xi$ requires a finite sample analysis of the wRV model selection error probability, which is an open problem. Here we provide a heuristic argument to demonstrate the ``quick 
response'' effect of a positive $\xi$ in a noiseless scenario, and give a 
simple rule-of-thumb choice of $\xi$.

Assume that  the expected risks have the form $\psi_{1,i}=A i^{-a}$ and $\psi_{2,i}=B i^{-b}$ for positive constants $(A,a,B,b)$ such that $0\le b <a\le 1$ and $A>B$.   Let $i^*=(A/B)^{1/(a-b)}$. Then $\psi_{1,i}>\psi_{2,i}$ if $i<i^*$ and $\psi_{1,i}<\psi_{2,i}$ if $i\ge i^*$.   We hope to have $\Xi_{1,i}<\Xi_{2,i}$ soon after $i$ exceeds $i^*$.  In the noiseless case, $\Xi_{r,i,\xi}$ reduces to its expected value, which equals (when $r=1$)
$$\sum_{j=1}^i\psi_{1,j} =\sum_{j=1}^i A j^{\xi-a}\approx  \frac{A}{\xi+1-a}i^{\xi+1-a}\,,$$
provided that $\xi-a\neq -1$.
A similar equation can be derived for $r=2$. Thus, approximately the noiseless wRV would be able to correctly favor the better model ($r=1$) when $i$ is large enough such that
$$
i\ge i(\xi)\coloneqq\left(\frac{A}{B}\right)^{\frac{1}{a-b}}\left(\frac{\xi+1-b}{\xi+1-a}\right)^{\frac{1}{a-b}} = i^* \left(1+\frac{a-b}{\xi+1-a}\right)^{\frac{1}{a-b}}\approx i^*\left(1+\frac{1}{\xi+1-a}\right)\,.
$$
The quantity $1/(\xi+1-a)$ in the above equation can be interpreted as the asymptotic \emph{delay ratio} caused by the rolling sum.  It is apparent that although the delay ratio is a monotone decreasing function of $\xi$, the improvement vanishes as $\xi$ increases.  In practice, a rule-of-thumb choice of $\xi$ is
$$\xi_{\rm RoT}\equiv 1\,,$$
which controls the asymptotic delay ratio to be at most $1$, while keeping the variance at a reasonable level.  Empirically, the model selection accuracy of wRV often observes the most significant improvement when $\xi$ increases from $0$ to $1$, and becomes relatively stable for $\xi\in[1,3]$.

\subsection{Bibliographic Notes}\label{subsec:3.5-wrv-increase-M-bibliography}
The pairwise model selection consistency of $K$-fold CV in mean regression under the squared loss is largely based on \cite{yang2007consistency}.  Our \Cref{thm:3.1-yang-consistency} is a simplified version of Theorem 1 in \cite{yang2007consistency}, which covers more general tail conditions.  As shown in \Cref{exa:3.1-simple-example}, the condition $n_{\te}\psi_{2,n_\tr}\rightarrow\infty$ captures the correct sufficient condition in this simple example.  In terms of difference, our result is stated without the ``convergence at exact rate'' condition, which is required in Yang's Theorem 1. This simplification comes from a more streamlined treatment.  Earlier investigations of model selection consistency by cross-validation include \cite{Zhang93,Shao93}.

The online model selection problem is motivated by recent developments in nonparametric stochastic approximation, such as kernel SGD \citep{kivinen2004online,dieuleveut2016nonparametric} and sieve SGD \citep{zhang2022sieve,quan2024optimal}.
The rolling validation and weighted rolling validation results are based on \cite{zhang2023online}, which contains a detailed treatment of the diverging candidate set setting, as well as numerical examples demonstrating the practical advantage of wRV over traditional CV model selection.

This section focuses on mean regression with squared loss.  A natural question is whether similar results can be obtained for more general learning problems and loss functions.  In this direction, \cite{yang2006comparing} developed CV model selection consistency in binary classification that parallels \Cref{subsec:3.1-yang07} and \cite{yang2007consistency}. 
\cite{wang2006estimation} also attempted to obtain a similar extension, requiring an accurate estimate of the conditional probability of $Y$ given $X$ used for a resampling step.  In contrast, the proofs of the consistency results in this section mainly rely on the concentration of loss functions around their mean values.  The general tools developed in \Cref{subsec:bousquet_consistency}, \Cref{subsec:bousquet_concentration}, and \Cref{subsec:concentration_ineq} could be used to provide useful results to establish similar model selection guarantees for general loss functions and online algorithms.

The model selection consistency of cross-validation has also been studied outside of the i.i.d. case.  For example, \cite{racine2000consistent} developed a variant of cross-validation for consistent time series model selection. Some early explorations of methods similar to the weighted rolling validation in the time series forecasting setting include \cite{wei1992predictive,ing2007accumulated}. 
In the network analysis literature, \cite{chen2018network,li2020network,qin2022consistent} used different sample splitting and cross-validation techniques for community-structured network model selection.
\newpage

\section{Central Limit Theorems For Cross-Validation}\label{sec:4-clt}
The focus of this section shifts back to risk estimation, where the target of CV risk estimate $\hat R_{\cv,n,K}$, defined in \eqref{eq:2.1-cv-risk-est}, is either the average risk of the leave-one-fold-out estimates defined as $\bar R_{\cv,n,K}$ in \eqref{eq:2.1-R-bar} or its population version $\mu_{n_\tr}$ defined in \eqref{eq:2.1-mu_n}.  Consistency and exponential concentration of $\hat R_{\cv,n,K}$ around these targets have been established under various stability conditions in \Cref{subsec:bousquet_consistency,subsec:bousquet_concentration,subsec:concentration_ineq}. In this section, we are interested in pushing these results one step further to establish corresponding central limit theorems.  
These results, which are of independent theoretical interest, also pave the way for new methodological developments in a variety of inference problems, including constructing model selection confidence sets. As shown in \Cref{subsec:5.2-cvc-consist}, these can be further developed into improved model selection methods that overcome the theoretical limitation of standard CV methods observed in \Cref{sec:3-model-selection}.

Starting in this section, we will use a new notation. For $i\in[n]$, let $k_i\in[K]$ be the fold index such that $i\in I_{n,k_i}$.  So we can write the cross-validation risk estimate as a sum over $i$:
$\hat R_{\cv,n,K}=n^{-1}\sum_{i=1}^n \ell(X_i,\D_{n,-k_i})$.

\subsection{CLT with Random Centering}\label{subsec:4.1-clt-random}
We first consider the ``random centering'' case, where the quantity $\bar R_{\cv,n,K}$ is treated as the target of $\hat R_{\cv,n,K}$ and hence the centering of the central limit theorem. \Cref{pro:bousquest_concentration} and \Cref{thm:2.5-loocv-concentration-sw} suggest that perturb-one stability of the loss function with rate $o(1/\sqrt{n})$ implies sub-Gaussian type concentration of $\hat R_{\cv,n,K}$ at rate $n^{-1/2}+o(1)$. In this section, we further refine the analysis to show that under essentially the same stability condition the CV risk estimate centers at $\bar R_{\cv,n,K}$ with $\sqrt{n}$-asymptotic normality.

Our starting point is an intuition that if the risk function $\ell(X_0,\D_n)$ is perturb-one stable at a rate that vanishes faster than $1/\sqrt{n}$, then the variance contributed to $\ell(X_0,\D_n)$ by the fitting part $\D_n$ will be upper bounded by $o(1)$ using the Efron--Stein inequality.
\begin{theorem}[Efron--Stein Inequality]\label{thm:4.1-Efron--Stein}
    Let $h$ be a function that maps $\D_n$ to a real number, where $\D_n=(X_i:1\le i\le n)$ consists of independent entries (not necessarily identically distributed or even defined in the same space). Then
    $$\var\left[h(\D_n)\right]\le \frac{1}{2}\sum_{i=1}^n \|\nabla_i h\|_2^2\,.$$
\end{theorem}
The proof of the Efron--Stein inequality is elementary and is based on the telescoping decomposition \eqref{eq:2.5-nabla-telescope}.  A full proof can be found in Chapter 3.1 of \cite{blm}.

As a consequence,  one can expect the summand term $\ell(X_i,\D_{n,-k_i})$ in $\hat R_{\cv,n,K}$ to behave like its de-randomized version
\begin{equation*}
\ell_{n_\tr}(X_i)\coloneqq \mathbb E\left[\ell(X_i,\D_{n,-k_i})|X_i\right]\,.\end{equation*}
The subindex $n_{\tr}$ in the notation $\ell_{n_\tr}(\cdot)$ emphasizes that this function only depends on the training sample size.  More generally, we can define, for any $n\ge 1$,
\begin{equation}\label{eq:4.1-bar-ell}
\ell_n(x) = \mathbb E\left[\ell(X_0,\D_n)|X_0=x\right]\,,~~\text{and}~~
\bar\ell_n(x) = \ell_n(x)-\mu_n\,.
\end{equation}
The following result quantifies the contribution to the total variance of the loss function $\ell(X_i,\D_{n,-k_i})$ from $\D_{n,-k_i}$, i.e., the fitting randomness.
\begin{proposition}
\label{pro:4.1-var-approx}
Assume the fitting procedure is symmetric in the sense of \Cref{asm:symmetric-f} and $\D_n$ consists of i.i.d. sample points. Let $\epsilon_n=\|\nabla_1\ell(X_0,\D_n)\|_2$ \footnote{By symmetry of $\hat f$ and the i.i.d. assumption, the particular choice of $i$ in $\nabla_i \ell(X_0,\D_n)$ does not matter as long as $i\in[n]$}. Then
$$
\var[\ell_n(X_0)]\le \var[\ell(X_0,\D_n)] \le \var[\ell_n(X_0)]+\frac{n}{2}\epsilon_n^2\,.
$$
\end{proposition}

\begin{proof}
    [Proof of \Cref{pro:4.1-var-approx}]
    Using the variance decomposition
    \begin{align*}
    &\var[\ell(X_0,\D_n)]=\var\left[\mathbb E(\ell(X_0,\D_n)|X_0)\right]+\mathbb E\left[\var(\ell(X_0,\D_n)|X_0)\right]\\
    =&\var\left[\ell_n(X_0)\right]+\mathbb E\left[\var(\ell(X_0,\D_n)|X_0)\right]\,.
    \end{align*}
    Thus it suffices to show 
    $$\mathbb E\left[\var(\ell(X_0,\D_n)|X_0)\right]\le (n/2)\|\nabla_1\ell(X_0,\D_n)\|_2^2\,.$$
    Using the Efron--Stein inequality on the conditional distribution of $\ell(X_0,\D_n)$ given $X_0$ we have
    \begin{align*}
        \var(\ell(X_0,\D_n)|X_0)\le \frac{1}{2}\sum_{i\in [n]} \mathbb E\left[(\nabla_i\ell(X_0,\D_n))^2|X_0\right]\,.
    \end{align*} 
    The desired result follows by taking expected value on both sides.
\end{proof}

\Cref{pro:4.1-var-approx} suggests that if the loss function is stable enough then the randomness in $\ell(X_0,\D_n)$ comes mostly from $X_0$ and it is reasonable to expect
\begin{equation}\label{eq:4.1-approx-lin}
\hat R_{\cv,n,K}\approx \frac{1}{n}\sum_{i=1}^n \ell_{n_\tr}(X_i)
\end{equation}
which is the average of $n$ i.i.d. random variables.  However, \eqref{eq:4.1-approx-lin} may not be accurate enough in general cases for the purpose of Gaussian approximation. We will need a slight modification on the LHS so that it focuses on the randomness contributed from the loss function evaluation, as detailed in the following theorem.

\begin{theorem}[CV-CLT, Random Centering]\label{thm:4.1-clt-rand}
Let $(X_i:i\ge 0)$ be an i.i.d. sequence of sample points in $\X$ and $\D_n=\{X_i:1\le i\le n\}$.
Let $\hat f:\bigcup_{n=1}^\infty \X^{\otimes n}\mapsto \mathcal F$ be a symmetric estimation procedure, $\ell(\cdot,\cdot)$ a loss function such that
 \begin{equation}\label{eq:4.1-nabla-l2-bound}
 \|\nabla_1 \ell(X_0,\D_n)\|_2=o(\sigma_n/\sqrt{n})\,,
 \end{equation}
 where $\sigma_n^2=\var\left[\ell_n(X_0)\right]\in(0,\infty)$ and
 \begin{equation}
     \label{eq:4.1-lindeberg}
     \lim_{n\rightarrow\infty} \mathbb E\left(\frac{\bar\ell_n(X_0)}{\sigma_n}\right)^2\mathds{1}\left(\frac{|\bar\ell_n(X_0)|}{\sigma_n}\ge \epsilon\sqrt{n}\right)=0
 \end{equation}
 for all $\epsilon>0$ and $\ell_n(X_0)$, $\bar\ell_n(X_0)$ defined in \eqref{eq:4.1-bar-ell}.
Then for any value of $K$ that evenly divides $n$ and is allowed to change with $n$, the $K$-fold cross-validated risk $\hat R_{\cv,n,K}$ satisfies
$$
\frac{\sqrt{n}}{\sigma_{n_\tr}}\left(\hat R_{\cv,n,K}-\bar R_{\cv,n,K}\right)\rightsquigarrow N(0,1)\,,
$$
where the symbol ``$\rightsquigarrow$'' means ``convergence in distribution'', and $\bar R_{\cv,n,K}$ is defined in \eqref{eq:2.1-R-bar}.
\end{theorem}

\begin{proof}[Proof of \Cref{thm:4.1-clt-rand}]
The proof proceeds by rigorously establishing a modified version of the ``approximate linearity'' of $\hat R_{\cv,n,K}$ in \eqref{eq:4.1-approx-lin}.

Let $\xi_i = \ell(X_i,\D_{n,-k_i}) - R(\D_{n,-k_i})$ and $\bar \xi_i =\mathbb E\left[\xi_i|X_i\right]=\bar\ell_{n_\tr}(X_i)$.

The rest of the proof consists of two main steps.  In the first step we show that
 \begin{equation}\label{eq:4.1-approx-lin-long}
     \frac{\sqrt{n}}{\sigma_{n_\tr}}\left(\hat R_{\cv,n,K}-\bar R_{\cv,n,K}\right)-\frac{1}{\sqrt{n}\sigma_{n_\tr}}\sum_{i=1}^n\bar\xi_i=o_P(1)\,.
 \end{equation}
In fact, by construction, the LHS of \Cref{eq:4.1-approx-lin-long} multiplied by $\sqrt{n}\sigma_{n_\tr}$ equals 
\begin{align*}
n\left(\hat R_{\cv,n,K}-\bar R_{\cv,n,K}\right)-\sum_{i=1}^n\bar\xi_i=\sum_{i=1}^n (\xi_i-\bar\xi_i)\,,
\end{align*}
whose squared $L_2$ norm is
\begin{align}\label{eq:4.1-squared-norm-rand-center-residual}
   \sum_{i=1}^n\mathbb E(\xi_i-\bar\xi_i)^2 + \sum_{1\le i\neq j\le n} \mathbb E(\xi_i-\bar\xi_i)(\xi_j-\bar \xi_j)\,. 
\end{align}
For the squared term $(\xi_i-\bar\xi_i)^2$ we have, using the same argument as in \Cref{pro:4.1-var-approx},
\begin{align*}
    &\mathbb E(\xi_i-\bar\xi_i)^2 
    = \mathbb E\left[\var(\xi_i|X_i)\right]\\
    \le & \frac{1}{2}\mathbb E\left[\sum_{j\notin I_{k_i}}\mathbb E((\nabla_j \xi_i)^2|X_i)\right]\\
    = & \frac{1}{2}\sum_{j\notin I_{k_i}}\|\nabla_j \xi_i\|_2^2\\
    \le & \frac{1}{2}\sum_{j\notin I_{k_i}}\|\nabla_j \ell(X_i,\D_{n,-k_i})\|_2^2=o(\sigma_{n_\tr}^2)\,,
\end{align*}
where the second inequality follows from the fact
$$\|\nabla_j\xi_i\|_2\le \|\nabla_j\ell(X_i,\D_{n,-k_i})\|_2$$ as implied by  \Cref{lem:4.4-nabla-xi} in \Cref{subsec:4.4-aux}.  As a result, the total contribution from the squared terms in \eqref{eq:4.1-squared-norm-rand-center-residual} is $o(n\sigma_{n_\tr}^2)$.

For the cross term $(\xi_i-\bar\xi_i)(\xi_j-\bar\xi_j)$, if $k_i=k_j$, then by conditioning on $\D_{n,-k_i}$ and using conditional independence of $X_i$ and $X_j$ given $\D_{n,-k_i}$ we have
$\mathbb E(\xi_i-\bar\xi_i)(\xi_j-\bar\xi_j)=0$. Now we only need to control the cross-term contributed by $(i,j)$ pairs such that $j\notin I_{k_i}$.  The main trick here is the ``free nabla'' lemma (\Cref{lem:4.4-free-nabla} in \Cref{subsec:4.4-aux}).  By applying \Cref{lem:4.4-free-nabla} twice we have, for $j\notin I_{k_i}$
$$
\mathbb E(\xi_i-\bar\xi_i)(\xi_j-\bar \xi_j) = \mathbb E [\nabla_j(\xi_i-\bar\xi_i)][\nabla_i(\xi_j-\bar\xi_j)]
$$
whose absolute value is $o(\sigma_{n_\tr}^2/n_\tr)$. Thus the total contribution from the cross-terms in \eqref{eq:4.1-squared-norm-rand-center-residual} is  $o(n\sigma_{n_\tr}^2)$.  This establishes \eqref{eq:4.1-approx-lin-long}.

It remains to show that 
$$\frac{1}{\sqrt{n}\sigma_{n_\tr}}\sum_{i=1}^n\bar\xi_i\rightsquigarrow N(0,1)\,,$$
which, under \eqref{eq:4.1-lindeberg}, follows from the standard Lindeberg central limit theorem for triangular arrays \citep[Theorem 3.4.5]{durrett2019probability}.
\end{proof}

\begin{remark}\label{rem:4.1-wasserstein}
The proof of \Cref{thm:4.1-clt-rand} implies that 
    $$
    d_{W,1}\left(\frac{\sqrt{n}}{\sigma_{n_\tr}}\left(\hat R_{\cv,n,K}-\bar R_{\cv,n,K}\right),~N(0,1)\right)\rightarrow 0\,,
    $$
    where $d_{W,1}$ is the Wasserstein-1 distance.
Moreover, if we assume a specific rate of convergence in \eqref{eq:4.1-nabla-l2-bound}, as well as a stronger integrability condition than \eqref{eq:4.1-lindeberg}, such as uniform boundedness of $\|[\ell_n(X_0)-\mu_n]/\sigma_n\|_3$, then we can establish the Berry–Esseen-type finite sample error bounds of normal approximation in the Wasserstein-1 and Kolmogorov–Smirnov distances.
\end{remark}

\begin{remark}\label{rem:4.1-var-lower-bound}
In most applications, it is reasonable to assume that there exists a pair of positive constants $(\underline{c},\bar c)$ such that
\begin{equation}\label{eq:4.1-var-lower}
0< \underline{c}^2\le \var[\ell_n(X_0)]\le \|\ell(X_0,\D_n)\|_2^2\le \bar c^2<\infty
\end{equation}
for all $n$.  For example, in regression problems, the lower bound corresponds to the aleatoric uncertainty that comes from the observation noise, and cannot be reduced by the prediction function. The upper bound usually corresponds to the parameter space and the total variability of the random variables.
\end{remark}

As another consequence of \Cref{pro:4.1-var-approx}, one can expect to obtain a consistent estimate of $\sigma_{n_\tr}$ using the sample variance of $\{\ell(X_i,\D_{n,-k_i}):i\in[n]\}$ as the variance is mainly contributed by the evaluation point $X_i$.
\begin{theorem}[Variance Estimation]\label{thm:4.1-var-est}
    Let 
    $$\hat\sigma_{n_\tr}^2=\frac{1}{n}\sum_{i=1}^n\left(\ell(X_i,\D_{n,-k_i})-\hat R_{\cv,n,K}\right)^2.$$
In addition to the same condition as in \Cref{thm:4.1-clt-rand}, assume that
 $[\bar\ell_n(X_0)/\sigma_n]^2$ is uniformly integrable.
Then 
    $$
    \frac{\hat\sigma_{n_\tr}^2}{\sigma_{n_\tr}^2}\stackrel{P}{\rightarrow} 1\,.
    $$
\end{theorem}
The additional uniform integrability assumption ensures that the sample variance of $\{\bar\ell_{n_\tr}(X_i):1\le i\le n\}$ is a consistent estimate of its population version. It can be satisfied under moment conditions, such as uniformly bounded $\|\bar\ell_n(X_0)/\sigma_n\|_{2+\delta}$ for any $\delta>0$.
\begin{proof}
    [Proof of \Cref{thm:4.1-var-est}]
    
    Define 
    $$\tilde\sigma_{n_\tr}^2\coloneqq \frac{1}{n}\sum_{i=1}^n [\bar\ell_{n_\tr}(X_i)]^2\,,$$
    which, under uniform integrability of $[\bar\ell_n(X_0)/\sigma_n]^2$, satisfies
    $$
    \frac{\tilde\sigma_{n_\tr}^2}{\sigma_{n_\tr}^2}\stackrel{P}{\rightarrow} 1\,.
    $$

    It suffices to show that $\hat\sigma_{n_\tr}^2-\tilde\sigma_{n_\tr}^2=o_P(\sigma_{n_\tr}^2)$.
    By construction, we have
    \begin{align*}
       | \hat \sigma_{n_\tr}^2-\tilde\sigma_{n_\tr}^2| = & \left|\frac{1}{n}\sum_{i=1}^n \left[\ell(X_i,\D_{n,-k_i}) - \hat R_{\cv,n,K}-\ell_{n_\tr}(X_i)+\mu_{n_\tr}\right]\right.\\
       &~~\left.\times \left[\ell(X_i,\D_{n,-k_i}) - \hat R_{\cv,n,K}-\ell_{n_\tr}(X_i)+\mu_{n_\tr}+2(\ell_{n_\tr}(X_i)-\mu_{n_\tr})\right]\right|\,.
    \end{align*}
    Using Jensen's inequality and the same argument as bounding $\|\xi_i-\bar\xi_i\|_2^2$ in the proof of \Cref{thm:4.1-clt-rand} we obtain
    $$ \|\ell(X_i,\D_{n,-k_i})-\ell_{n_\tr}(X_i)\|_2= o(\sigma_{n_\tr})\,.
    $$
Following the proof of \Cref{thm:2.5-loocv-concentration-sw}, we can obtain, under the stability condition \eqref{eq:4.1-nabla-l2-bound},
$$\|\nabla_i \hat R_{\cv,n,K}\|_2=o(\sigma_{n_\tr}/\sqrt{n_\tr})$$
which, by Efron--Stein, further implies that
$$\|\hat R_{\cv,n,K}-\mu_{n_\tr}\|_2=o(\sigma_{n_\tr})\,.$$

Then Cauchy–Schwartz and the triangle inequality imply that
$$
\|\tilde\sigma_{n_\tr}^2-\hat\sigma_{n_\tr}^2\|_1=o(\sigma_{n_\tr}^2)\,,
$$
which concludes the proof.
\end{proof}

\subsection{CLT with Deterministic Centering}\label{subsec:4.2-cv-clt-det}
In this subsection, we develop a central limit theorem for the CV risk estimate with the deterministic centering: $\mu_{n_\tr}=\mathbb E\hat R_{\cv,n,K}$.

The random centering CLT proved in the previous subsection provides a head start. A key intermediate step in that proof is, under the loss stability condition \eqref{eq:4.1-nabla-l2-bound} and recalling that $\bar\ell_n(x)=\ell_n(x)-\mu_n$,
$$
\hat R_{\cv,n,K} - \bar R_{\cv,n,K} = \frac{1}{n}\sum_{i=1}^n\bar\ell_{n_\tr}(X_i) +o_P(1/\sqrt{n})\,,
$$
which can be rewritten equivalently as
\begin{equation}
\label{eq:4.2-asymp-lin-with-G}
\hat R_{\cv,n,K} - \mu_{n_\tr} = \frac{1}{n}\sum_{i=1}^n\bar\ell_{n_\tr}(X_i)+\frac{1}{n} G+o_P(1/\sqrt{n})\,,
\end{equation}
where
\begin{equation}\label{eq:4.2-def-G}
G = \sum_{i=1}^n \left[R(\D_{n,-k_i})-\mu_{n_\tr}\right] = \sum_{k=1}^K n_\te \left[R(\D_{n,-k})-\mu_{n_\tr}\right]
\end{equation}
accounts for the difference between the random centering $\bar R_{\cv,n,K}$ and the deterministic centering $\mu_{n_\tr}$.

The additional challenge in order to establish a deterministic centering CLT for $\hat R_{\cv,n,K}$ is to characterize the randomness in $G$, a mean zero random variable that involves each sample point in a symmetric fashion.
Unlike the loss function, where each sample point $X_i$ is evaluated exactly once in $\ell(X_i,\D_{n,-k_i})$, the randomness in $G$ features the sum of the risks of parameters fitted from different leave-one-fold-out subsamples, and each $X_i$ contributes, in a symmetric fashion as other fitting sample points, to all but one such risks.  Thus we will proceed by considering the projection of $G$ onto each $X_i$, which is involved in the parameter fitting part in $R(\D_{n,-k_j})$ for all $j\notin I_{k_i}$. Define
\begin{equation}\label{eq:4.2-g}
g_n(x) \coloneqq n\mathbb E\left[R(\D_n)|X_1=x\right]\,,~~~~\text{and}~~~~ \bar g_n(x)=g_n(x)-n\mu_n\,.
\end{equation}
The quantity $g_n(X_i)$ is the counterpart of $\ell_n(X_i)$ when considering the contribution of $X_i$ to the \emph{fitting variability}.

By construction, it is straightforward to check that 
$$\bar g_{n_\tr}(X_i)=\mathbb E(G|X_i)\,.$$
It is reasonable to expect the linear decomposition
\begin{equation}\label{eq:4.2-approx-lin-of-G}
G\approx\sum_{i=1}^n \bar g_{n_\tr}(X_i)\,.
\end{equation}
If we can establish \eqref{eq:4.2-approx-lin-of-G} with enough level of accuracy, then one can establish a CLT with deterministic centering from \eqref{eq:4.2-asymp-lin-with-G} using standard i.i.d. sums.

To further explore the plausibility of \eqref{eq:4.2-approx-lin-of-G}, we consider the following martingale decomposition of $G$:
\begin{equation}\label{eq:4.2-first-mg-decomp}
G = \sum_{i=1}^n \mathbb E(G|\D_i) - \mathbb E(G|\D_{i-1}) = \sum_{i=1}^n \mathbb E(\nabla_i G|\D_i)
\end{equation}
with the convention $\D_0=\emptyset$.

The martingale increment $\mathbb E(\nabla_i G|\D_i)$ represents the contribution to $G$ from the $i$th sample point.  We shall also see that, at least intuitively, when the loss function is stable, the choice of ordering of $i$ in \eqref{eq:4.2-first-mg-decomp} does not matter.  Now it seems natural to expect this increment to be close to $\bar g_{n_\tr}(X_i)$. Indeed we have, using another round of martingale decomposition with respect to the $\sigma$-fields generated by $(\D_j,X_i)$,
\begin{align}
    \mathbb E(\nabla_i G|\D_i) = & \mathbb E(\nabla_i G|X_i)+\sum_{j=1}^{i-1}\left[\mathbb E(\nabla_i G|\D_j,X_i)-\mathbb E(\nabla_i G|\D_{j-1},X_i)\right]\nonumber\\
    = & \bar g_{n_\tr}(X_i)+\sum_{j=1}^{i-1}\mathbb E(\nabla_j\nabla_i G|\D_j,X_i)\nonumber\\
    \coloneqq & \bar g_{n_\tr}(X_i)+\Lambda_i\,. \label{eq:4.2-second-mg-decomp}
\end{align}

Now we arrive at the following approximate linear form for $G$:
\begin{equation}\label{eq:4.2-approx-G-with-res}
    G = \sum_{i=1}^n \bar g_{n_\tr}(X_i)+\sum_{i=1}^n \Lambda_i\,.
\end{equation}
It remains to control the residual term $\sum_{i=1}^n \Lambda_i$.  Two observations are useful.
\begin{enumerate}
    \item The sequence $(\Lambda_i:1\le i\le n)$ is a martingale increment sequence with respect to the filtration corresponding to $(\D_i:1\le i\le n)$.
    \item For each $i$, we have
    \begin{equation}
        \label{eq:4.2-Lambda_i-bound}
        \|\Lambda_i\|_2^2\le (i-1)n^2 \max_{1\le j< l\le n_\tr} \|\nabla_j\nabla_l R(\D_{n_\tr})\|_2^2\,.
    \end{equation}
\end{enumerate}
To prove \eqref{eq:4.2-Lambda_i-bound}, using the martingale decomposition of $\Lambda_i$, we have
\begin{align*}
    \|\Lambda_i\|_2^2  = & \sum_{j=1}^{i-1}\|\mathbb E(\nabla_j\nabla_i G|\D_j,X_i)\|_2^2\\
    \le & \sum_{j=1}^{i-1} \|\nabla_j\nabla_i G\|_2^2\,.
\end{align*}
But $\nabla_j\nabla_i G = \sum_{k=1}^K n_\te \nabla_j\nabla_i R(\D_{n,-k})$ and observe that $\nabla_j\nabla_i R(\D_{n,-k})\neq 0$ only if $i,j\notin I_{n,k}$.  Thus $\|\nabla_j\nabla_i G\|_2\le n\max_{1\le j<l\le n_\tr}\|\nabla_j\nabla_l R(\D_{n_\tr})\|_2$.

The above derivation essentially proves the following central limit theorem for CV risk estimates with deterministic centering.
\begin{theorem}[CV-CLT, Deterministic Centering]
    \label{thm:4.2-cv-clt-det}
    Let $(X_i:0\le i<\infty)$ be a sequence of i.i.d. sample points in $\X$, $\hat f:\bigcup_{n\ge 1}\X^{\otimes n}\mapsto \mathcal F$ a symmetric estimation procedure, and $\ell:\X\times \mathcal F\mapsto \mathbb R$ a loss function. Let $K$ be an integer that evenly divides $n$ and is allowed to change with $n$.  Suppose 
    $$\tau_{n}^2\coloneqq \var(\ell_n(X_0)+g_{n}(X_0))>0$$
    for all $n$
    and the following stability condition holds
    \begin{enumerate}
    \item For any $1\le i\le n$, $\|\nabla_i \ell(X_0,\D_n)\|_2=o(\tau_n/\sqrt{n})$;
    \item For any $1\le j < i \le n$, $\|\nabla_j\nabla_i R(\D_n)\|_2=o(\tau_n n^{-3/2})$.
    \item For any $\epsilon>0$, \begin{equation}
        \label{eq:4.2-lindeberg-det}
        \lim_{n\rightarrow \infty}\mathbb E\left(\frac{\bar\ell_n(X_0)+\bar g_n(X_0)}{\tau_n}\right)^2\mathds{1}\left(\left|\frac{\bar\ell_n(X_0)+\bar g_n(X_0)}{\tau_n}\right|\ge \epsilon\sqrt{n}\right)=0\,.
   \end{equation}
    \end{enumerate}
    Then $$
    \frac{\sqrt{n}}{\tau_{n_\tr}}\left(\hat R_{\cv,n,K}-\mu_{n_\tr}\right)\rightsquigarrow N(0,1)\,.
    $$
\end{theorem}

\begin{proof}
    [Proof of \Cref{thm:4.2-cv-clt-det}]
 According to the proof of \Cref{thm:4.1-clt-rand},  the first stability condition establishes \eqref{eq:4.2-asymp-lin-with-G}. Now the second stability condition ensures that $n^{-1/2}\sum_i \Lambda_i=o_P(\tau_{n_\tr})$ according to \eqref{eq:4.2-Lambda_i-bound}.  The result then follows from the standard Lindeberg CLT for the triangular array $\{\bar\ell_{n_\tr}(X_i)+\bar g_{n_\tr}(X_i):1\le i\le n\}_{n=1}^\infty$.
\end{proof}
The asymptotic variance $\tau_{n_\tr}^2$ is the total variance of the sum of two parts: $\ell_{n_\tr}(X_i)$ and $g_{n_\tr}(X_i)$, which corresponds to the evaluation and estimation randomness contributed by the sample point $X_i$, respectively. In contrast, the random centering CLT \Cref{thm:4.1-clt-rand} only captures the evaluation part involving $\ell_{n_\tr}(X_i)$, leaving the randomness of estimation to the random centering $\bar R_{\cv,n,K}$.  While it is relatively straightforward to assume the boundedness of $\var(\ell_{n_\tr}(X_1))$, the behavior of $\var(\ell_{n_\tr}(X_1)+g_{n_\tr}(X_1))$ is harder to track in general.  Intuitively, $\ell_{n_\tr}(X_1)$ and $g_{n_\tr}(X_1)$ should have a weak positive correlation, because if $X_1$ is a ``good sample point'' then it should reduce the loss function as well as the risk of the estimate obtained using it as one of the training sample points.  This makes a positive constant lower bound of $\var(\ell_{n_\tr}(X_1)+g_{n_\tr}(X_1))$ more or less reasonable.  In order to assume an upper bound on $\var(g_{n_\tr}(X_1))$, it helps to consider the variability of $n_\tr\nabla_1 R(\D_{n_\tr})$, because there are $n_\tr$ terms in $G$ that involves $X_1$.  An upper bound on $\var[g_{n_\tr}(X_1)]$ can be obtained through the first-order risk stability $\nabla_1 R(\D_{n_\tr})$:
\begin{align}
    \label{eq:4.2-var-G-risk-stab}
    \var(g_{n_\tr}(X_1))= &\var\left[n_\tr \mathbb E(R(\D_{n_\tr})|X_1)\right]\nonumber\\
     = &\frac{1}{2}\var\left[n_\tr \nabla_1 \mathbb E(R(\D_{n_\tr})|X_1)\right]\nonumber\\
     \le &\frac{1}{2} n_\tr^2\|\nabla_1 R(\D_{n_\tr})\|_2^2\,.
\end{align}
In order for the upper bound in \eqref{eq:4.2-var-G-risk-stab} to be $O(1)$, it would require a risk stability rate of $O_P(1/n_\tr)$, which seems substantially more stringent than the corresponding rate assumed for the loss stability in \eqref{eq:4.1-nabla-l2-bound} as well as Part 1 of the stability assumption in \Cref{thm:4.2-cv-clt-det}.   We argue that in many machine learning contexts, the risk can have much better stability than the loss function itself.

As a first example, consider $X_i\stackrel{i.i.d.}{\sim} N(0,1)$ and $\ell(x,\D_n)=(x-\hat\mu_n)^2$ where $\hat\mu_n$ is the empirical mean of $\D_n$.  It is straightforward to check that $\|\nabla_1\ell(X_0,\D_n)\|_2\asymp n^{-1}$ but $\|\nabla_1 R(\D_n)\|_2\asymp n^{-3/2}$.

More generally, suppose that $\mathcal F$ is a Hilbert space, and we have the first-order approximation
$$
\nabla_i R(\D_n) \approx \left\langle\left(\frac{d R}{d f}\bigg|_{f=\hat f(\D_n)}\right)\,,~ \left(\nabla_i \hat f\right)\right\rangle\,.
$$
Usually $\nabla_i\hat f$ corresponds to the stability of the loss function $\ell(X_0;\hat f(\D_n))$ if $\ell(X_0,f)$ is smooth in $f$, which could be as small as $O(n^{-1})$.  For risk minimization estimators, it is typically the case that $\frac{d R(f)}{d f}\big|_{f=\hat f(\D_n)}\approx 0$.  Thus for such approximate risk minimization estimators it is reasonable to expect $\nabla_i R(\D_n)$ to be much smaller than $\nabla_i\ell(X_0,\D_n)$, making the requirement of $|\nabla_i R(\D_n)|\ll n^{-1}$ plausible.   This intuition also helps explain the condition $\|\nabla_j\nabla_i R(\D_n)\|_2=o(n^{-3/2})$ required in \Cref{thm:4.2-cv-clt-det}.  We will resume this discussion after presenting the variance estimation result below.  Further discussion and examples of stable algorithms are presented in \Cref{sec:6-stability}, especially in \Cref{subsec:6.3-stab-diff} for the risk stability.

\paragraph{Variance Estimation with Deterministic Centering}
It is unclear how to construct a consistent estimate of the asymptotic variance of $\hat R_{\cv,n,K}$ without auxiliary sample points that are not in $\D_n$.  Now suppose we do have access to some auxiliary sample points, independently drawn from the same distribution.  The question we seek to answer here is how many such auxiliary sample points would be sufficient to guarantee a consistent estimate of the asymptotic variance of the CV risk estimate with deterministic centering. 

If we have $\omega(n)$ auxiliary sample points, then we can generate $\omega(1)$ i.i.d. copies of $\hat R_{\cv,n,K}$ and then obtain a consistent variance estimate from these i.i.d. copies.  Now we show that one can also obtain a consistent variance estimate using only $\omega(1)$ auxiliary sample points.

We begin with the proof of \Cref{thm:4.2-cv-clt-det}, which is based on the following  more refined version of \eqref{eq:4.2-asymp-lin-with-G} and \eqref{eq:4.2-approx-lin-of-G} with an explicit expression of the residual,
$$
\hat R_{\cv,n,K} - \mu_{n_\tr} = \frac{1}{n}\sum_{i=1}^n\left[\bar\ell_{n_\tr}(X_i)+\bar g_{n_\tr}(X_i)\right]+ \frac{1}{n}\sum_{i=1}^n(\xi_i-\bar\xi_i)+\frac{1}{n}\sum_{i=1}^n \Lambda_i\,,
$$
where $\xi_i=\ell(X_i,\D_{n,-k_i})-R(\D_{n,-k_i})$, $\bar\xi_i=\mathbb E(\xi_i|X_i)=\ell_{n_\tr}(X_i)-\mu_{n_\tr}$, and $\Lambda_i=\mathbb E(G|\D_i)-\mathbb E(G|\D_{i-1})$.

To study the variance of $\hat R_{\cv,n,K}$, it is natural to consider $\nabla_i \hat R_{\cv,n,K}$.
In particular we will consider $i=n$, and, assuming that $k_n=K$,
\begin{equation}\label{eq:4.2-nabla_n-R}
    n\nabla_n\hat R_{\cv,n,K} =  \nabla_n (\ell_{n_\tr}(X_n)+g_{n_\tr}(X_n))+ \sum_{i\notin I_K}\nabla_n(\xi_i-\bar\xi_i) + \nabla_n(\xi_n-\bar\xi_n) + \nabla_n \Lambda_n\,.
\end{equation}
The convenience of choosing $i=n$ in $\nabla_i$ is reflected in the last term of the equation above, where all $(\Lambda_i:i<n)$ disappear under the $\nabla_n$ operator. It is important to note that \eqref{eq:4.2-nabla_n-R} also holds for arbitrary index $i\in[n]$, with the last term becoming $\nabla_i \Lambda_{i,n}$, where $\Lambda_{i,n}$ is the counterpart of $\Lambda_n$ when one swaps index $i$ and index $n$ in the martingale decompositions in \eqref{eq:4.2-first-mg-decomp} and \eqref{eq:4.2-second-mg-decomp}.

The first term in the RHS of \eqref{eq:4.2-nabla_n-R} gives us the desired i.i.d. random variable with the desired asymptotic variance. Thus, if we can show that the other terms in the RHS of \eqref{eq:4.2-nabla_n-R} are all sufficiently small with vanishing $L_2$ norms, then one can use the empirical variance of $n\nabla_i \hat R_{\cv,n,K}$ with $\omega(1)$ many distinct $i$'s, which would only require $\omega(1)$ auxiliary sample points.

Among the residual terms in the RHS of \eqref{eq:4.2-nabla_n-R}, the most nontrivial one to control is the second term
$\nabla_n\sum_{i\notin I_K}(\xi_i-\bar\xi_i)$.  In the proof of \Cref{thm:4.1-clt-rand} we have shown 
$$\left\|\sum_{i\notin I_K}(\xi_i-\bar\xi_i)\right\|_2=o(\sqrt{n}\sigma_{n_\tr})\,,$$
thus we only need the $\nabla_n$ operator to cancel out the additional $\sqrt{n}$ factor in this upper bound. While this is intuitively plausible, we will need some additional assumptions for a rigorous result.  We consider the following.
\begin{assumption}
    \label{asm:4.2-second-order-stab}
    \begin{enumerate}
        \item [(a)]  $\|\nabla_1\ell(X_0,\D_n)\|_2=o((n K)^{-1/2}\tau_n)$.
        \item [(b)]  The loss function satisfies $\|\nabla_1\nabla_2\ell(X_0,\D_n)\|_2=o(n^{-1}\tau_n)$.
    \end{enumerate}
\end{assumption}
The two parts of \Cref{asm:4.2-second-order-stab} represent different possible strengthening of the stability condition \eqref{eq:4.1-nabla-l2-bound}.  Part (a) is easy to hold when the number of folds $K$ is a constant or grows with $n$ at a slow rate.  When $K$ is a constant, Part (a) of \Cref{asm:4.2-second-order-stab} is equivalent to the original stability condition \eqref{eq:4.1-nabla-l2-bound}.  On the other hand, part (b) of \Cref{asm:4.2-second-order-stab} requires a ``second-order stability''.  Intuitively, if applying the perturb-one difference operator $\nabla_i$ reduces the magnitude of $\ell(X_0,\D_n)$ by a factor of $o(1/\sqrt{n})$, then applying two such operators will reduce the magnitude by a factor of $o(1/n)$.  The advantage of Part (b) is that it can potentially cover the case of large values of $K$, such as $K=n$.  But such an advantage comes at the cost of having to verify the second-order stability condition.  We will discuss examples of second-order stability in \Cref{subsec:6.2-stab-sgd-2}.

Formally, assume we have auxiliary sample points $X_1',...,X_{n_\aux}'$ independent of $\D_n$ with sample size $n_\aux=\omega(1)$.  Consider the following pseudo-estimate of the asymptotic variance of the CV risk with deterministic centering.
\begin{equation}\label{eq:4.2-var-est-determ}
    \hat \tau_{n_\tr}^2 = \frac{1}{2n_\aux}\sum_{i=1}^{n_\aux} (n\nabla_i \hat R_{\cv,n,K})^2\,.
\end{equation}

\begin{theorem}
    \label{thm:4.2-var-est-determ}
    Under the assumptions of \Cref{thm:4.2-cv-clt-det}, assuming in addition that either part (a) or part (b) of \Cref{asm:4.2-second-order-stab} holds, if the auxiliary sample size $n_\aux$ satisfies
    \begin{equation}\label{eq:4.2-aux-sample-size}
    \frac{\|\bar\ell_{n_\tr}(X_1)+\bar g_{n_\tr}(X_1)\|_4^4}{n_\aux \tau_{n_\tr}^4}\rightarrow 0\,,
    \end{equation} we have
    $$
    \frac{\hat\tau_{n_\tr}^2}{\tau_{n_\tr}^2}\stackrel{P}{\rightarrow} 1\,.
    $$
\end{theorem}
If $\|\bar\ell_{n_\tr}(X_1)+\bar g_{n_\tr}(X_1)\|_4 =  O(\tau_{n_\tr})$, then \Cref{thm:4.2-var-est-determ} implies consistent estimation with $n_\aux=\omega(1)$.   In practice, we can hold out $n_\aux=\sqrt{n}$ for variance estimation and conduct cross-validation on the remaining sample points.

\begin{proof}
    [Proof of \Cref{thm:4.2-var-est-determ}]
Using \eqref{eq:4.2-nabla_n-R}, we have
$$n\nabla_i \hat R_{\cv,n,K}=\nabla_i[\ell_{n_\tr}(X_i)+g_{n_\tr}(X_i)]+\Delta_i$$
where $\Delta_i$ equals the sum of the last three terms in the RHS of \eqref{eq:4.2-nabla_n-R}\footnote{Although \eqref{eq:4.2-nabla_n-R} is derived for $i=n$, the validity of this decomposition for all $i\in[n]$  is due to the symmetry, as discussed in the text after \eqref{eq:4.2-nabla_n-R}.}.
Let $W_i=\nabla_i[\ell_{n_\tr}(X_i)+g_{n_\tr}(X_i)]$. Then
$$\frac{\hat\tau_{n_\tr}^2}{\tau_{n_\tr}^2}-1=\frac{\frac{1}{n_\aux}\sum_{i=1}^{n_\aux} (W_i^2/2-\tau_{n_\tr}^2)}{\tau_{n_\tr}^2}+\frac{1}{n_\aux}\sum_{i=1}^{n_\aux}\frac{W_i \Delta_i}{\tau_{n_\tr}^2}+\frac{1}{n_\aux}\sum_{i=1}^{n_\aux}\frac{\Delta_i^2}{2\tau_{n_\tr}^2}\,.$$
The second and third terms on the RHS of the above equation has $L_1$ norm bounded by $o(1)$ because $\|W_i/\tau_{n_\tr}\|_2 = \sqrt{2}$ and $\|\Delta_i/\tau_{n_\tr}\|_2=o(1)$ uniformly over $i$ according to \Cref{lem:4.4-useful-fact-loss-stab}.  Thus it suffices to show the first term is $o_P(1)$. In fact,
the first term has mean $0$ and variance
\begin{align*}
\frac{\var(W_1^2)}{4n_\aux\tau_{n_\tr}^4}\lesssim & \frac{\|\bar\ell_{n_\tr}(X_1)+\bar g_{n_\tr}(X_1)\|_4^4}{n_\aux\tau_{n_\tr}^4}=o(1)\,.\qedhere
\end{align*}
\end{proof}

\paragraph{Revisiting the Risk Stability Conditions}
The variance estimation argument above also provides further insights to the asymptotic variance $\tau_{n_\tr}^2$ and the risk stability condition in \Cref{thm:4.2-cv-clt-det}.  To begin with, we repeat the derivation of \eqref{eq:4.2-first-mg-decomp}, \eqref{eq:4.2-second-mg-decomp}, \eqref{eq:4.2-approx-G-with-res}, \eqref{eq:4.2-Lambda_i-bound}, and \eqref{eq:4.2-nabla_n-R} for the risk $R(\D_n)$. In particular, \eqref{eq:4.2-approx-G-with-res} becomes
\begin{equation*}
    R(\D_{n_\tr})-\mu_{n_\tr} = \sum_{i=1}^n\left[ \mathbb E(R(\D_{n_\tr})|X_i)-\mu\right]+\sum_{i=1}^n \tilde \Lambda_i
\end{equation*}
where $\tilde\Lambda_i$ is the counterpart of $\Lambda_i$ in \eqref{eq:4.2-approx-G-with-res}.
Also \eqref{eq:4.2-nabla_n-R} becomes (suppose we put $i=1$ in the final position in the first martingale decomposition \eqref{eq:4.2-first-mg-decomp})
\begin{equation}\label{eq:4.2-nabla_1-R=g+lambda}
    \nabla_1 R(\D_{n_\tr}) = \nabla_1 \mathbb E(R(\D_{n_\tr})|X_1)+\nabla_1\tilde\Lambda_1\,,
\end{equation}
where $\|\nabla_1\tilde\Lambda_1\|_2\le 2\sqrt{n_\tr}\|\nabla_1\nabla_2 R(\D_{n_\tr})\|_2$ by a similar argument used in \eqref{eq:4.2-Lambda_i-bound}.

Therefore, we have the following result on the relationship between the variances of $\nabla_1 R(\D_{n_\tr})$ and $g_{n_\tr}(X_1)$, which provides the reverse direction of \eqref{eq:4.2-var-G-risk-stab}.
\begin{lemma}
    \label{lem:4.2-nabla_1-R-var}
    If $\var(R(\D_n))>0$ and \begin{equation}
        \label{eq:4.2-2nabla-1nabla}
        \sqrt{n}\|\nabla_1\nabla_2 R(\D_n)\|_2 = o(\|\nabla_1 R(\D_n)\|_2)\,,
    \end{equation}
    then 
    $$
    \frac{\var(g_n(X_1))}{n^2\var(\nabla_1 R(\D_n) )}\rightarrow \frac{1}{2}\,.
    $$
\end{lemma}
\begin{proof}
    [Proof of \Cref{lem:4.2-nabla_1-R-var}]
    Equation \eqref{eq:4.2-var-G-risk-stab} shows that the ratio is upper bounded by $1/2$.  On the other hand, recall that $g_n(X_1)=n \mathbb E[R(\D_n)|X_1]$ and \eqref{eq:4.2-nabla_1-R=g+lambda} implies that
    $\|\nabla_1 R(\D_n)\|_2\le \sqrt{2}n^{-1}\|\bar g_n(X_1)\|_2+o(\|\nabla_1 R(\D_n)\|_2)$, which provides the other direction.
\end{proof}

Putting things together we can have the following more refined view of \Cref{thm:4.2-cv-clt-det} as well as its relationship with \Cref{thm:4.1-clt-rand}.
\begin{proposition}
    [CV-CLT, Consolidated Version]\label{pro:4.2-cv-clt-consolidate}
Assuming an i.i.d. sequence of $X_i:i\ge 0$ and a symmetric loss function $\ell(\cdot,\cdot)$.
\begin{enumerate}
    \item If $n\|\nabla_1 R(\D_{n})\|_2=o(\sigma_{n})$, then
    $$\frac{\tau_{n_\tr}^2}{\sigma_{n_\tr}^2}\stackrel{P}{\rightarrow}1$$
    and \eqref{eq:4.1-nabla-l2-bound}, \eqref{eq:4.1-lindeberg} imply that
    $$
   \frac{\sqrt{n}}{\sigma_{n_\tr}}(\hat R_{\cv,n,K}-\mu_{n_\tr})\rightsquigarrow N(0,1)\,. 
    $$
    \item If $n\|\nabla_1 R(\D_n)\|_2=\Omega(\sigma_n)$ and $\tau_n=\Omega(\sigma_n)$, then
    $$
   \frac{\sqrt{n}}{\tau_{n_\tr}}(\hat R_{\cv,n,K}-\mu_{n_\tr})\rightsquigarrow N(0,1)\,.
    $$
    provided that $\sqrt{n}\|\nabla_i\ell(X_0,\D_n)\|_2=o(\tau_n)$ and that \eqref{eq:4.2-2nabla-1nabla}, \eqref{eq:4.2-lindeberg-det} hold.
\end{enumerate}
\end{proposition}

\begin{proof}
    [Proof of \Cref{pro:4.2-cv-clt-consolidate}]
    For the first part, it suffices to show that $\|\bar R_{\cv,n,K}-\mu_{n_\tr}\|_2=o(\sigma_{n_\tr}/\sqrt{n})$, which is in turn implied by
    $\|R(\D_{n_\tr})-\mu_{n_\tr}\|_2=o(\sigma_{n_\tr}/\sqrt{n})$ using symmetry among the folds and triangle inequality.  The condition $n\|\nabla_1 R(\D_{n})\|_2=o(\sigma_{n})$, combined with Efron--Stein, guarantees $\|R(\D_{n_\tr})-\mu_{n_\tr}\|_2=o(\sigma_{n_\tr}/\sqrt{n})$.

    For the second part, it suffices to show that the conditions in \Cref{thm:4.2-cv-clt-det} are satisfied.  In particular, we only need to verify the second-order risk stability in \Cref{thm:4.2-cv-clt-det}. First, \Cref{lem:4.2-nabla_1-R-var} ensures that $\|g_{n_\tr}(X_0)\|_2\asymp n_\tr\|\nabla_1 R(\D_{n_\tr})\|_2$. The additional assumptions $n\|\nabla_1 R(\D_n)\|_2=\Omega(\sigma_n)$ and $\tau_n=\Omega(\sigma_n)$ ensure that $\tau_{n_\tr}\asymp n_\tr\|\nabla_1 R(\D_{n_\tr})\|_2$. Thus the second-order risk stability condition (part 2) of \Cref{thm:4.2-cv-clt-det} is satisfied because by \eqref{eq:4.2-2nabla-1nabla} we have $\|\nabla_1\nabla_2 R(\D_n)\|_2\le n^{-1/2}\|\nabla_1 R(\D_n)\|_2=o(n^{-3/2}\tau_n)$.
\end{proof}

\Cref{pro:4.2-cv-clt-consolidate} provides a more transparent view of the role played by the risk stability quantities $\|\nabla_1 R(\D_n)\|_2$ and $\|\nabla_1\nabla_2 R(\D_n)\|_2$ in the CLT with deterministic centering.  In the first part of \Cref{pro:4.2-cv-clt-consolidate}, the condition $n\|\nabla_1 R(\D_n)\|_2=o(\sigma_n)$ indicates that the fitting randomness in $nR(\D_n)$ contributed from a single sample point $X_1$ has much smaller variability than the evaluation $\ell_n(X_1)$.  In this case the second-order risk stability $\|\nabla_1\nabla_2 R(\D_n)\|_2$ is irrelevant and the deterministic centering essentially behaves like the random centering.

In the second part of \Cref{pro:4.2-cv-clt-consolidate}, the condition $n\|\nabla_1 R(\D_n)\|_2=\Omega(\sigma_n)$ means that the variability in the fitting is at least as large as that in the evaluation.  The condition $\tau_n=\Omega(\sigma_n)$ ensures that these two sources of variability do not cancel each other when we add up $\ell_n(X_1)$ and $g_n(X_1)$. As mentioned in the discussion after \Cref{thm:4.2-cv-clt-det}, it is often plausible to assume a nonnegative correlation between $\ell_n(X_1)$ and $g_n(X_1)$, and hence no cancellation of variance.  In this case, the key condition on the second-order risk stability is that $\|\nabla_1\nabla_2 R(\D_n)\|_2$ is much smaller than $n^{-1/2}\|\nabla_1 R(D_n)\|_2$.  So this is analogous to the original loss stability condition \eqref{eq:4.1-nabla-l2-bound} in the sense that we obtain an $o(n^{-1/2})$ factor from applying one more $\nabla$-operator.  Moreover, if $n\|\nabla_1 R(\D_n)\|_2\gg \sigma_n$, the requirement of $\sqrt{n}\|\nabla_1 \ell(X_0,\D_n)\|_2=o(\tau_n)$ is also weaker than \eqref{eq:4.1-nabla-l2-bound} because the fitting variability dominates and we can have a higher tolerance of the instability.

\subsection{High-Dimensional Gaussian Comparison}\label{subsec:4.3-hd-clt}
The central limit theorems developed in \Cref{subsec:4.1-clt-random} and \Cref{subsec:4.2-cv-clt-det} are concerned with a single estimator $\hat f$.  In the context of model selection and model comparison, it is necessary to study the simultaneous fluctuation of the CV risk estimates of multiple estimators.  While the extension of the central limit theorem to multivariate cases with a finite and fixed number of estimators is straightforward, in this subsection we consider the situation with a potentially large number of estimators.

To begin with, we first introduce a high-dimensional Gaussian comparison result for sums of independent random vectors, initiated by \cite{chernozhukov2013gaussian}.
A typical form of high-dimensional Gaussian comparison is to compare the distribution of the maximum of a random vector to that of a Gaussian vector with matching mean and variance.
For a vector $z=(z_r:1\le r\le m)\in\mathbb R^m$, let $\max z=\max_{r\in [m]}z_r$ be the value of the maximum entry of $z$.
The following result is a simplified version of Theorem 2.5 of \cite{chernozhukov2022improved} adapted for our setting.
\begin{theorem}[Gaussian Approximation, Simplified i.i.d. Case]\label{thm:4.3-cckk}
    Let $(\mathbf Z_i:i\in[n])$ be i.i.d. centered random vectors in $\mathbb R^m$ such that each coordinate of $\mathbf Z_1$ has unit variance.
    Let $\mathbf W_n=n^{-1/2}\sum_{i=1}^n \mathbf Z_i$ and $\mathbf{Y}_n\sim N(0,\mathbb E (\mathbf Z_1 \mathbf Z_1^T))$.

Let $\|\max|\mathbf Z_1|\|_4= B_n$ and assume
$\max_{1\le r\le m}\|Z_{1,r}\|_4\le \kappa$ for a positive constant $\kappa$ not depending on $(n,m)$, then for any $t\in\mathbb R$
\begin{equation}\label{eq:4.3-hd-gaussian}
\left|\mathbb P(\max \mathbf W_n \le t)-\mathbb P(\max \mathbf Y_n\le t)\right|\le C_{\kappa}\frac{B_n(\log m)^{5/4}}{n^{1/4}}\,,
\end{equation}
for a constant $C_\kappa$ depending on $\kappa$ only.
\end{theorem}

Gaussian comparison results like this provide probability-theoretic foundations for high-dimensional inference tasks, such as construction of simultaneous confidence intervals for high dimensional mean vectors.

When studying the joint randomness of the CV risk estimates of $m$ estimators, it is more convenient to use vector notation. Let $\boldsymbol{\ell}$ be the vector version of $(\ell^{(r)}:1\le r\le m)$. Define $\boldsymbol{\ell}_n$, $\bar{\boldsymbol{\ell}}_n$, $\hat{\mathbf R}_{\cv,n,K}$, and $\bar{\mathbf R}_{\cv,n,K}$ accordingly.
Let  $S_n$ be the $m\times m$ diagonal matrix with the $r$th diagonal entry being $\sigma_n^{(r)}\coloneqq \sqrt{\var(\ell_n^{(r)}(X_0))}$, and $\Gamma_n$ the correlation matrix of $\boldsymbol{\ell}_n(X_0)$:
\begin{equation}
    \label{eq:4.3-cov-mat-rand}
    \Gamma_n(r,s) = 
        {\rm Corr}\left[\ell_n^{(r)}(X_0),~\ell_n^{(s)}(X_0)\right]\,.
\end{equation}
Our goal is to characterize the simultaneous fluctuation of $\hat{\mathbf R}_{\cv,n,K}$ as an $m$-dimensional random vector.
\Cref{thm:4.1-clt-rand} implies that the $m$-dimensional vector $\sqrt{n}S_{n_\tr}^{-1}(\hat{\mathbf R}_{\cv,n,K}-\bar{\mathbf R}_{\cv,n,K})$
is entry-wise asymptotically standard normal, provided that, among other regularity conditions, the loss stability condition \eqref{eq:4.1-nabla-l2-bound} holds for each estimator $\hat f^{(r)}$.  Therefore, it is natural to consider the corresponding multivariate Gaussian vector $\mathbf Y_n\sim N(0,\Gamma_n)$ as the target in the high-dimensional Gaussian comparison. 

More specifically:
Does \eqref{eq:4.3-hd-gaussian} still hold when $\mathbf W_n$ is replaced by 
$\sqrt{n}S_n^{-1}(\hat{\mathbf R}_{\cv,n,K}-\bar{\mathbf R}_{\cv,n,K})$ and $\mathbf Y_n\sim N(0,\Gamma_n)$?

The following result provides a positive answer to this question under a strengthened version of the stability condition \eqref{eq:4.1-nabla-l2-bound}.
\begin{theorem}
    \label{thm:4.3-hd-cv-clt-rand}
    Let $(X_i:i\ge 0)$ be a sequence of i.i.d. sample points, and $(\hat f^{(r)}:r\in[m])$ a collection of symmetric estimators. Assume
    \begin{enumerate}
        \item there exists a positive constant $\alpha$ and a sequence $\epsilon_n$ such that, for all $r\in[m]$ and all $n$,
    \begin{equation}
        \label{eq:4.3-nabla_ell_sw}
        \frac{\sqrt{nK}}{\sigma_n^{(r)}}\nabla_1 \ell^{(r)}(X_0,\D_n) \text{ is } (\epsilon_n,\alpha)\text{-SW}\,,
    \end{equation}
 \item and there exists a finite positive constant $\kappa$ such that for all $r$ and $n$
    \begin{equation}
        \label{eq:4.3-4th-moment}
       \frac{\bar\ell_n^{(r)}(X_0)}{\sigma_n^{(r)}} ~~\text{is }(\kappa,\alpha)\text{-SW}\,.
    \end{equation}
    \end{enumerate}
If the number of models $m=m_n$ satisfies
\begin{equation}
    \label{eq:4.3-m-bound}
    \frac{(\log m)^{4\alpha+5}}{n} = o(1)\,,~~\text{and}~~(\log m)^{\alpha+3/2}\epsilon_{n_\tr} = o(1),
\end{equation}
    then we have, for $\mathbf Y_n\sim N(0,\Gamma_{n_\tr})$ with $\Gamma_{n_\tr}$ defined in \eqref{eq:4.3-cov-mat-rand},
    \begin{equation}
        \label{eq:4.3-hd-cv-gaussian-comp}
       \sup_{t\in \mathbb R}\left|\mathbb P\left(\max \sqrt{n} S_{n_\tr}^{-1}(\hat{\mathbf R}_{\cv,n,K}-\bar{\mathbf R}_{\cv,n,K})\le t\right)-\mathbb P\left(\max \mathbf Y_{n_\tr}\le t\right)\right|=o(1)\,.
    \end{equation}
\end{theorem}

In \Cref{thm:4.3-hd-cv-clt-rand}, we allow the number of models $m=m_n$ to vary with $n$.  Equations \eqref{eq:4.3-nabla_ell_sw} and \eqref{eq:4.3-4th-moment} are counterparts of \eqref{eq:4.1-nabla-l2-bound} and \eqref{eq:4.1-lindeberg} respectively.  The sub-Weibull tail enables a simultaneous upper bound on many residual terms.  The factor of $\sqrt{K}$ on the LHS of \eqref{eq:4.3-nabla_ell_sw} is necessary here because the ``free nabla trick'' used in the proof of \Cref{thm:4.1-clt-rand} no longer works for general $L_q$ norms and we can only use Rio's inequality within each test sample fold and resort to triangle inequality to combine the folds.  
Condition \eqref{eq:4.3-nabla_ell_sw} is less stringent for smaller values of $K$, and becomes comparable to \eqref{eq:4.1-nabla-l2-bound} if $K=O(1)$.
In practice, $K$-fold CV with a constant $K$, such as $K=5$ and $K=10$, is perhaps the most common version.

\begin{remark}
    \label{rem:4.3-sign-flip}
    \Cref{thm:4.3-hd-cv-clt-rand} is stated for the maximum entry of the marginally standardized error vector, which corresponds to one-sided simultaneous confidence bound.  The same result can be easily extended to a two-sided version in that \eqref{eq:4.3-hd-cv-gaussian-comp} holds true, under the same conditions, for the corresponding vectors of absolute values $|\hat{\mathbf R}_{\cv,n,K}-\bar{\mathbf R}_{\cv,n,K}|$ and $|\mathbf Y_n|$. This can be achieved by considering $m$ more loss functions $\ell^{(r+m)}=-\ell^{(r)}$ for $r\in[m]$ and applying \Cref{thm:4.3-hd-cv-clt-rand} to these $2m$ loss functions.
\end{remark}

\begin{proof}
    [Proof of \Cref{thm:4.3-hd-cv-clt-rand}]
    
The proof is based on a more refined version of the asymptotic linear representation \eqref{eq:4.1-approx-lin-long} using the sub-Weibull stability condition.
Let 
$$\boldsymbol{\zeta}_n=\sqrt{n}S_{n_\tr}^{-1}(\hat{\mathbf R}_{\cv,n,K}-\bar{\mathbf R}_{\cv,n,K}) - \frac{1}{\sqrt{n}}S_{n_\tr}^{-1}\sum_{i=1}^n \bar{\boldsymbol{\ell}}_{n_\tr}(X_i)\,.$$
By \Cref{lem:4.4-rand-center-residual-bound}, each $\zeta_{n,r}$, the $r$th coordinate of $\boldsymbol{\zeta}_n$, is $(2\epsilon_{n_\tr},\alpha+1)$-SW.

Consider i.i.d. random vectors $\mathbf Z_i = S_{n_\tr}^{-1}\bar{\boldsymbol{\ell}}_{n_\tr}(X_i)$. Condition \eqref{eq:4.3-4th-moment}
and \Cref{lem:2.5-sw-maximal} imply that $\max |\mathbf Z_1|$ is $(c_\alpha (\log m)^\alpha\kappa,\alpha)$-SW and hence
$\|\max |\mathbf Z_1|\|_4 \le c_\alpha'(\log m)^\alpha\kappa$ for another constant $c_\alpha'$.  On the other hand, by \eqref{eq:4.3-4th-moment} $\|\bar\ell_{n_\tr}^{(r)}/\sigma_{n_\tr}^{(r)}\|_4\le 4^\alpha\kappa$ for each $r\in[m]$.
Now let $\mathbf W_{n_\tr}=\frac{1}{\sqrt{n}}S_{n_\tr}^{-1}\sum_{i=1}^n\bar{\boldsymbol{\ell}}_{n_\tr}(X_i)$. \Cref{thm:4.3-cckk} implies that, for any $t\in\mathbb R$
\begin{equation}
    \label{eq:4.3-cckk-bar-ell}
    |\mathbb P(\max \mathbf W_{n_\tr}\le t)-\mathbb P(\max\mathbf Y_{n_\tr}\le t)|\le C_{\kappa,\alpha}\frac{(\log m)^{\alpha+5/4}}{n^{1/4}}=o(1)\,,
\end{equation}
where the last $o(1)$ statement comes from the first part of \eqref{eq:4.3-m-bound}.

Under the second part of \eqref{eq:4.3-m-bound}, one can find a $\delta_n$ such that $\delta_n\sqrt{\log m}=o(1)$ and $\epsilon_{n_\tr}(\log m)^{\alpha+1}=o(\delta_n)$ (for example, one possibility is  $\delta_n=\epsilon_{n_\tr}^{1/2}(\log m)^{\alpha/2}$).

Then
\begin{align*}
    &\mathbb P\left[\max(\mathbf W_{n_\tr}+\boldsymbol{\zeta}_n)\le t\right]\\
    \le & \mathbb P\left[\max\mathbf W_{n_\tr}\le t+\delta_n\right]+\mathbb P(\max|\boldsymbol{\zeta}_n|>\delta_n)\\
    \le & \mathbb P(\max \mathbf Y_{n_\tr}\le t+\delta_n)+o(1)+o(1)\\
    \le & \mathbb P(\max \mathbf Y_{n_\tr}\le t)+C\delta_n\sqrt{\log m}+o(1)\\
    = &\mathbb P(\max \mathbf Y_{n_\tr}\le t)+o(1)\,,
\end{align*}
where the first $o(1)$ in the second inequality comes from \eqref{eq:4.3-cckk-bar-ell}, the second $o(1)$ in the same line comes from the fact that $\max|\boldsymbol{\zeta}_n|$ is $(c_{\alpha}\epsilon_{n_\tr}(\log m)^{\alpha+1},\alpha+1)$-SW as guaranteed by \Cref{lem:2.5-sw-maximal} and the choice of $\delta_n$, and the last inequality comes from a Gaussian anti-concentration result \citep[Lemma J.3]{chernozhukov2022improved}.
Similarly we can obtain the other direction and conclude the proof.
\end{proof}

\paragraph{High-dimensional Gaussian comparison with deterministic centering}
It is possible to extend \Cref{thm:4.3-hd-cv-clt-rand} to deterministic centering.  
While it is desirable to have a result analogous to \Cref{pro:4.2-cv-clt-consolidate}, the dichotomy of whether the validation uncertainty dominates the fitting uncertainty needs to be checked for each coordinate.
Define $\boldsymbol{\mu}_n$, $\boldsymbol{g}_n(X_0)$, and $\bar{\boldsymbol{g}}_n(X_0)$ as the vector versions of $\mu_n$, $g_n(X_0)$, and $\bar g_n(X_0)$, respectively. 
For simplicity, we consider the following two scenarios. 

In the first scenario, where $\nabla_1 R^{(r)}(\D_n)$ is $(n^{-1}\sigma_n^{(r)}\epsilon_n,\alpha)$-SW for all $r\in[m]$ with $\epsilon_n$ vanishing fast enough in the sense of \eqref{eq:4.3-m-bound}, then \Cref{thm:4.3-hd-cv-clt-rand} would hold for the deterministic centering without much modification.  The following theorem is a deterministic centering counterpart of Part 1 of \Cref{pro:4.2-cv-clt-consolidate}.
\begin{theorem}\label{thm:4.3-hd-cv-clt-det-1}
    Assume the same conditions in \Cref{thm:4.3-hd-cv-clt-rand}. If $\nabla_1 R^{(r)}(\D_n)$ is $(n^{-1}\sigma_n^{(r)}\tilde\epsilon_n,\tilde\alpha)$-SW for all $r\in[m]$ with $\tilde\epsilon_n(\log m)^{\tilde\alpha+1}=o(1)$, then
        $$\sup_{t\in\mathbb R}\left|\mathbb P\left(\max \sqrt{n}S_{n_\tr}^{-1}(\hat{\mathbf R}_{\cv,n,K}-\boldsymbol{\mu}_{n_\tr})\le t\right) - \mathbb P(\max\mathbf Y_{n_\tr}\le t) \right|=o(1)\,.$$
\end{theorem}

\begin{proof}[Proof of \Cref{thm:4.3-hd-cv-clt-det-1}]
For the first part, let $\tilde\zeta_n = \sqrt{n}\max|S_{n_\tr}^{-1}(\bar{\mathbf R}_{\cv,n,K}-\boldsymbol{\mu}_{n_\tr})|$.   The sub-Weibull condition assumed on $\nabla_1 R^{(r)}(\D_n)$ implies that, by \Cref{thm:2.5-sw-mcdiarmid} and triangle inequality,  each coordinate of $S_{n_\tr}^{-1}(\bar R_{\cv,n,K}-\boldsymbol{\mu}_{n_\tr})$ is $(c_K\tilde\epsilon_{n_\tr},\tilde\alpha+1/2)$-SW, with $c_K=\sqrt{K/(K-1)}\in(1,1.5)$.
Thus \Cref{lem:2.5-sw-maximal} implies that $\tilde\zeta_n$ is $(c_Kc_{\tilde\alpha}(\log m)^{\tilde\alpha+1/2}\tilde\epsilon_{n_\tr},\tilde\alpha+1/2)$-SW, for some constant $c_{\tilde\alpha}$ depending only on $\tilde\alpha$.  

The condition $\tilde\epsilon_n(\log m)^{\tilde\alpha+1}=o(1)$ ensures existence of a sequence $\tilde\delta_n$ such that $\tilde\delta_n=o(1/\sqrt{\log m})$ and 
$(\log m)^{\tilde\alpha+1/2}\epsilon_{n_\tr}=o(\tilde\delta_n)$. Therefore,
\begin{align*}
&\mathbb P\left(\sqrt{n}\max(\hat{\mathbf R}_{\cv,n,K}-\boldsymbol{\mu}_{n_\tr})\le t\right)\\
\le & \mathbb P\left(\sqrt{n}\max(\hat{\mathbf R}_{\cv,n,K}-\bar{\mathbf R}_{\cv,n,K})\le t+\tilde\delta_n\right)+\mathbb P(\tilde\zeta_n>\tilde\delta_n)\\
\le & \mathbb P(\max \mathbf Y_{n_\tr}\le t+\tilde\delta_n)+o(1)\\
\le & \mathbb P(\max \mathbf Y_{n_\tr}\le t)+o(1)\,,
\end{align*}
where the first inequality follows from union bound, the second inequality from \Cref{thm:4.3-hd-cv-clt-rand} and the sub-Weibull tail of $\tilde\zeta_n$, the third from Gaussian anti-concentration (Lemma J.3 of \cite{chernozhukov2022improved} since $\tilde\delta_n=o(1/\sqrt{\log m})$.
The other direction can be established similarly.
\end{proof}

In the second scenario, when $\nabla_1 R^{(r)}$ is large as in the second case in \Cref{pro:4.2-cv-clt-consolidate} for some values of $r$, the deterministic centering would lead to a different asymptotic correlation matrix, whose $(r,s)$-entry equals ${\rm Corr}(\bar\ell_{n_\tr}^{(r)}(X_0)+\bar g_{n_\tr}^{(r)}(X_0),~\bar\ell_{n_\tr}^{(s)}(X_0)+\bar g_{n_\tr}^{(s)}(X_0))$. The following theorem is a deterministic centering counterpart of \Cref{thm:4.2-cv-clt-det}.
\begin{theorem}
    \label{thm:4.3-hd-cv-clt-det-2}
     Let $(\tau_n^{(r)})^2=\var(\ell_n^{(r)}(X_0)+g_n^{(r)}(X_0))>0$.
        If, for all $r$ and $n$, $\nabla_1\ell(X_0,\D_n)$ is $(n^{-1/2}\epsilon_n\tau_n^{(r)},\alpha)$-SW,
        $\nabla_1\nabla_2 R(\D_n)$ is $(n^{-3/2}\epsilon_n\tau_n^{(r)},\alpha)$-SW, and $[\bar\ell_n^{(r)}(X_0)+\bar g_n^{(r)}(X_0)]/\tau_n^{(r)}$ is $(\kappa,\alpha)$-SW with constants $(\kappa$, $\alpha)$, and $\epsilon_n$ such that \eqref{eq:4.3-m-bound} holds, then 
$$\sup_{t\in\mathbb R}\left|\mathbb P\left(\max \sqrt{n}T_{n_\tr}^{-1}(\hat{\mathbf R}_{\cv,n,K}-\boldsymbol{\mu}_{n_\tr})\le t\right) - \mathbb P(\max\mathbf Y_{n_\tr}\le t) \right|=o(1)\,,$$
where $T_n$ is an $m\times m$ diagonal matrix with the $r$th diagonal entry being $\tau_n^{(r)}$, and $\mathbf Y_n$ a centered Gaussian vector whose covariance matrix has $(r,s)$-entry ${\rm Corr}(\bar\ell_{n_\tr}^{(r)}(X_0)+\bar g_{n_\tr}^{(r)}(X_0),~\bar\ell_{n_\tr}^{(s)}(X_0)+\bar g_{n_\tr}^{(s)}(X_0))$.
\end{theorem}
\begin{proof}
    [Proof of \Cref{thm:4.3-hd-cv-clt-det-2}]
    Using the vector versions of \eqref{eq:4.2-asymp-lin-with-G}, \eqref{eq:4.2-def-G}, and \eqref{eq:4.2-approx-G-with-res}, we have
\begin{align*}
   & \sqrt{n}T_{n_\tr}^{-1}(\hat{\mathbf R}_{\cv,n,K}-\boldsymbol{\mu}_{n_\tr})\\
   =&\frac{1}{\sqrt{n}}T_{n_\tr}^{-1}\sum_{i=1}^n \left[\bar{\boldsymbol{\ell}}_{n_\tr}(X_i)+\bar{\boldsymbol{g}}_{n_\tr}(X_i)\right] +\boldsymbol{\tilde\zeta}_n+
   \frac{1}{\sqrt{n}}T_{n_\tr}^{-1}\sum_{i=1}^n \boldsymbol{\Lambda}_i\,,
\end{align*}
where $\tilde{\boldsymbol{\zeta}}_n$ is the analogous version of $\boldsymbol{\zeta}_n$ in the proof of \Cref{thm:4.3-hd-cv-clt-rand} with $S_n$ replaced by $T_n$.

The first-order stability condition implies that $\tilde{\boldsymbol{\zeta}}_n$ is component-wise $(c_\alpha \epsilon_n,\alpha+1/2)$-SW, and the martingale decomposition of $\boldsymbol{\Lambda}_i$ implies that $\frac{1}{\sqrt{n}}T_{n_\tr}^{-1}\sum_{i=1}^n \boldsymbol{\Lambda}_i$ is component-wise $(c_\alpha\epsilon_n,\alpha+1)$-SW.  Thus
$\boldsymbol{\tilde\zeta}_n+
   \frac{1}{\sqrt{n}}T_{n_\tr}^{-1}\sum_{i=1}^n \boldsymbol{\Lambda}_i$ is component-wise $(c_\alpha\epsilon_n,\alpha+1)$-SW. The rest of the proof follows a similar argument as in \Cref{thm:4.3-hd-cv-clt-det-1}.
\end{proof}

A fully consolidated version of deterministic centering Gaussian comparison would involve studying the scenario where the risk stability condition $\nabla_1 R^{(r)}(\D_n)\ll n^{-1}$ for some but not all coordinates $r\in[m]$, which is an open problem. Practically, the random centering Gaussian comparison is easier to use with easier-to-verify assumptions. In some applications, random centering Gaussian comparison provides sufficient theoretical support for non-trivial inference problems. We provide a couple of such examples in \Cref{sec:5-applications}.

\paragraph{Variance and Correlation Estimation} In order to use the high-dimensional Gaussian comparison \Cref{thm:4.3-hd-cv-clt-rand} to construct confidence sets for $\bar{\mathbf R}_{\cv,n,K}$, one needs to estimate not only the marginal variances $(\sigma_{n_\tr}^{(r)})^2$ for $r\in[m]$, but also the pairwise correlation $\Gamma_{n_\tr}(r,s)$ in \eqref{eq:4.3-cov-mat-rand}. The same argument in the proof of \Cref{thm:4.1-var-est} can provide entry-wise $L_2$-consistency of the sample variance and correlation estimates. However, such entry-wise consistency is not sufficient for high-dimensional inference, because when $m=m_n$ increases with $n$, the Gaussian comparison result is not a central limit theorem as there is no fixed limit of convergence in distribution.  To make valid inference with the estimated variances and correlations, we need a more refined error analysis of the estimated variance and covariance, which will be detailed in \Cref{subsec:5.1-model-conf-set} below.

\paragraph{Rates of Convergence} In this article we do not keep track of the rates of convergence in CLT or Gaussian comparison as we focus on asymptotically valid inference which only requires vanishing Kolmogorov-Smirnov distance between the scaled and centered CV risk estimate and the corresponding Gaussian distribution.  As mentioned in \Cref{rem:4.1-wasserstein}, such rates of convergence are available by adapting the proofs, assuming a specific rate of convergence on the loss stability.

\subsection{Auxiliary Lemmas}\label{subsec:4.4-aux}
In this subsection we collect the auxiliary lemmas used in earlier subsections.

The following lemma explains the ``free nabla'' trick.
\begin{lemma}
    [Free Nabla Lemma]\label{lem:4.4-free-nabla}
    Let $\mathbf D=(X_1,...,X_n,X_1',...,X_n',W)$ be independent random objects such that for each $i\in[n]$, $X_i$ and $X_i'$ have the same distribution. 
    Let $h$ and $g$ be two functions that act on subsets of $\mathbf D$ such that for some $i\in[n]$
    \begin{enumerate}
        \item $g(\mathbf D)$, $h(\mathbf D)$, and $g(\mathbf D)h(\mathbf D)$ all have finite $L_1$ norms,
        \item $\mathbb E_{X_i} g =0$, and
        \item $h$ does not involve $X_i'$.
    \end{enumerate}
    Then
    $$\mathbb E (g h)=\mathbb E(g\nabla_i h)\,.$$
\end{lemma}
The operator $\nabla_i$ in this lemma is defined in the usual sense since $h$ does not involve $X_i'$. In particular, $\nabla_i h$ is the  difference between the original $h$ and the one obtained
by replacing $X_i$ by $X_i'$ in the calculation of $h$.
\begin{proof}
    [Proof of \Cref{lem:4.4-free-nabla}]
    Let $h^i$ be the version of $h$ obtained by swapping $X_i$ and $X_i'$ in $\mathbf D$. It suffices to show that $\mathbb E (g h^i) = 0$.
    In fact, by construction $h^i$ does not involve $X_i$ since $h$ does not involve $X_i'$.  Thus
    \begin{align*}\mathbb E g h^i=&\mathbb E\left[(\mathbb E_{X_i} g h^i) \right] =\mathbb E\left[(\mathbb E_{X_i} g) h^i\right]=0\,.\qedhere\end{align*}
\end{proof}

The following lemma provides a convenient bound on $\nabla_i\xi_j$ for $i\notin I_{k_j}$.
\begin{lemma}\label{lem:4.4-nabla-xi}
    For $i\neq j\in [n]$ we have
    $\|\nabla_i\xi_j\|_2\le \|\nabla_i\ell(X_j,\D_{n,-k_j})\|_2$.
\end{lemma}
\begin{proof}
    [Proof of \Cref{lem:4.4-nabla-xi}]
    Let $\D_{n,-k_j}^i$ be the dataset obtained by replacing $X_i$ by $X_i'$ in $\D_{n,-k_j}$ while keeping everything else the same.
    Then
    \begin{align*}
        \|\nabla_i\xi_j\|_2^2=& \|\nabla_i \ell(X_j,\D_{n,-k_j})\|_2^2-2\mathbb E \nabla_i\ell(X_j,\D_{n,-k_j})\nabla_i R(\D_{n,-k_j})+\|\nabla_i R(\D_{n,-k_j})\|_2^2\\
        =&\|\nabla_i \ell(X_j,\D_{n,-k_j})\|_2^2-\|\nabla_i R(\D_{n,-k_j})\|_2^2\,,
    \end{align*}
    where the last step follows from
    \begin{align*}
        \mathbb E \nabla_i\ell(X_j,\D_{n,-k_j})\nabla_i R(\D_{n,-k_j}) = \|\nabla_i R(\D_{n,-k_j})\|_2^2
    \end{align*}
    by first taking expectation over $X_j$.
\end{proof}

The following lemma is a special case of Rio's inequality \Cref{thm:rio} and generalizes \Cref{thm:2.5-sw-mcdiarmid}.
\begin{lemma}[Rio's Inequality For Conditionally Independent Sequences]\label{lem:4.4-rio_cond_i.i.d.}
    Let $Z_1,...,Z_n$ be random variables that are conditionally independent given a $\sigma$-field $\mathcal C$ such that $\mathbb E(Z_i|\mathcal C)=0$ for all $i$. Then for any $q\ge 2$,
    $$
    \left\|\sum_{i=1}^n Z_i\right\|_q^2\le (q-1)\sum_{i=1}^n \|Z_i\|_q^2\,.
    $$
    Moreover, if $Z_i$ is $(\kappa_i,\alpha_i)$-SW ($i=1,...,n$) then
    $\sum_{i=1}^n Z_i$ is $(\kappa,\alpha)$-SW with
    $\kappa=2^\alpha\sqrt{\sum_{i=1}^n\kappa_i^2}$ and $\alpha=\max_{i=1}^n\alpha_i+1/2$.
\end{lemma}
\begin{proof}
   [Proof of \Cref{lem:4.4-rio_cond_i.i.d.}]
The proof is identical to that of \Cref{thm:2.5-sw-mcdiarmid}. 
\end{proof}

\Cref{lem:4.4-rio_cond_i.i.d.} can be used to prove various important consequences of the loss stability, collected in the following lemma.
\begin{lemma}[Useful Consequences of Loss Stability]\label{lem:4.4-useful-fact-loss-stab}
Assume an i.i.d. data sequence and a symmetric estimator. Let $\xi_i=\ell(X_i,\D_{n,-k_i})-R(\D_{n,-k_i})$ and $\bar\xi_i=\mathbb E(\xi_i|X_i)$.  The following holds for any fixed $q\ge 2$ and each $i\in[n]$,
\begin{enumerate}
\item
$\|\ell(X_0,\D_n)-\ell_n(X_0)\|_q\le \sqrt{q-1}\sqrt{n}\|\nabla_1\ell(X_0,\D_n)\|_q$ and as a result,
$$
\|R(\D_n)-\mu_n\|_q\le \sqrt{q-1}\sqrt{n}\|\nabla_1\ell(X_0,\D_n)\|_q
$$
$$
\|\xi_i-\bar\xi_i\|_q\le 2\sqrt{q-1}\sqrt{n_\tr}\|\nabla_1\ell(X_0,\D_{n_\tr})\|_q
$$
    \item
    $$
   \left\|\nabla_i\sum_{j\notin I_{k_i}}(\xi_j-\bar\xi_j)\right\|_2^2
   \le n_\tr\|\nabla_1\ell(X_0,\D_{n_\tr})\|_2^2+n_{\tr}^2\|\nabla_1\nabla_2\ell(X_0,\D_{n_\tr})\|_2^2\,.
    $$
    \item
    $$
    \left\|\nabla_i\sum_{j\notin I_{k_i}}(\xi_j-\bar\xi_j)\right\|_q\le
    2\sqrt{q-1}\sqrt{K}\sqrt{n}\|\nabla_1\ell(X_0,\D_{n_\tr})\|_q\,.
    $$
    \item 
    $$
    \left\|\sum_{i=1}^n(\xi_i-\bar\xi_i)\right\|_q \le 2(q-1)\sqrt{K}n\|\nabla_1\ell(X_0,\D_{n_\tr})\|_q\,.
    $$
\end{enumerate}
\end{lemma}
\begin{proof}
    [Proof of \Cref{lem:4.4-useful-fact-loss-stab}]
For Part 1, the first inequality follows from the conditional version Rio's inequality as applied to the martingale increment sequence 
$$\mathbb E[\ell(X_0,\D_n)|\D_i,X_0]-\mathbb E[\ell(X_0,\D_n)|\D_{i-1},X_0]=\mathbb E[\nabla_i \ell(X_0,\D_n)|\D_i,X_0]\,,~~1\le i\le n\,,$$ as in the proof of \Cref{thm:2.5-sw-mcdiarmid}. The second inequality follows from the first one and Jensen's inequality. The third one follows from triangle inequality.

In the proofs of Parts 2 and 3, without loss of generality we assume $i=1$ and $k_1=1$.

For Part 2, observe that $\nabla_1(\xi_i-\bar\xi_i)=\nabla_1\xi_i$ for $i\notin I_1$. So we have
\begin{align*}
    &\left\|\sum_{i\notin I_1}\nabla_1\xi_i\right\|_2^2
    =\sum_{i\notin I_1} \|\nabla_1\xi_i\|_2^2 + \sum_{i\neq j\notin I_1}\mathbb E(\nabla_1 \xi_i)(\nabla_1\xi_j)\\
    =&\sum_{i\notin I_1} \|\nabla_1\xi_i\|_2^2 + \sum_{i\neq j\notin I_1}\mathbb E(\nabla_j\nabla_1 \xi_i)(\nabla_i\nabla_1\xi_j)\\
    \le & n_\tr \|\nabla_1\ell(X_0,\D_{n_\tr})\|_2^2 + n_\tr^2 \|\nabla_1\nabla_2\ell(X_0,\D_{n_\tr})\|_2^2\,,
 \end{align*}
 where the second line follows from an application \Cref{lem:4.4-free-nabla}.

For Part 3, consider $k\neq 1$ and observe that $(\nabla_1\xi_i:i\in I_{n,k})$ is a martingale increment sequence with respect to the collection of nested $\sigma$-fields: $\{\D_{n,-k}\cup\{X_1'\}\cup \{X_j:j\in I_{n,k},~j\le i\}:i \in I_{n,k}\}$. Therefore,
 \begin{align*}
 \left\|\nabla_1\sum_{i\in I_{n,k}}\xi_i\right\|_q^2 \le (q-1) \sum_{i\in I_{n,k}} \|\nabla_1 \xi_i\|_q^2 \le  4(q-1) n_\te \|\nabla_1 \ell(X_0,\D_{n_\tr})\|_q^2\,,
 \end{align*}
 where the first inequality follows from \Cref{lem:4.4-rio_cond_i.i.d.} because $(\nabla_1\xi_i:i\in I_{n,k})$ are conditionally i.i.d. given $(\D_{n,-k},X_1')$ and have conditional expectation $0$. 
 Then the desired result follows from the triangle inequality by summing over $k\in\{2,...,K\}$.

For Part 4, fix a $k\in[K]$,
then using \Cref{lem:4.4-rio_cond_i.i.d.} and Part (1) we obtain
$$
\left\|\sum_{i\in I_{n,k}}(\xi_i-\bar\xi_i)\right\|_q\le 2(q-1)\sqrt{n_\te}\sqrt{n_\tr}\|\nabla_1\ell(X_0,\D_{n_\tr})\|_q\,,
$$
Thus the claimed result follows from triangle inequality.
\end{proof}

The following lemma provides $\ell_q$ norm bound for the random centering residuals.
\begin{lemma}\label{lem:4.4-rand-center-residual-bound}
Let $(X_i:i\ge 0)$ be an i.i.d. sequence, $\hat f$ a symmetric estimator, and $\ell(\cdot,\cdot)$ a loss function. For $q\in[2,\infty]$ let $\kappa_n=\|\nabla_1\ell(X_0,\D_n)\|_q$. Then
$$
\left\|\frac{\sqrt{n}}{\sigma_{n_\tr}}(\hat R_{\cv,n,K}-\bar R_{\cv,n,K}) - \frac{1}{\sqrt{n}\sigma_{n_\tr}}\sum_{i=1}^n \bar\ell_{n_\tr}(X_i)\right\|_q\le 2(q-1)\frac{(nK)^{1/2}\kappa_{n_\tr}}{\sigma_{n_\tr}}\,.
$$
Moreover, if $\nabla_1\ell(X_0,\D_n)$ is $(\kappa_n,\alpha)$-SW, then
$$\frac{\sqrt{n}}{\sigma_{n_\tr}}(\hat R_{\cv,n,K}-\bar R_{\cv,n,K}) - \frac{1}{\sqrt{n}\sigma_{n_\tr}}\sum_{i=1}^n \bar\ell_{n_\tr}(X_i)$$
is $(2\sqrt{nK}\kappa_{n_\tr}/\sigma_{n_\tr},~\alpha+1)$-SW.
\end{lemma}
\begin{proof}
    [Proof of \Cref{lem:4.4-rand-center-residual-bound}]
    Use notation $\xi_i = \ell(X_i,\D_{n,-k_i}) - R(\D_{n,-k_i})$ and $\bar \xi_i =\mathbb E\left[\xi_i|X_i\right]=\bar\ell_{n_\tr}(X_i)$.

    Then
    \begin{align*}
\frac{\sqrt{n}}{\sigma_{n_\tr}}\left(\hat R_{\cv,n,K}-\bar R_{\cv,n,K}\right)-\frac{1}{\sqrt{n}\sigma_{n_\tr}}\sum_{i=1}^n\bar\ell_{n_\tr}(X_i)=\frac{1}{\sqrt{n}\sigma_{n_\tr}}\sum_{i=1}^n (\xi_i-\bar\xi_i)\,,
\end{align*}
and the result follows from Part 4 of \Cref{lem:4.4-useful-fact-loss-stab}.
\end{proof}

\subsection{Bibliographic Notes}\label{subsec:4.5-bib}
The rigorous normal approximation of cross-validated statistics has been conjectured and regarded as an open problem in \cite{yang2006comparing}.
The univariate and low-dimensional central limit theorems for cross-validated risk estimates are first established in \cite{bayle2020cross} and \cite{austern2020}. The random centering CLT based on the first-order loss stability \eqref{eq:4.1-nabla-l2-bound} presented in \Cref{thm:4.1-clt-rand} largely follows the argument in \cite{bayle2020cross}.  The same random centering CLT is also proved in \cite{austern2020}, which relied on an additional second-order loss stability as in Part (b) of \Cref{asm:4.2-second-order-stab}. This second-order stability condition turned out to be redundant for the random centering CLT.    While \cite{bayle2020cross} does not offer a deterministic centering CLT, \cite{austern2020} provides one.  Our deterministic centering CLT (\Cref{thm:4.2-cv-clt-det}) requires fewer assumptions and offers additional insights through the asymptotic linear representation of the $G$ term as well as the consolidated version \Cref{pro:4.2-cv-clt-consolidate}.  In particular, unlike \cite{austern2020}, our deterministic centering CLT does not require the second-order loss stability or a specific bound on the first-order risk stability.  Our consolidated version \Cref{pro:4.2-cv-clt-consolidate} also makes it clear that, when the risk variability is non-negligible, it is the relative size of the second-order risk stability to the first-order risk stability that matters.   In addition to simplified assumptions, our proof of the deterministic centering CLT is also substantially simpler than that of \cite{austern2020}.  Our proof clearly reflects the role played by stability, which is solely in the asymptotic linearity of the CV risk estimate. The deterministic asymptotic variance estimate presented in \eqref{eq:4.2-var-est-determ} is the same as the one proposed in \cite{austern2020}.

A deterministic-centering CLT for \(K\)-fold cross-validation is established by \cite{li2023asymptotics} for regression under high-level assumptions, including (i) the third derivative of the loss with respect to the regression function’s value is identically zero, and (ii) the estimated regression function admits an asymptotic linear representation. These assumptions are highly nontrivial: The third-derivative condition forces \(\ell((z_0,y_0),\hat f)\) to depend on \(\hat f\) quadratically in \(\hat f(z_0)\), and the asymptotic linear representation is an abstract property of the loss/estimator pair which requires verification.
The developments in \Cref{subsec:4.2-cv-clt-det} show that the stability conditions in \Cref{thm:4.2-cv-clt-det,pro:4.2-cv-clt-consolidate} are sufficient to guarantee such asymptotic linear representations.

The high-dimensional Gaussian comparison for independent random vectors has been studied as early as in \cite{bentkus2003dependence}, with a major breakthrough in \cite{chernozhukov2013gaussian} that allows the dimensionality to be $\exp(n^a)$ for small values of $a$.  Many further developments have been made since \cite{chernozhukov2013gaussian}, with the most recent ones being \cite{chernozhukov2022improved} and \cite{chernozhukov2023nearly}. Readers are referred to those papers for a more comprehensive literature review on this topic.  In our case we used results from \cite{chernozhukov2022improved} as it does not require the asymptotic covariance matrix $\Gamma_n$ to be non-degenerate.
It is possible to use more refined techniques, such as those in \cite{chernozhukov2023nearly}, to obtain approximation error bounds with better dependence on $n$.

The high-dimensional Gaussian approximation for cross-validated risks presented in \Cref{subsec:4.3-hd-clt} is developed in \cite{kissel2022high}, which is motivated from the model selection literature \citep{lei2020cross}.  The proofs presented in this article are much simpler and more streamlined than those in \cite{kissel2022high}, and feature a similar reduction in required assumptions as in the univariate CLT.  Our Gaussian comparison result for cross-validation can also be viewed as an instance of high-dimensional Gaussian comparison for dependent data.  Existing results in the literature include vector-valued $U$-statistics \citep{chen2018gaussian} and vector-valued stochastic processes with mixing properties \citep{chang2021central,kurisu2021gaussian}.  These results are based on a short-range dependence structure, where the dependence between two random vectors decays quickly as the indices move apart from each other.  However, such a structure does not hold for cross-validation, where the loss functions evaluated at any two sample points have a small but nonzero correlation due to the involvement of sample points in both evaluation and estimation.
\newpage

\section{Some Applications}\label{sec:5-applications}
In this section, we demonstrate several interesting applications of the general tools developed in the previous sections.
In \Cref{subsec:5.1-model-conf-set}, we use the Gaussian comparison results in \Cref{sec:4-clt} to construct confidence sets for the risks and then convert them into model confidence sets.  In \Cref{subsec:5.2-cvc-consist}, these model confidence sets are used to obtain consistent model selection in classical settings where the standard CV method is known to fail.
In \Cref{subsec:5.3-argmin}, we use the central limit theorems developed in \Cref{subsec:4.1-clt-random} to develop a new method for high-dimensional global mean inference, which further leads to a method for constructing confidence sets for the argmin index of a vector from noisy observations (\Cref{subsec:5.4-argmin-conf}).
Finally, in \Cref{subsec:5.5-cross-conf}, we combine the techniques and tools in \Cref{sec:bousquet,sec:4-clt} to prove asymptotic validity for the cross-conformal prediction method.

\subsection{Model Selection Confidence Sets}\label{subsec:5.1-model-conf-set}
Consider the model selection setting in \Cref{subsec:3.1-yang07}, where we have shown that if the estimators are well-separated in the sense of \Cref{asm:3.1-stochastic_domination}, then cross-validation can consistently select the better model when the sample size increases to infinity.  However, in many applications such as tuning parameter selection, the stochastic dominance condition becomes less plausible, and, in order to take into account the uncertainty in the CV risk estimates, it becomes more reasonable to aim at finding all ``nearly optimal'' models.  The tools developed in \Cref{subsec:4.3-hd-clt} are suitable for this purpose.

\paragraph{From risk confidence set to model confidence set}
Suppose we run $K$-fold CV for each of the $m$ candidate estimators. 
Our goal is to learn
\begin{equation}\label{eq:5.1-r^*-rand-center}
r^*=\arg\min_{r\in[m]} \bar R_{\cv,n,K}^{(r)}\,,
\end{equation}
and in the case that there exist multiple good models with close to optimal values, find a confidence set $\hat\Theta\subseteq[m]$ such that
\begin{equation}\label{eq:5.1-conf-set-def-rand}
\mathbb P(r^*\in \hat\Theta)\ge 1-\beta
\end{equation}
for some pre-specified Type I error level $\beta$.

It may not seem entirely natural to choose $r^*$ as the target parameter of interest, as it is the average risk of $K$ fitted parameters obtained from $K$ leave-one-fold-out estimates. This choice is aligned with \Cref{thm:4.3-hd-cv-clt-rand}, which provides simultaneous uncertainty quantification of $\hat{\mathbf R}_{\cv,n,K}-\bar{\mathbf R}_{\cv,n,K}$.  Moreover, \Cref{thm:4.3-hd-cv-clt-det-1} suggests that if the risks are stable enough, then such a random centering result is equivalent to the deterministic centering with sample size $n_\tr$.

Let $\beta\in(0,1)$ be a nominal Type I error level and $\Sigma$ a positive semi-definite matrix. Define 
\begin{equation}\label{eq:5.1-gaussian-quantile-abs}
t_{\beta,\Sigma} \coloneqq \text{upper }\beta\text{-quantile of } \max|N(0,\Sigma)|\,,\end{equation}
where $|\cdot|$ means entry-wise absolute value.
Then according to \Cref{thm:4.3-hd-cv-clt-rand}, under the conditions therein, the following is an asymptotic $(1-\beta)$ simultaneous confidence set\footnote{The  terminology ``confidence set'' is a bit different from its standard use, as the target $\bar{\mathbf R}_{\cv,n,K}$ itself is random. However, the usual frequentist interpretation of confidence sets are still valid for the probabilistic statements about $\bar{\mathbf R}_{\cv,n,K}$. Moreover, as discussed in the previous paragraph, the random centering will be asymptotically equivalent to the deterministic centering at $\boldsymbol{\mu}_{n_\tr}$ if the risks are stable enough in the sense of Part 1 of \Cref{pro:4.2-cv-clt-consolidate}.} of $\hat{\mathbf R}_{\cv,n,K}$:
\begin{equation}
    \label{eq:5.1-simu-cs-ideal}
    \prod_{r\in[m]} \left[\hat R^{(r)}_{\cv,n,K}-\frac{\sigma_{n_\tr}^{(r)}}{\sqrt{n}}t_{\beta,\Gamma_{n_\tr}}\,,~\hat R^{(r)}_{\cv,n,K}+\frac{\sigma_{n_\tr}^{(r)}}{\sqrt{n}}t_{\beta,\Gamma_{n_\tr}}\right]\,,
\end{equation}
where $\Gamma_{n_\tr}$ is the correlation matrix defined in \Cref{eq:4.3-cov-mat-rand}.

However, the confidence set in \Cref{eq:5.1-simu-cs-ideal} is not practical as it involves unknown quantities: $(\sigma_{n_\tr}^{(r)},\Gamma_{n_\tr},t_{\beta,\Gamma_{n_{\tr}}})$.

We can estimate $\sigma_{n_\tr}$ for each $r$ as in \Cref{thm:4.1-var-est}. Similarly, we can estimate $\Gamma_{n_\tr}(r,s)$ by
\begin{equation}\label{eq:5.1-gamma-hat}
\hat\Gamma_{n_\tr}(r,s)=\frac{1}{n \hat\sigma_{n_\tr}^{(r)}\hat\sigma_{n_\tr}^{(s)}}\sum_{i=1}^n[\ell^{(r)}(X_i,\D_{n,-k_i})-\hat R_{\cv,n,K}^{(r)}][\ell^{(s)}(X_i,\D_{n,-k_i})-\hat R_{\cv,n,K}^{(s)}]\,.
\end{equation}

After obtaining $\hat\Gamma_{n_\tr}$, we can numerically approximate $t_{\beta,\hat\Gamma_{n_\tr}}$ using Monte Carlo methods.  The quality of approximation depends on the Monte Carlo sample size and can be made arbitrarily small. In the following analysis we will ignore this approximation error and simply work under the assumption that $t_{\beta,\hat\Gamma_{n_\tr}}$ is available.  However, our analysis does explicitly account for the estimation error in $\hat \Gamma_{n_\tr}$ and the difference between $t_{\beta,\hat\Gamma_{n_\tr}}$ and $t_{\beta,\Gamma_{n_\tr}}$.

\begin{theorem}
    \label{thm:5.1-simul-conf-set-naive}
    If \Cref{eq:4.3-nabla_ell_sw} and \Cref{eq:4.3-4th-moment} hold and
    \begin{equation}\label{eq:5.1-corr-est-condition}(\epsilon_{n_\tr}+n^{-1/2})(\log m)^{3\alpha+4}=o(1)\end{equation} then for any $\beta\in(0,1)$
    \begin{equation*}
        \mathbb P\left[\left|\hat R_{\cv,n,K}^{(r)}-\bar R_{\cv,n,K}^{(r)}\right|\le \frac{\hat\sigma_{n_\tr}^{(r)}}{\sqrt{n}}t_{\beta,\hat\Gamma_{n_\tr}}\,,~~\forall~r\in[m]\right] = 1-\beta+o(1)\,.
    \end{equation*}
\end{theorem}
Recall that \eqref{eq:4.3-nabla_ell_sw} and \eqref{eq:4.3-4th-moment} are the stability and tail conditions of the loss $\ell^{(r)}(X_0,\D_n)$, respectively. These conditions are stated uniformly over $r\in[m]$.
Intuitively, a $K$-fold version of \Cref{thm:2.5-loocv-concentration-sw} suggests that each coordinate of $\hat{\mathbf R}_{\cv,n,K}$ concentrates around its expected value at scale $O(n^{-1/2}+\epsilon_{n_\tr})$. The additional $(\log m)^{3\alpha+4}$ factor required in \eqref{eq:5.1-corr-est-condition} is used to ensure (i) the concentration of $\hat{\mathbf R}_{\cv,n,K}$ holds for all entries at the same time, and (ii) the correlation matrix estimation is uniformly accurate over all entries.
\begin{proof}
    [Proof of \Cref{thm:5.1-simul-conf-set-naive}]
    We use simplified notation by dropping the subindices whenever there is no confusion: $\hat{\mathbf R} = \hat{\mathbf R}_{\cv,n,K}$,   $\bar{\mathbf R}=\bar{\mathbf R}_{\cv,n,K}$, $\hat\Gamma=\hat\Gamma_{n_\tr}$, $\Gamma=\Gamma_{n_\tr}$.
   Let $\hat S$ be the $m\times m$ diagonal matrix whose $r$th diagonal entry is $\hat\sigma_{n_\tr}^{(r)}$, and $S$ its population version whose corresponding entry is $\sigma_{n_\tr}^{(r)}$.  

Define   \begin{align*}
      W_0 = & \left\|\sqrt{n}\hat S^{-1}(\hat{\mathbf R}-\bar{\mathbf R})\right\|_\infty\,,\\
      W_1 = & \left\|\sqrt{n} S^{-1}(\hat{\mathbf R}-\bar{\mathbf R})\right\|_\infty\,,     \\
      W_2= & \|N(0,\Gamma)\|_\infty\,.  \end{align*}

We first control $|W_0-W_1|$, observe that
\begin{align*}
    |W_0-W_1|\le \|\sqrt{n}S^{-1}(\hat{\mathbf R}-\bar{\mathbf R})\|_\infty \cdot \max_{r\in[m]}\left|\frac{\sigma_{n_\tr}^{(r)}}{\hat\sigma_{n_\tr}^{(r)}}-1\right|\,.
\end{align*}
By \Cref{lem:4.4-rand-center-residual-bound} and \Cref{pro:2.5-sum-prod-sw}, each entry of
$\sqrt{n}S^{-1}(\hat{\mathbf R}-\bar{\mathbf R})$ is $(c_{\alpha,\kappa},\alpha+1)$-SW and hence, by \Cref{lem:2.5-sw-maximal}, the maximum entry-wise absolute value of $\sqrt{n}S^{-1}(\hat{\mathbf R}-\bar{\mathbf R})$ is $O_P((\log m)^{\alpha+1})$. Combine this with Part 1 of \Cref{lem:5.1-corr-est}  we have 
$$
|W_0-W_1|=O_P\left((\epsilon_{n_\tr}+n^{-1/2})(\log m)^{3\alpha+2}\right)\,.
$$

Next we relate $t_{\beta,\hat\Gamma}$ to $t_{\beta,\Gamma}$. Under the assumption that $(\epsilon_{n_\tr}+n^{-1/2})(\log m)^{3\alpha+4}=o(1)$, there exists a sequence $\delta_n$ such that $\delta_n=\omega\left((\epsilon_{n_\tr}+n^{-1/2})(\log m)^{3\alpha+2}\right)$ and $\delta_n=o((\log m)^{-2})$.
Define event $\mathcal E = \{\sup_{r,s}|\hat\Gamma(r,s)-\Gamma(r,s)|\le \delta_n\}$.
Part 2 of \Cref{lem:5.1-corr-est} implies that $\mathbb P(\mathcal E^c)=o(1)$.
By the Gaussian-to-Gaussian comparison result (Proposition 2.1 of \cite{chernozhukov2022improved}) we have, on $\mathcal E$,
$$
\mathbb P(W_2\le t_{\beta,\hat\Gamma})\le 1-\beta+ c \delta_n^{1/2}\log m \coloneqq 1-\tilde\beta\,,
$$
where $c$ is a universal constant, and  hence
$$
t_{\beta,\hat\Gamma}\le t_{\tilde\beta,\Gamma}\,.
$$
Finally we arrive at the following inequalities
\begin{align*}
    \mathbb P(W_0\le t_{\beta,\hat\Gamma})\le & \mathbb P(W_1\le t_{\beta,\hat\Gamma}+\delta_n)+\mathbb P(|W_0-W_1|>\delta_n)\\
    \le &\mathbb P(W_1\le t_{\tilde\beta,\Gamma}+\delta_n)+\mathbb P(\mathcal E^c)+o(1)\\
    \le &\mathbb P(W_2\le t_{\tilde\beta,\Gamma}+\delta_n)+o(1)\\
    \le & 1-\tilde\beta+O(\delta_n\sqrt{\log m})+o(1)\\
    = & 1-\beta+o(1)\,,
\end{align*}
where the first inequality follows from union bound, the second from another union bound over $(\mathcal E,~\mathcal E^c)$ and the fact that $t_{\beta,\hat\Gamma}\le t_{\tilde\beta,\Gamma}$ on $\mathcal E$, the third from \Cref{thm:4.3-hd-cv-clt-rand}, the fourth  from the Gaussian anti-concentration (Lemma J.3 of \cite{chernozhukov2022improved}),  and the last line follows from the choice of $\delta_n$.  A similar argument provides the other direction and concludes the proof.
\end{proof}

\Cref{thm:5.1-simul-conf-set-naive} leads to a confidence set $\hat\Theta$ satisfying \eqref{eq:5.1-conf-set-def-rand}:
\begin{equation}\label{eq:5.1-cvc-naive}
\hat\Theta_0=\left\{r\in[m]:~\hat R_{\cv,n,K}^{(r)}-\frac{\hat\sigma_{n_\tr}^{(r)}}{\sqrt{n}}t_{\beta,\hat\Gamma_{n_\tr}}\le \min_{s\in[m]}\hat R_{\cv,n,K}^{(s)}+\frac{\hat\sigma_{n_\tr}^{(s)}}{\sqrt{n}}t_{\beta,\hat\Gamma_{n_\tr}}\right\}\,.
\end{equation}
In other words, this confidence set includes all models whose confidence lower bound is not dominated by all the confidence upper bounds, because otherwise we are $1-\beta$ confident that there exists an $s$ such that $\bar R_{\cv,n,K}^{(s)}<\bar R_{\cv,n,K}^{(r)}$, as witnessed by the $s$ that violates the inequality in \eqref{eq:5.1-cvc-naive}. 
\begin{proposition}
    \label{pro:5.1-coverage-cvc-naive}
    Under the same conditions in \Cref{thm:5.1-simul-conf-set-naive}, we have
    $$\mathbb P\left(r^*\in\hat\Theta_0\right)\ge 1-\beta+o(1)\,,$$
    for all $r^*\in \arg\min_r \bar R_{\cv,n,K}^{(r)}$.
\end{proposition}
\begin{remark}
    \label{rem:5.1-non-uniqueness}
    The model confidence set given in \eqref{eq:5.1-cvc-naive} and its theoretical guarantee, \Cref{pro:5.1-coverage-cvc-naive}, do not require the minimizer of $\bar {\mathbf R}_{\cv,n,K}$ to be unique.  The interpretation of the coverage guarantee in \Cref{pro:5.1-coverage-cvc-naive} is that it holds marginally for each individual $r^*$ that minimizes $\bar R_{\cv,n,K}^{(r)}$.
\end{remark}

\paragraph{Model Confidence Set Based on Differences}
The model confidence set $\hat\Theta_0$ given in \eqref{eq:5.1-cvc-naive} is often too conservative as it compares the least optimistic prediction of $\bar R_{\cv,n,K}^{(s)}$ with the most optimistic prediction of $\bar R_{\cv,n,K}^{(r)}$. Here the term ``most optimistic'' corresponds to the lower bound of the simultaneous confidence set and vice versa.  In practice, the correlation between the loss functions calculated from different estimators tends to be positive, as they are targeting the same task of estimating the parameter. Therefore, it is highly unlikely that a large positive fluctuation of the CV risk estimate in one estimator is associated with a large negative fluctuation in another. 

To take into account the likely positive correlation between the cross-validated risks, one approach is to consider the difference of CV risks and loss functions. 
For a given $r$, we want to test whether $r\in\arg\min_s \bar R_{\cv,n,K}^{(s)}$, which is equivalent to test $$\max_{s\neq r} \bar R^{(r)}_{\cv,n,K}-\bar R^{(s)}_{\cv,n,K}>0\,.$$ To this end, consider the difference loss function
\begin{equation}
    \label{eq:5.1-diff-loss}
    \ell^{(r,s)}(X_i,\D_{n,-k_i})\coloneqq \ell^{(r)}(X_i,\D_{n,-k_i})-\ell^{(s)}(X_i,\D_{n,-k_i})\,.
\end{equation}
The difference based model confidence is obtained by constructing a confidence lower bound for $$\max_{s\neq r} \bar R_{\cv,n,K}^{(r)}-\bar R_{\cv,n,K}^{(s)}$$ by applying \Cref{thm:4.3-hd-cv-clt-rand} to the $(m-1)$-dimensional vectors $(\ell^{(r,s)}(X_i,\D_{n,-k_i}):s\neq r)$.
We include $r$ in the model confidence set $\hat\Theta$ if and only if this confidence lower bound is non-positive.

Similarly define $\ell_n^{(r,s)}(X_0)=\mathbb E[\ell^{(r,s)}(X_0,\D_n)|X_0]$, $(\sigma_n^{(r,s)})^2=\var(\ell_n^{(r,s)}(X_0))$, and $\Gamma_n^{(r)}$ the $(m-1)\times (m-1)$ correlation matrix whose $(s,s')$-entry is ${\rm Corr}(\ell_n^{(r,s)}(X_0),\ell_n^{(r,s')}(X_0))$.

Let $\hat\sigma_{n_\tr}^{(r,s)}$ be the empirical standard deviation estimated from $\{\ell^{(r,s)}(X_i,\D_{n,-k_i}):1\le i\le n\}$, and $\hat\Gamma_{n_\tr}^{(r)}$ the corresponding $(m-1)\times (m-1)$ empirical correlation matrix.

Since we only need a confidence lower bound for $\max_{s\neq r} \bar R_{\cv,n,K}^{(r)}-\bar R_{\cv,n,K}^{(s)}$, we only need the quantile of the maximum of a one-sided Gaussian vector.  For a positive semidefinite matrix $\Gamma$, define
\begin{equation}
    \label{eq:5.1-gauss-quantile-one-side}
    u_{\beta,\Gamma} = \text{upper }\beta\text{-quantile of } \max N(0,\Gamma)\,.
\end{equation}
Then we have the following result.

\begin{proposition}
    \label{pro:5.1-cvc-diff}
    For $r\neq s\in[m]$, if \eqref{eq:4.3-nabla_ell_sw} and \eqref{eq:4.3-4th-moment} hold with $\ell^{(r)}(\cdot,\cdot)$ and $\sigma_n^{(r)}$ replaced by
    $\ell^{(r,s)}(\cdot,\cdot)$ and $\sigma_n^{(r,s)}$, respectively, and \eqref{eq:5.1-corr-est-condition} holds, then, for each $r$
    $$
    \max_{s\neq r}\hat R_{\cv,n,K}^{(r)}-\hat R_{\cv,n,K}^{(s)}-\frac{\hat\sigma_{n_\tr}^{(r,s)}}{\sqrt{n}}u_{\beta,\hat\Gamma_{n_\tr}^{(r)}}
    $$
    is an asymptotic $1-\beta$ confidence lower bound for $\max_{s\neq r}\bar R_{\cv,n,K}^{(r)}-\bar R_{\cv,n,K}^{(s)}$.

    Moreover, the confidence set 
    \begin{equation}
        \label{eq:5.1-cvc}
        \hat\Theta=\left\{r\in[m]:~\max_{s\neq r}\hat R_{\cv,n,K}^{(r)}-\hat R_{\cv,n,K}^{(s)}-\frac{\hat\sigma_{n_\tr}^{(r,s)}}{\sqrt{n}}u_{\beta,\hat\Gamma_{n_\tr}^{(r)}}\le 0\right\}
    \end{equation}
    satisfies $\mathbb P(r^*\in\hat\Theta)\ge 1-\beta+o(1)$ for all $r^*\in\arg\min_s \bar{R}_{\cv,n,K}^{(s)}$.
\end{proposition}
The proof of \Cref{pro:5.1-cvc-diff} is almost identical to that of \Cref{thm:5.1-simul-conf-set-naive} and is omitted.



\begin{remark}\label{rem:5.1-difference_loss_standardize}
As we mentioned in the motivation of the difference based CV model confidence set, the loss functions $\ell^{(r)}(\cdot,\cdot)$ and $\ell^{(s)}(\cdot,\cdot)$ are often positively correlated. If the estimators $\hat f^{(r)}$ and $\hat f^{(s)}$ are very similar, such as the same algorithm with two slightly different tuning parameter values, then the correlation can even be close to $1$. Such a high correlation may lead to a very small standard deviation  $\sigma_n^{(r,s)}$ of the difference loss function, and hence pose additional challenge for the stability condition \eqref{eq:4.3-nabla_ell_sw} to hold.  Nevertheless, \eqref{eq:4.3-nabla_ell_sw} is still conceptually similar to the original version in that it requires an $o(1/\sqrt{n})$ factor after taking a $\nabla_i$ operator in the parameter estimation part. We will provide an example of the difference loss stability in \Cref{subsec:6.3-stab-diff}.
\end{remark}

\begin{lemma}
    [Correlation Estimation Accuracy]\label{lem:5.1-corr-est}
    Under the setting as in \Cref{thm:4.3-hd-cv-clt-rand}, assuming \eqref{eq:4.3-nabla_ell_sw} and \eqref{eq:4.3-4th-moment} such that \eqref{eq:5.1-corr-est-condition} holds.  Then 
    \begin{enumerate}
        \item $
    \max_{r}\left|\frac{\hat\sigma_{n_\tr}^{(r)}}{\sigma_{n_\tr}^{(r)}}-1\right|=O_P\left((\epsilon_{n_\tr}+n^{-1/2})(\log m)^{2\alpha+1}\right)\,.
    $
 \item  
    $
\max_{r,s}\left|\hat\Gamma_{n_\tr}(r,s) - \Gamma_{n_\tr}(r,s)\right|=O_P\left((\epsilon_{n_\tr}+n^{-1/2})(\log m)^{2\alpha+1}\right)\,.$
    \end{enumerate}
\end{lemma}
\begin{proof}
    [Proof of \Cref{lem:5.1-corr-est}]
    For any $r\in[m]$ and $i\in[n]$, we use simplified notation
$\ell_{i,r}=\ell^{(r)}(X_i,\D_{n,-k_i})$, $\bar\ell_{i,r}=\bar\ell_{n_\tr}^{(r)}(X_i)$, $\hat R_{\cv,n,K}^{(r)}=\hat R_r$.
For $r,s\in[m]$, define
$\tilde \rho_{rs} = \frac{1}{n}\sum_{i=1}^n \bar\ell_{i,r}\bar\ell_{i,s}$
and $\rho_{rs} = \mathbb E \tilde \rho_{rs} = {\rm Cov}(\bar\ell_{1,r},\bar\ell_{1,s})$.  By definition, $\rho_{rr}=(\sigma_{n_\tr}^{(r)})^2$.

Let $\hat\rho_{rs}=\frac{1}{n}\sum_{i=1}^n(\ell_{i,r}-\hat R_r)(\ell_{i,s}-\hat R_s)$ then
$$
\hat\Gamma(r,s) = \frac{\hat\rho_{rs}}{\sqrt{\hat\rho_{rr}\hat\rho_{ss}}}
$$

Condition \eqref{eq:4.3-4th-moment} and \Cref{pro:2.5-sum-prod-sw} imply that $\bar\ell_{i,r}\bar\ell_{i,s}$ is $(c_\alpha \kappa^2 \sqrt{\rho_{rr}\rho_{ss}},2\alpha)$-SW.
Then \Cref{thm:2.5-sw-mcdiarmid} implies that 
$\frac{\tilde\rho_{rs}-\rho_{rs}}{\sqrt{\rho_{rr}\rho_{ss}}}$ is $(c_{\alpha,\kappa} n^{-1/2},2\alpha+1/2)$-SW for a constant $c_{\alpha,\kappa}$ depending only on $(\alpha,\kappa)$.

    Next we control $\tilde\rho_{rs}-\hat\rho_{rs}$ where $\hat\rho_{rs}=n^{-1}\sum_{i=1}^n(\ell_{i,r}-\hat R_r)(\ell_{i,s}-\hat R_s)$ is the empirical estimate of $\rho_{rs}$.  Using a decomposition
\begin{align*}
&(\ell_{i,r}-\hat R_r)(\ell_{i,s}-\hat R_s)-\bar\ell_{i,r}\bar\ell_{i,s}\\
=&(\ell_{i,r}-\hat R_r-\bar\ell_{i,r})(\ell_{i,s}-\hat R_s)+\bar\ell_{i,r}(\ell_{i,s}-\hat R_s-\bar\ell_{i,s})\,,
\end{align*}
by Part 1 of \Cref{lem:4.4-useful-fact-loss-stab} we have, for each $r$,
 $\ell_{i,r}-\hat R_r-\bar\ell_{i,r}$ is $(2\sqrt{\rho_{rr}}\epsilon_{n_\tr},\alpha+1/2)$-SW.
As a result, by \Cref{pro:2.5-sum-prod-sw},
$\ell_{i,r}-\hat R_r$ is $((\kappa+2\epsilon_{n_\tr})\sqrt{\rho_{rr}},\alpha+1/2)$-SW, and
therefore by \Cref{pro:2.5-sum-prod-sw},
$$
\frac{\tilde\rho_{rs}-\hat\rho_{rs}}{\sqrt{\rho_{rr}\rho_{ss}}}~~
\text{is}~~(c_\alpha \epsilon_{n_\tr}(\kappa+\epsilon_{n_\tr}),~2\alpha+1)\text{-SW}\,.
$$
Thus for $n$ large enough, we have
$$
\frac{\hat\rho_{rs}-\rho_{rs}}{\sqrt{\rho_{rr}\rho_{ss}}}~~\text{is}~~(c_{\alpha,\kappa}(\epsilon_{n_\tr}+n^{-1/2}),2\alpha+1)\text{-SW}\,.
$$

By \Cref{lem:2.5-sw-maximal} we have
\begin{equation}\label{eq:4.4-cov-est-err-sw}
\max_{r,s}\frac{|\hat\rho_{rs}-\rho_{rs}|}{\sqrt{\rho_{rr}\rho_{ss}}}~~\text{is}~~(c_{\alpha,\kappa}(\epsilon_{n_\tr}+n^{-1/2})(\log m)^{2\alpha+1},~2\alpha+1)\text{-SW}\,.
\end{equation}
The first part of the claimed result follows from the fact that $|\sqrt{x}-1|\le |x-1|$ for all $x\ge 0$.

Now we prove the second part. Let $\delta_n= (\epsilon_{n_\tr}+n^{-1/2})(\log m)^{2\alpha+1}=o(1)$,
\begin{align*}
     &|\hat\Gamma_{n}(r,s)-\Gamma_{n}(r,s)|=\left|\frac{\hat\rho_{rs}}{\sqrt{\hat\rho_{rr}\hat\rho_{ss}}}-\frac{\rho_{rs}}{\sqrt{\rho_{rr}\rho_{ss}}}\right|\\
    \le & \frac{|\hat\rho_{rs}-\rho_{rs}|}{\sqrt{\rho_{rr}\rho_{ss}}} + \frac{|\hat\rho_{rs}|}{\sqrt{\rho_{rr}\rho_{ss}}}\left(\sqrt{\frac{\rho_{rr}\rho_{ss}}{\hat\rho_{rr}\hat\rho_{ss}}}-1\right)\\
    \le & O_P(\delta_n) + \left(1+\frac{|\hat\rho_{rs}-\rho_{rs}|}{\sqrt{\rho_{rr}\rho_{ss}}}\right)\left(\sqrt{\frac{\rho_{rr}\rho_{ss}}{\hat\rho_{rr}\hat\rho_{ss}}}-1\right)\\
     = & O_P(\delta_n)\,.\qedhere
\end{align*}
\end{proof}
\subsection{Consistent Subset Selection Using CV Confidence Sets}\label{subsec:5.2-cvc-consist}
What can we do with a model confidence set as constructed in either \eqref{eq:5.1-cvc-naive} or \eqref{eq:5.1-cvc}?  In general, model confidence sets can be used in the following three ways.
\begin{enumerate}
    \item \emph{Model elimination}.  The coverage guarantee of model confidence sets ensures that we can only keep those in the confidence set for future exploration.
    \item \emph{Stopping rule in sequential model building}.  For example, in order to determine the number of steps in forward stepwise regression, in each step one can construct a confidence set by comparing the current model and all models with one more variable, then if the current model is included in the confidence set, it means going one more step will not significantly increase the performance and the procedure should stop.
    \item \emph{Improved single model selection}.  In many scenarios, it is straightforward to pick a single element from the model confidence set.  For example, one can pick a model having a particular desired feature suggested by subject knowledge.  In the common situation of nested models, it is often beneficial to select the most parsimonious model in the confidence set.  The motivation is a simple reflection of Occam's razor: With similar performance, the simpler one wins.
\end{enumerate}

In this section, we will focus on a concrete example related to the third item in the list above: improved single model selection.  What we will see is that, under fairly standard conditions in a classical linear subset selection problem, the most parsimonious model in the CV model confidence set can yield consistent subset selection while the ordinary CV fails.  

To proceed, we consider a fixed-dimensional subset selection problem, where the data $\D_n$ consists of $n$ i.i.d. copies of $X=(Y,Z)$ with
$$Y=Z^T f + \zeta\,,$$
where $Z\in\mathbb R^p$ is the covariate vector, $f\in\mathbb R^p$ the regression coefficient, and $\zeta$ a mean-$0$ noise, independent of $Z$.
Let $\mathcal J=\{0,1\}^p$ be the collection of all possible subsets of $[p]$. 
The goal is to recover $J^*$, the support of $f$.  It is well known, as illustrated in \Cref{exa:3.1-simple-example}, that $K$-fold cross-validation that chooses $$\hat J_{\cv}\coloneqq\arg\min \hat R_{\cv,n,K}^{(J)}$$ is generally inconsistent, regardless of the value of $K$.  Instead, consistency requires reversed $K$-fold with a diverging $K$.  The following result shows that, under certain regularity conditions, the most parsimonious element in the CV model confidence set \eqref{eq:5.1-cvc-naive} is consistent for conventional $K$-fold CV.  In particular, let
$$
\hat J_{\rm cvc} = \arg\min\{|J|:J\in\hat\Theta_0\}\,.
$$
Here the subindex ``cvc'' refers to ``\textbf{C}ross-\textbf{V}alidation with \textbf{C}onfidence'' \citep{lei2020cross}.  By construction, it is straightforward to check that $\hat J_{\cv}$ is always contained in $\hat\Theta_0$ (and $\hat\Theta$) as long as $t_{\beta,\hat\Gamma_{n_\tr}}\ge 0$.  Therefore, $\hat J_{\rm cvc}$ is guaranteed to be at least as parsimonious as $\hat J_{\cv}$, and can be viewed as a correction of $\hat J_{\cv}$, which is known to be prone to overfitting in practice. In this regard, $\hat J_{\rm cvc}$ is a principled approach to reflect a common intuition about CV: The ordinary CV tends to select a slightly over-fitting model, so the best model is usually achieved by one that is slightly more regularized than the ordinary CV choice.

\begin{proposition}
\label{pro:5.2-cvc-consist}
Assume the linear regression model with a fixed distribution of $(Z,\zeta)$ that does not change with $n$, such that $\Sigma_Z=\mathbb E (Z Z^T)$ is full-rank, $\sigma_\zeta^2=\mathbb E \zeta^2>0$, and all entries of $Z$ and $\zeta$ are $(1,\alpha)$-SW for an $\alpha>0$. Let $\ell(x,\hat f)=(y-z^T\hat f)^2$ be the squared loss.  If, for each $J$, $\hat f^{(J)}(\D_n)$ is supported only on $J$ and satisfies the stability condition
\begin{equation}
    \label{eq:5.2-reg-stab}
    \sqrt{n}\|\nabla_1 \hat f^{(J)}(\D_n)\|\text{ is }(\epsilon_n,\alpha)\text{-SW}
\end{equation} for a sequence $\epsilon_n=o(1)$;
and for each $J\supseteq J^*={\rm supp}(f)$
\begin{equation}\label{eq:5.2-root-n-consist}
\sqrt{n}\|\hat f^{(J)}(\D_n)-f\|\text{ is }(1,\alpha)\text{-SW}\,.
\end{equation}
If $J^*=\arg\min_J \mu_n^{(J)}$ for all $n$ large enough, and $\beta=\beta_n$ satisfies $\beta_n=o(1)$ and $\beta_n>1/n$ then
$$
\mathbb P(\hat J_{\rm cvc}=J^*)\rightarrow 1\,.
$$
\end{proposition}
Here, the inner $\|\cdot\|$ denotes the Euclidean norm of a $p$-dimensional vector.

\begin{remark}
    \label{rem:5.2-cvc-consist}
    \Cref{pro:5.2-cvc-consist} is stated with a set of pretty strong but reasonable conditions to deliver the desired model selection consistency result. These conditions can be verified in the setting of \Cref{exa:3.1-simple-example} to show that the most parsimonious element in the confidence set can lead to consistent (or bounded error probability if $\beta$ is a fixed constant) model selection with a standard $K$-fold cross-validation scheme.  Nearly all the conditions can be relaxed to some extent with lengthier bookkeeping and/or more refined arguments.  We list some possible relaxations below.
    \begin{enumerate}
        \item The dimensionality $p$, the smallest non-zero entry of $f$, and the eigenvalues of $\Sigma_Z$, can be allowed to change with $n$ at slow enough rates.  These may require faster rates of the stability term.  In particular, let $\lambda_{\min}$ and $f_{\min}$ be the smallest eigenvalue of $\Sigma_Z$ and the smallest non-zero entry of $f$, respectively. Instead of \eqref{eq:5.2-root-n-consist}, now assume that $\sqrt{n}\|\hat f^{(J)}-f\|$ is $(\sqrt{p},\alpha)$-SW for all $J\supseteq J^*$.
        We will need $(n^{-1/2}+\epsilon_n)p^{\alpha+1/2}+p/n=o(\lambda_{\rm min} f_{\rm min}^2)$ to ensure \eqref{eq:5.2-exclude}, and 
        $p^{1/2}\epsilon_n=o(1)$ to establish \eqref{eq:5.2-Rcv-gap-good}.
        \item The stability condition can be relaxed to $L_2$-norm instead of sub-Weibull, with the addition of the uniform integrability as in \Cref{thm:4.1-var-est} (which implies the Lindeberg condition \eqref{eq:4.1-lindeberg}).
        \item The requirement of $J^*=\arg\min\mu_n^{(J)}$ can be relaxed to $$\mu_{n}^{(J^*)}\le \min_J \mu_n^{(J)}+O((n\log m)^{-1/2})$$ using the Gaussian anti-concentration inequality.
    \end{enumerate}
    Moreover, it is also possible to obtain a similar consistency of $\hat J_{\rm cvc}$ for the difference-based CV model selection confidence set $\hat\Theta$ in \eqref{eq:5.1-cvc}.  The conditions and arguments would be similar to those of \Cref{pro:5.2-cvc-consist}, with everything stated for the difference of the estimated parameters.
\end{remark}

Now we examine the practical plausibility of the two conditions \eqref{eq:5.2-reg-stab} and \eqref{eq:5.2-root-n-consist}.  First, \eqref{eq:5.2-root-n-consist} corresponds to the standard root-$n$-consistency for correctly specified linear regression models.  The maximum likelihood estimator and least squares estimator often satisfy this in the fixed-$p$ setting.  To examine the stability condition \eqref{eq:5.2-reg-stab}, let $\hat f$ be an estimate of $f$ obtained from $\D_n$ using the least squares method. For notational convenience we will consider $\nabla_n$ instead of $\nabla_1$. Let $\hat\Sigma_{n}$ be the sample covariance matrix from $\D_{n}$, and $\hat\Gamma_{n}=n^{-1}\sum_{i=1}^n Z_i Y_i$.
Then we have, by the Sherman-Morrison formula,
\begin{align*}
\hat f = & \hat\Sigma_n^{-1}\hat\Gamma_n = \left(\frac{n-1}{n}\hat\Sigma_{n-1}+\frac{1}{n}Z_n Z_n^T\right)^{-1}\left(\frac{n-1}{n}\hat\Gamma_{n-1}+\frac{1}{n}Z_n Y_n\right)\\
=& \left(\frac{n}{n-1}\hat\Sigma_{n-1}^{-1}-\frac{\frac{n}{(n-1)^2}\hat\Sigma_{n-1}^{-1}Z_n Z_n^T\hat\Sigma_{n-1}^{-1}}{1+\frac{n-1}{n^2}Z_n^T\hat\Sigma_{n-1}^{-1}Z_n}\right)\left(\frac{n-1}{n}\hat\Gamma_{n-1}+\frac{1}{n}Z_n Y_n\right)\,,
\end{align*}
where only the second terms in each factor involve $(Z_n,Y_n)$, both of which are of order $O_P(1/n)$. Thus $\nabla_n \hat f = O_P(n^{-1})$, making \eqref{eq:5.2-reg-stab} plausible for any $\epsilon_n\gg n^{-1/2}$.
\begin{proof}
    [Proof of \Cref{pro:5.2-cvc-consist}]
    Let $\mathcal J_0=\{J\subseteq[p]:J^*\nsubseteq J\}$ and $\mathcal J_1=\mathcal J\setminus \mathcal J_0$.

   For an arbitrary $J\in\mathcal J$, using $\hat f$ and $\hat f'$ to denote $\hat f^{(J)}(\D_n)$ and $\hat f^{(J)}(\D_n^1)$, respectively, we have
    $$\ell(X_0,\hat f)=\zeta_0^2-2\zeta_0 Z_0^T \hat f+(\hat f-f)^T Z_0Z_0^T(\hat f-f)\,,$$
    $$
    R(\hat f) = \sigma_{\zeta}^2 +(\hat f-f)^T \Sigma_Z(\hat f-f)\,,
    $$
    and
    \begin{align*}
        \nabla_1 \ell^{(J)}(X_0,\D_n)= -2\zeta_0 Z_0^T(\nabla_1\hat f)+(\nabla_1 \hat f)^T Z_0Z_0^T(\hat f-f)+(\hat f'-f)^TZ_0Z_0^T(\nabla_1\hat f)\,,
    \end{align*}
    which is $(cn^{-1/2}\epsilon_n,\alpha')$-SW for some universal constant $c$ and $\alpha'$ depending only on $\alpha$.

    Therefore, the $K$-fold version of \Cref{thm:2.5-loocv-concentration-sw} implies that $\hat R_{\cv,n,K}^{(J)}=\mu_{n_\tr}^{(J)}+O_P(n^{-1/2}+\epsilon_n)$ for each $J$.
    For $J\in \mathcal J_0$, the corresponding $\hat f^{(J)}$ satisfies
$\mu_{n_\tr}^{(J)}\ge \var(\zeta_0)+ C$ for a positive constant $C$ depending on the minimum eigenvalue of $\Sigma_Z$ and the minimum nonzero entry of $f$, while $\mu_{n_\tr}^{(J^*)}\le \var(\zeta_0)+ O(1/n_\tr)$.
As a result, with probability tending to $1$ we have
\begin{equation}\label{eq:5.2-Rcv-gap-bad}
\hat R_{\cv,n,K}^{(J)}-\frac{\hat\sigma_{n_\tr}^{(J)}}{\sqrt{n}}t_{\beta_n,\hat\Gamma_{n_\tr}}>\hat R_{\cv,n,K}^{(J^*)}+\frac{\hat\sigma_{n_\tr}^{(J^*)}}{\sqrt{n}}t_{\beta_n,\hat\Gamma_{n_\tr}}\,,
\end{equation}
    because the fluctuation $\hat\sigma_{n_\tr}^{(J)} t_{\beta_n,\hat\Gamma_{n_\tr}}/\sqrt{n}$ in the confidence set is of order at most $\log(p/\beta_n)/\sqrt{n}$, which is not enough to override the gap between $\mu_{n_\tr}^{(J)}$ and $\mu_{n_\tr}^{(J^*)}$. As a result we have
    \begin{equation}\label{eq:5.2-exclude}
\mathbb P(\hat\Theta_0\cap \mathcal J_0\neq \emptyset)=o(1)\,.\end{equation}

Now for $J\in \mathcal J_1$, we have, letting $\hat f^1=\hat f(\D_n^1)$,
$$
\nabla_1 R(\hat f)=(\nabla_1 \hat f)^T\Sigma_Z (\hat f-f)+(\hat f^1-f)^T\Sigma_Z(\nabla_1\hat f)
$$
which is $(cn^{-1}\epsilon_n,\alpha')$-SW.  Thus using the deterministic version of \Cref{thm:5.1-simul-conf-set-naive} and \Cref{thm:4.3-hd-cv-clt-det-1},
we have
\begin{equation}\label{eq:5.2-Rcv-gap-good}
\mathbb P\left[\left|\hat R^{(J)}_{\cv,n,K}-\mu^{(J)}_{n_\tr}\right|\le \frac{\hat\sigma^{(J)}_{n_\tr}}{\sqrt{n}}t_{\beta_n,\hat\Gamma_{n_\tr,1}}~~\forall~J\in\mathcal J_1\right]\ge 1-\beta_n+o(1)\,,
\end{equation}
where $\hat\Gamma_{n_\tr,1}$ is the submatrix of $\hat\Gamma_{n_\tr}$ restricted to $\mathcal J_1$.

Because $t_{\beta_n,\hat\Gamma_{n_\tr}}\ge t_{\beta_n,\hat\Gamma_{n_\tr,1}}$, \eqref{eq:5.2-Rcv-gap-good} also holds for $t_{\beta_n,\hat\Gamma_{n_\tr}}$. Now combine this with  \eqref{eq:5.2-exclude} we conclude that $\mathbb P(\hat J_{\rm cvc}=J^*)=1-o(1)$.
\end{proof}

\subsection{Testing Many Means}\label{subsec:5.3-argmin}
In this subsection, we demonstrate an application of the tools developed in \Cref{sec:4-clt} in a prototypical inference problem at the intersection of high-dimensional inference, multiple comparison, and selective inference.

Let $(X_i:1\le i\le n)$ be an i.i.d. sample in $\mathbb R^m$ with $\mathbb E X_1=\vartheta=(\vartheta_1,...,\vartheta_m)^T$.
We are interested in testing 
\begin{equation}
    \label{eq:5.3-many-means}
    H_0: ~\max\vartheta\le 0\,,~~~\text{vs}~~~ H_a:~\max \vartheta>0\,.
\end{equation}

This problem has many connections to other well-studied problems. Here we briefly mention two examples. A more comprehensive literature review is given in \Cref{subsec:5.6-bib}.  One example is marginal regression, where we have a response variable and an $m$-dimensional covariate vector.  Let each $X_{i,r}$ be the product of the response and the (centered) $r$th covariate in the $i$th sample point; then the problem \eqref{eq:5.3-many-means} corresponds to testing whether any covariate has positive correlation with the response, and it can be extended to testing any non-zero correlation. Such a test is useful, for example, in deciding the stopping rule in forward stepwise regression. Another example is evaluation of many fitted prediction models.  Assume there are $m$ regression or classification probability functions, fitted from external training data.  We want to compare their performance on a target data distribution from which we have $n$ i.i.d. samples points.  Then each $X_{i,r}$ corresponds to the realized loss function of the $r$th prediction function on the $i$th sample point. The testing problem \eqref{eq:5.3-many-means} corresponds to testing whether any of these $m$ prediction models are better than a benchmark, assuming, after re-centering, that the benchmark has $0$ risk.

Our proposed solution to the testing problem \eqref{eq:5.3-many-means} is based on constructing a confidence lower bound for the parameter $\theta=\max\vartheta$ and shares some similarity with the problem of \emph{selective inference}. To begin with, observe that \eqref{eq:5.3-many-means} is equivalent to
$$
H_0:~\vartheta_{s^*}\le 0\,~~~\text{vs}~~ H_a:~\vartheta_{s^*}>0\,,
$$
where
$$s^*=\arg\max_{s\in[m]}\vartheta_s$$
or an arbitrary element in $\arg\max_{s\in[m]}\vartheta_s$ if the maximum is achieved at multiple entries.

The optimal index $s^*$ is usually unknown and can be regarded as an auxiliary parameter (or nuisance parameter), the knowledge of which would make the inference problem much easier.  If $s^*$ were known, then one can apply any one-sample location test such as Student's $t$-test or its robust versions to solve \eqref{eq:5.3-many-means}.  When $s^*$ is unknown, the ordinary one-sample location test applied to $\hat s$, defined as the argmax of the sample mean $\hat\vartheta$, will generally lead to biased inference due to the ``double-dipping'' or ``selection bias'' issue, meaning that the same data is used to select the index $\hat s$ and to compute the test statistic. Thus any valid inference must take into account the randomness in $\hat s$. Existing approaches include the following.
\begin{itemize}
    \item \emph{Selective Inference.}  This framework aims at making the $p$-values conditionally valid on the event $\hat s^*=s$ for each $s\in[m]$. These methods typically require non-trivial assumptions on the data distribution and the power can be less than ideal due to the conditioning. 
    \item \emph{Simultaneous Inference.}  This framework relies on controlling the simultaneous fluctuation of the sample mean vector $\hat\vartheta$, such as in the high-dimensional Gaussian comparison result in a similar fashion as in \Cref{pro:5.1-cvc-diff}.  This approach does not require conditioning but the worst-case bound could be loose if there are many ``obviously bad'' coordinates.
\end{itemize}

Our approach aims at striking a balance between these two: We want to focus on the few entries that showed enough promise to be legitimate candidates for $s^*$, which allows us to avoid having the typical $\sqrt{\log m}$ additional width of the confidence bound as in the simultaneous inference. On the other hand, we also want to avoid conditioning on selection.  The key idea is that conditioning is unnecessary if the selection is stable and
    the test statistic is cross-validated.
In fact, we have more to offer: If the selection is stable and the test statistic is cross-validated, a confidence lower bound on $\max\vartheta$ can be constructed from an asymptotically normal test statistic with a small conservative bias.

\paragraph{Cross-validated max-entry evaluation}
The first idea for fixing the double-dipping issue with data-driven max-entry selection is to split the sample used for selecting $\hat s$ and to evaluate $\vartheta_{\hat s}$. Consider a leave-one-out CV scheme:
\begin{equation}\label{eq:5.3-T-split}
    \hat T_n=\frac{1}{n}\sum_{i=1}^n X_{i,\hat s^{(-i)}}
\end{equation}
where
$$\hat s^{(-i)}=\arg\max_{1\le s\le m}\hat\vartheta_s^{(-i)}\,,~~\text{and}~~\hat\vartheta^{(-i)}=\frac{1}{n-1}\sum_{j\neq i}X_j\,.$$

Unfortunately this $\hat T_n$ is not generally asymptotically normal.  To see why, observe that for each $i$, the quantity $X_{i,\hat s^{(-i)}}$ is an instance of 
$$\ell(X_i,\D_{n,-i})$$
with $$\ell(x,\D_n)=x_{\hat s(\D_n)}$$
and $$\hat s(\D_n)=\arg\max_{s\in[m]}\sum_{i\in[n]} X_{i,s}\,.$$
The stability condition \eqref{eq:4.1-nabla-l2-bound} does not hold for this $\ell(x,\D)$.
In particular, consider the case of $m=2$, $\vartheta_1=\vartheta_2=0$, $X_i\sim N(0, I_2)$. The contribution of sample point $j$ in the comparison between $\hat\vartheta_1^{(-i)}$ and $\hat\vartheta_2^{(-i)}$ is $N(0,2/n)$, which has a probability $\asymp 1/\sqrt{n}$ to flip the sign of $\hat\vartheta_1^{(-i)}-\hat\vartheta_2^{(-i)}$ if $X_j$ is replaced by an i.i.d. copy $X_j'$. When this happens, the resulting change in $X_{i,\hat s^{(-i)}}$ is $N(0,2)$.  Therefore $\nabla_j X_{i,\hat s^{(-i)}}$ is $O_P(n^{-1/2})$. But \eqref{eq:4.1-nabla-l2-bound} requires $\|\nabla_j X_{i,\hat s^{(-i)}}\|_2=o(n^{-1/2})$.  Even such a narrow miss of $o(\cdot)$ to $O(\cdot)$ would lead to a failure of the CLT, as demonstrated in the simulation in \Cref{fig:5.3-hist}, where the histogram of $\hat T_n$ in \eqref{eq:5.3-T-split} is plotted in green, which  clearly deviates from the standard normal density curve (black curve).  This also demonstrates the sharpness of the stability condition \eqref{eq:4.1-nabla-l2-bound} in \Cref{thm:4.1-clt-rand} --- even relaxing $o_P(n^{-1/2})$ to $O_P(n^{-1/2})$ may lead to failure of the CLT!

\begin{figure}
    \centering
    \includegraphics[width=0.7\linewidth]{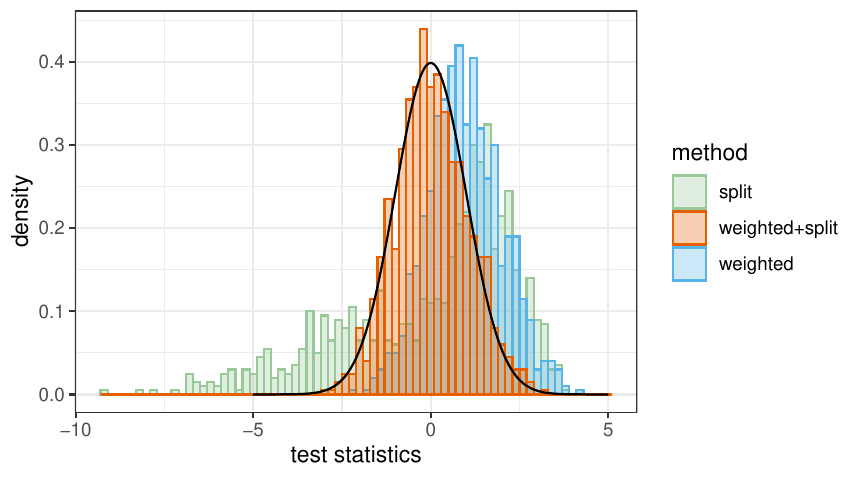}
    \caption{Histograms of standardized test statistics.  Black solid line is the standard normal density curve. Light green: LOOCV in \eqref{eq:5.3-T-split}; Red: LOOCV+stable weights in \eqref{eq:5.3-T-split-exp}; Blue: stable weights only.}
    \label{fig:5.3-hist}
\end{figure}

\paragraph{Stabilizing the maximum value}
The main reason for the instability of $X_{i,\hat s^{(-i)}}$ is that $\hat s^{(-i)}$ may change completely when a single data point is perturbed.
Such an instability of discrete argmax can be treated using techniques developed in differential privacy \citep{dwork2014algorithmic}.  In general, differential privacy requires a statistic to be ``insensitive'' to perturbations at any individual sample point, which is closely related to the first-order stability condition.  Inspired by the ``exponential mechanism'' \citep{mcsherry2007mechanism}, we consider the following cross-validated softmax statistic:
\begin{equation}
    \label{eq:5.3-T-split-exp}
    T_n=\frac{1}{n}\sum_{i=1}^n Q_i\coloneqq \frac{1}{n}\sum_{i=1}^n \sum_{s\in[m]} \hat w_s^{(-i)} X_{i,s}\,,
\end{equation}
with weights
$$
\hat w_s^{(-i)} = \frac{\exp(\lambda \hat\vartheta_s^{(-i)})}{\sum_{t\in[m]}\exp(\lambda \hat\vartheta_t^{(-i)})}\,.
$$
The exponential mechanism stabilizes the maximum value by ``softening'' the indicator of the maximum entry to a weight vector, $\hat w^{(-i)}$, with larger entries receiving more weight.  By construction, the weights are smooth functions of the sample means, and hence enjoy a similar level of stability as the sample means. 
The quantity $Q_i=\sum_{s\in[m]}\hat w_s^{(-i)}X_{i,s}$ can be interpreted as the conditional expectation of $X_{i,\tilde s^{(-i)}}$  given $(X_i,\hat w^{(-i)})$, where  $\tilde s^{(-i)}$ is sampled according to the weight $\hat w^{(-i)}$, which is a smoothed version of $\hat s^{(-i)}$.

The coefficient $\lambda$ is a tuning parameter in the exponential mechanism. A larger $\lambda$ puts more weight on large entries, which has less bias but is less stable. In the extreme case $\lambda=\infty$, $Q_i=X_{i,\hat s^{(-i)}}$.  In the other extreme case $\lambda=0$, $Q_i=m^{-1}\sum_{s\in[m]}X_{i,s}$.   

\begin{theorem}
    \label{thm:5.3-argmin-clt}
    Let $\sigma_n^2={\rm Var}[\mathbb E(Q_i|X_i)]$.
    If $\sqrt{n}\|\nabla_j Q_i\|_2 =o(\sigma_n)$ for $j\neq i$ and \eqref{eq:4.1-lindeberg} holds for $\bar\ell_n(X_i)=\mathbb E(Q_i|X_i)-\mathbb E Q_i$, then
    $$
    \frac{1}{\sqrt{n}\sigma_n}\sum_{i=1}^n(Q_i-R_i)\rightsquigarrow N(0,1)\,,
    $$
    where
    $$R_i=\mathbb E(Q_i|\D_{-i})=\sum_{s\in[m]}\hat w_s^{(-i)}\vartheta_s\,.$$
\end{theorem}
\Cref{thm:5.3-argmin-clt} is a direct consequence of \Cref{thm:4.1-clt-rand}.
The choice of notation $R_i$ has a direct correspondence with $R(\D_{n,-i})$ in the context of \Cref{thm:4.1-clt-rand}.   

The conditions in \Cref{thm:5.3-argmin-clt}, although directly corresponding to those in \Cref{thm:4.1-clt-rand}, are still somewhat abstract in the context of the many means testing problem. In particular, \Cref{thm:5.3-argmin-clt} does not specify which $\lambda$ would lead to the stability condition $\|\nabla_jQ_i\|_2=o(\sigma_n/\sqrt{n})$.  We provide a set of more transparent sufficient conditions in the proposition below.
\begin{proposition}
    \label{pro:5.3-suff-cond-argmin}
    Let $\Sigma=\mathbb E(X_1-\vartheta)(X_1-\vartheta)^T$ be the covariance matrix.
    Assume
    \begin{itemize}
        \item $|X_{i,s}|\le B$ almost surely for some finite constant $B$.
        \item There exist a pair of positive constants $c_1,c_2$ such that $w^T\Sigma w\in[c_1,c_2]$ for all $w\in \mathbb S_{m-1}$, where $\mathbb S_{m-1}$ is the $(m-1)$-dimensional simplex: $\{w\in \mathbb R^m:w\ge 0,~\|w\|_1=1\}$. 
    \end{itemize}
     Then $\sigma_n\asymp 1$ and $\sqrt{n}\|\nabla_j Q_i\|_2=o(\sigma_n)$ whenever $\lambda=\lambda_n=o(\sqrt{n})$.
\end{proposition}

\begin{remark}
    \label{rem:5.3-varying-m}
    In \Cref{thm:5.3-argmin-clt} and \Cref{pro:5.3-suff-cond-argmin}, we allow for the triangular-array setting where $m$ may change with $n$.
    Moreover, all the theoretical results and methods developed below can be extended to $K$-fold CV for any value of $K$.  It would be possible to relax the $L_\infty$-norm bound on $X_{i,s}$ to sub-Weibull type bound using more refined arguments.
\end{remark}

The choice of $\lambda_n=o(\sqrt{n})$ is used to cover the worst-case scenario. In most realistic situations, such as when there are a few coordinates whose $\vartheta_s$ values are substantially larger than the rest, then the stability condition can be satisfied by larger values of $\lambda_n$. We discuss data-driven choices of $\lambda$ below after proving \Cref{pro:5.3-suff-cond-argmin} and describe the test procedure.

\begin{proof}
    [Proof of \Cref{pro:5.3-suff-cond-argmin}]
We first show that $\sigma_n=\Omega(1)$.
In fact $\sigma_n^2=\var[\mathbb E(Q_i|X_i)]=\bar w^T\Sigma\bar w$, where
$\bar w=\mathbb E \hat w^{(-1)}$. By assumption $\sigma_n^2\in[c_1,c_2]$.
The boundedness of $X_{i,s}$ further implies the condition \eqref{eq:4.1-lindeberg} is satisfied for $\ell(X_i,\D_{n,-k_i})=Q_i$.

It suffices to show that $\|\nabla_j Q_i\|_2=O(\lambda_n/n)$.
To simplify notation, we write $\hat w_s^{(-i)}$ as $w_s$, and $\hat w_s^{(-i,j)}$ as $w_s'$, which is the version of $\hat w_s^{(-i)}$ obtained from replacing $X_j$ by $X_j'$. We also write $\hat\vartheta^{(-i)}$ as $\hat\vartheta$, and $\hat\vartheta'$ the corresponding perturb-one version. By construction
$$
|\nabla_j Q_i| = \left|\sum_{s\in[m]}(w_s'-w_s)X_{i,s}\right|\le B \max_s \left|\frac{w_s'}{w_s}-1\right|\,.
$$
Recall that $n_\tr=n-1$ since we are considering LOOCV. By construction we have for each $s\in[m]$
\begin{align*}
\frac{w_s'}{w_s}=&\frac{\exp(\lambda_n\hat\vartheta_s')}{\sum_{t\in[m]}\exp(\lambda_n \hat\vartheta_t')}\cdot\frac{\sum_{t\in[m]}\exp(\lambda_n \hat\vartheta_t)}{\exp(\lambda_n\hat\vartheta_s')}\\
=&\exp(\lambda_n n_\tr^{-1}(X_{j,s}'-X_{j,s}))\frac{\sum_{t\in[m]}\exp(\lambda_n\hat\vartheta_t')\exp(\lambda_n n_\tr^{-1}(X_{j,t}-X_{j,t}'))}{\sum_{t\in[m]}\exp(\lambda_n\hat\vartheta_t')}\\
\le & \exp(4\lambda_n n_\tr^{-1}B)\,.
\end{align*}
Similarly we can show $\frac{w_s'}{w_s}\ge \exp(-4\lambda_n n_\tr^{-1}B)$, and the claimed result follows.
\end{proof}

\begin{remark}
    \label{rem:5.3-necessity-of-both}
    One may wonder whether the exponential weighting trick alone can fix the failure of CLT in the statistic $\hat T_n$ in \eqref{eq:5.3-T-split} without cross-validation. That is, consider $$\tilde T_n=\frac{1}{n}\sum_{i=1}^n\tilde Q_i$$
    with $\tilde Q_i=\sum_{s\in[m]}\tilde w_s X_{i,s}$
    and $\tilde w_s\propto \exp(\lambda \hat\vartheta_s)$.  The empirical distribution of $\tilde T_n$ is plotted as the blue histogram in \Cref{fig:5.3-hist}, showing a clear shift from the desired standard normal curve.  This demonstrates that both cross-validation and exponential weighting are necessary to achieve asymptotic normality for the (randomly centered) statistic $T_n$ in \eqref{eq:5.3-T-split-exp}.
\end{remark}

\paragraph{Inference}
As seen in the proof of \Cref{pro:5.3-suff-cond-argmin}, we have $\sigma_{n}^2\asymp 1$, and $(Q_i-R_i)\in[-2B, 2B]$. Thus the conditions required in \Cref{thm:4.1-var-est} are also satisfied, which leads to the following.
\begin{corollary}\label{cor:5.3-var_est_argmin}
    Let $\hat\sigma_n$ be the empirical variance of $(Q_i:i\in [n])$. Under the conditions in \Cref{pro:5.3-suff-cond-argmin}, $\hat\sigma_n/\sigma_n\rightarrow 1$ in probability, where $\sigma_n^2=\var[\mathbb E(Q_i|X_i)]$.

Moreover, for $\beta\in(0,1)$, let $u_\beta$ be the $(1-\beta)$-quantile of the standard normal, then the statistic $T_n$ defined in \eqref{eq:5.3-T-split-exp} satisfies
    \begin{equation}\label{eq:5.3-max-val-conf-lower-bound}
    \mathbb P\left[\max\vartheta \ge T_n-\frac{\hat\sigma_n}{\sqrt{n}}u_\beta\right]\ge 1-\beta+o(1)\,.
    \end{equation}
\end{corollary}

\begin{proof}[Proof of \Cref{cor:5.3-var_est_argmin}]
The consistency of variance estimation follows directly from \Cref{thm:4.1-var-est}.  

Now we prove the second claim.
    \begin{align*}
       & \mathbb P\left[\max\vartheta \ge T_n-\frac{\hat\sigma_n}{\sqrt{n}}u_\beta\right]\\
  =&  \mathbb P\left[\frac{1}{n}\sum_{i=1}^n Q_i \le \max\vartheta+\frac{\hat\sigma_n}{\sqrt{n}}u_\beta\right]  \\
  =&  \mathbb P\left[\frac{1}{n}\sum_{i=1}^n (Q_i-R_i) \le \max\vartheta -\frac{1}{n}\sum_{i=1}^n R_i +\frac{\hat\sigma_n}{\sqrt{n}}u_\beta\right] \\
  \ge &  \mathbb P\left[\frac{1}{n}\sum_{i=1}^n (Q_i-R_i) \le \frac{\hat\sigma_n}{\sqrt{n}}u_\beta\right] \\
  =&  \mathbb P\left[\frac{1}{\sqrt{n}\hat\sigma_n}\sum_{i=1}^n (Q_i-R_i) \le u_\beta\right]\\
  = & 1-\beta+o(1)\,,
    \end{align*}
    where the inequality follows from the fact that $R_i\le \max\vartheta$, and the last equation follows from \Cref{thm:5.3-argmin-clt}.
\end{proof}

Therefore, the following rejection rule is asymptotically valid for the testing problem \eqref{eq:5.3-many-means} at nominal level $\beta$,
\begin{equation}\label{eq:5.3-rej-rule}
\text{reject } H_0 \text{ if } T_n\ge \frac{\hat\sigma_n}{\sqrt{n}}u_\beta\,,
\end{equation}
provided that $\lambda_n=o(\sqrt{n})$.

\paragraph{Bias and power analysis}
    In the confidence lower bound in \Cref{cor:5.3-var_est_argmin}, the term due to the variance, which is $\hat\sigma_n u_\beta/\sqrt{n}$, does not involve the dimensionality $m$. This is in sharp contrast with the simultaneous inference approach, which, in the case of independent coordinates, would typically involve a $\sqrt{\log m}$ factor associated with the variance term.  Such a reduction in variance is gained at a price of a bias term
    $\max\vartheta - n^{-1}\sum_{i\in[n]}R_i$.  Empirically the cross-validated and exponentially weighted confidence lower bound in \eqref{eq:5.3-max-val-conf-lower-bound} with an adaptively chosen $\lambda$ (see below) usually achieves a favorable bias-variance trade-off when compared to other methods. Here we provide a simple upper bound of this bias under a boundedness condition.
\begin{proposition}\label{pro:5.3-bias}
    If $|X_{i,s}|\le B<\infty$ almost surely for all $i,s$, and $\lambda_n=o(\sqrt{n})$. Then there exists a constant $C$ such that
    $$
    \max\vartheta-n^{-1}\sum_{i=1}^n R_i\le O_P\left(\frac{\log(m+n)}{\lambda_n}\right)\,.
    $$
\end{proposition}    
    \begin{proof}
        [Proof of \Cref{pro:5.3-bias}]
        Without loss of generality, assume $\vartheta_1=\max\vartheta$.
        Let $\mathcal T=\{s\in[m]: \vartheta_s\le \vartheta_1-c\log(m+n)/\lambda_n\}$, for a constant $c$ to be specified later.  Define $\mathcal E$ to be the event $|\hat\vartheta_s^{-i}-\vartheta_s|\le c'\sqrt{\log(m+n)/n}$ for all $i,s$, with a constant $c'$ large enough such that $\mathbb P(\mathcal E^c)\le (m+n)^{-1}$. Then on $\mathcal E$, we have
        $\hat w_s^{(-i)}\le (m+n)^{-2}$ for all $i\in [n]$ and $s\in\mathcal T$ if we pick $C$ large enough in the definition of $\mathcal T$. Therefore, on event $\mathcal E$, we have
        \begin{align*}
            &\max\vartheta-n^{-1}\sum_{i=1}^n R_i\\
            =& \sum_{s\in\mathcal T}n^{-1}\sum_{i=1}^n \hat w_s^{(-i)}(\vartheta_1-\vartheta_s)+\sum_{s\in\mathcal T^c}n^{-1}\sum_{i=1}^n \hat w_s^{(-i)}(\vartheta_1-\vartheta_s)\\
            \le & 2B |\mathcal T|(m+n)^{-2}+C\frac{\log(m+n)}{\lambda_n}\\
            =& O\left(\frac{\log(m+n)}{\lambda_n}\right)\,. \qedhere
        \end{align*}
    \end{proof}
    \Cref{pro:5.3-bias} implies that the test statistic $T_n$ is asymptotically powerful against alternatives with signal size $\omega(\log(m+n)/\lambda_n)$.
    See \Cref{subsec:5.6-bib} for further references on power analysis and numerical results.

\paragraph{Choosing the tuning parameter}
As mentioned in the discussion after \Cref{rem:5.3-varying-m}, 
the requirement of $\lambda_n=o(\sqrt{n})$ in \Cref{pro:5.3-suff-cond-argmin} is for the worst-case scenario.
For an example of a non-worst-case scenario, suppose the largest entry $\vartheta_{s^*}\gg \vartheta_s$ for $s\neq s^*$, the selection of $\arg\max\hat\vartheta_s$ would be stable even without the exponential mechanism.  
In general, the optimal choice of $\lambda_n$ should be as large as allowed by the stability condition, and is often a complex function of the mean vector $\vartheta$ and covariance matrix $\Sigma$.  

Suppose we are given a finite collection $\mathcal L$ of candidate values of $\lambda$, a plausible strategy is to select the largest one in this set such that the stability condition holds approximately.

By the insight in \Cref{pro:4.1-var-approx},  the stability condition we want to check for each $\lambda\in\mathcal L$ is the following:
\begin{equation}
    \label{eq:5.3-approx-stab-cond}
    n\|\nabla_j (Q_i-R_i)\|_2^2\ll \sigma_{n_\tr}^2
\end{equation}
for $j\neq i$, where $\sigma_{n_\tr}^2$ is defined as in \Cref{thm:5.3-argmin-clt}.

The quantity $\sigma_{n_\tr}^2$ can be estimated using the sample variance of $(Q_i:1\le i\le n)$ as in \Cref{thm:4.1-var-est}.  Here we only need to estimate $\|\nabla_j(Q_i-R_i)\|_2^2$.

We use a leave-two-out bootstrap estimate.  For a distinct triplet $(i,j,l)\subset[n]$, 
Let $$\hat\delta_{i,j,l}=\sum_s(\hat w_s^{(-i,-j)}-\hat w_s^{(-i,-l)})(X_{i,s}-\hat\vartheta_s)$$
which mimics $\nabla_j(Q_i-R_i)$.

Thus we estimate $\|\nabla_j(Q_i-R_i)\|_2^2$ by the empirical average of  
$\{\hat\delta_{i^b,j^b,l^b}^2:1\le b\le B\}$
where $\{(i^b,j^b,l^b):1\le b\le B\}$ are independently sampled index triplets.  Here $B$ is the bootstrap sample size. Denote this estimate as $\hat\delta^2(\lambda)$ for each $\lambda\in\mathcal L$.

Finally, we choose $\lambda$ as the largest value in a given candidate set such that,
for some small constant $\varepsilon$ (e.g., $\varepsilon=0.05$),
$$n\hat\delta^2(\lambda)\le \varepsilon \hat\sigma^2(\lambda)\,,$$
where $\hat\sigma^2(\lambda)$ is the corresponding estimate of $\sigma_{n_\tr}^2$.

\subsection{Argmin Confidence Set}\label{subsec:5.4-argmin-conf}
In this subsection, we shift our focus to inference on the location of the minimum of $\vartheta$, rather than the minimum value itself. Specifically, the inference target is $\Theta = \arg\min_s \vartheta_s$, the set of indices that achieve the minimum value of $\vartheta$. This problem is closely related to tasks such as identifying the univariate regression model with the highest marginal correlation or selecting the best prediction function based on a validation dataset. The confidence lower bound for the maximum value of the mean vector, constructed in the previous subsection, can be readily transformed into a confidence set for the argmin of the mean vector.

The idea is similar to the difference-based model confidence set approach introduced in \Cref{subsec:5.1-model-conf-set}.
For each $r\in[m]$, let $\vartheta^{(r)}=(\vartheta_r-\vartheta_s:s\neq r)\in \mathbb R^{m-1}$ and $X_{i,s}^{(r)}=X_{i,r}-X_{i,s}$ for $i\in[n]$, $s\in [m]\setminus\{r\}$.
Thus we construct a confidence set for $\arg\min_s\vartheta_s$ as follows.
$$
\hat\Theta = \{r\in[m]:~H_0^{(r)}\text{ is not rejected}\}\,,
$$
where $H_0^{(r)}$ is the hypothesis testing problem \eqref{eq:5.3-many-means} for the mean vector $\vartheta^{(r)}$ and sample $(X_i^{(r)}:i\in[n])$.
Formally, we have the following.
\begin{corollary}
\label{cor:5.3-argmin-conf-set}
Assume that for all $r\in[m]$, $|X_{i,s}^{(r)}|\le B$ almost surely with a positive constant $B$, and that
$\Sigma=\mathbb E(X_1-\vartheta)(X_1-\vartheta)^T$ has eigenvalues bounded and bounded away from zero by positive constants.  If $\hat\Theta$ is constructed by testing each $H_0^{(r)}$ using the rejection rule \eqref{eq:5.3-rej-rule} corresponding to the mean vector $\vartheta^{(r)}$ and data $(X_i^{(r)}:i\in [n])$.   Then for $\lambda_n=o(\sqrt{n})$ we have
$$
\mathbb P(r\in \hat\Theta)\ge 1-\beta-o(1) ~~~\forall~~r\in\Theta\,.
$$
\end{corollary}
One minor difference between \Cref{cor:5.3-argmin-conf-set} and \Cref{pro:5.3-suff-cond-argmin} is the minimum eigenvalue condition. This is because, letting $r=1$ for simplicity, the variance of $\mathbb E(Q_i^{(r)}|X_i^{(r)})$ is $(\bar w^{(r)})^T\Sigma \bar w^{(r)}$, where
$\bar w^{(r)}=(1,-\bar w_2,-\bar w_3,...,-\bar w_m)^T$  and $\bar w_s$ the expected value of $\hat w_s^{(r,-i)}$ (the counterpart of $\hat w_s^{(-i)}$ for $(\vartheta^{(r)},X_i^{(r)})$).  Thus the variance is bounded and bounded away from zero whenever all eigenvalues of $\Sigma$ are so.

\subsection{Asymptotics of Cross-Conformal Inference}\label{subsec:5.5-cross-conf}
In this subsection, we use the tools developed in \Cref{sec:bousquet} and \Cref{sec:4-clt} to study the asymptotic properties of \emph{cross-conformal prediction}, a variant of split conformal prediction methods.  
In general, conformal prediction is a general framework that uses the exchangeability between sample points to convert arbitrary point estimators into prediction sets that enjoy certain worst-case coverage guarantee under very mild conditions.  To keep the presentation concise and self-contained, we will focus on a regression setting, where each sample point $X_i=(Y_i,Z_i)$ consists of a response $Y_i\in\mathcal Y$ and a covariate $Z_i\in\mathcal Z$. The goal is to construct a prediction set $\hat C\subseteq\mathcal Y\times\mathcal Z$ from an i.i.d. sample $(Y_i,Z_i)_{i=1}^n$ such that, for another i.i.d. sample point $(Y_{n+1},Z_{n+1})$,
\begin{equation}
  \label{eq:5.5-conf-coverage}  
  \mathbb P(Y_{n+1}\in\hat C(Z_{n+1}))\ge 1-\beta\,,
\end{equation}
for a nominal non-coverage level $\beta\in(0,1)$. Here $\hat C(z)\coloneqq\{y\in\mathcal Y:~(y,z)\in\hat C\}$, and hence \eqref{eq:5.5-conf-coverage} is equivalent to 
$$
\mathbb P((Y_{n+1},Z_{n+1})\in\hat C)\ge 1-\beta\,.
$$
Therefore, the coverage guarantee in \eqref{eq:5.5-conf-coverage} is \emph{marginal}, which averages over the randomness of both $(Y_{n+1},Z_{n+1})$ and $\D_n=\{(Y_i,Z_i):1\le i\le n\}$.

One popular method to construct such a $\hat C$ satisfying \eqref{eq:5.5-conf-coverage} is the \emph{split conformal} method, which works as follows.
\begin{enumerate}
\item Let $n_\tr\in[1,n)$ be a \emph{fitting sample size}, and $n_\te = n-n_\tr$ be the corresponding \emph{calibration sample size}.
    \item Split the sample $\D_n$.
    Let $\D_\tr = \D_{n_\tr}$ and $\D_\te = \{(Y_i,Z_i):i\in [n]\setminus[n_\tr]\}$,
     which correspond to the fitting subsample and calibration subsample, respectively.
     \item Fit a function $s:\mathcal Y\times\mathcal Z\mapsto \mathbb R$ using $\D_\tr$.
     \item For $(y,z)\in\mathcal Y\times\mathcal Z$, define
     $$\hat p_{\rm sc}(y,z)=\frac{1}{n_\te+1}\sum_{i\in [n]\setminus[n_\tr]}\mathds{1}[s(Y_i,Z_i)< s(y,z)]+\frac{1}{n_\te+1}\,.$$
     \item  Output 
     \begin{equation*}
     \hat C = \{(y,z): \hat p_{\rm sc}(y,z)\le (n_{\te}+1)^{-1}\lceil(n_
     {\te}+1)(1-\beta)\rceil\}\,.
     \end{equation*}
\end{enumerate}

The function $s(x)$, known in the conformal prediction literature as the ``conformity score function'', measures the level of conformity (or agreement) of a sample point $x$ with respect to the fitting subsample $\D_\tr$.  It usually consists of two components: a fitting part and an evaluation part.  A typical example of $s$ is 
$s(y,z)=|y-\hat f(z)|$, where $\hat f$ is a fitted regression function obtained from $\D_\tr$. It can be shown that, by exchangeability among $\{s(Y_i,Z_i):n_\tr+1\le i\le n+1\}$, $\hat C_{\rm sc}$ satisfies \eqref{eq:5.5-conf-coverage} for an arbitrary choice of $s$.  The quantity $\hat p_{\rm sc}(y,z)$ in Step 4 of the procedure above is called the ``split conformal $p$-value'' (a variant of conformal $p$-value.)

The split conformal method is simple and universal.  However, its practical usefulness is often limited by two concerns: the reduced sample size used for model fitting and the added randomness due to sample splitting. Both concerns can be alleviated by the natural extension to a cross-validated version of split conformal.

The connection between split conformal and cross-validation is almost obvious: we fit a model from the fitting sub-sample and evaluate the average conformity on the calibration sub-sample.  Then it seems natural to follow the cross-validation scheme to repeat this procedure by rotating the sample splitting through $K$ folds and average over the resulting split conformal $p$-values.  This leads to the \emph{cross-conformal prediction}, as formally described below using the notation introduced in \Cref{subsec:prelim}.
\begin{enumerate}
    \item Split the data $\D_n$ into $K$ folds.
    \item For a given $(y,z)\in\mathcal Y\times\mathcal Z$, for each $k\in[K]$,
    compute
    $$\hat p_{k}(y,z)=\frac{1}{n_\te}\sum_{i\in I_{n,k}}\mathds{1}[s(Y_i,Z_i;\D_{n,-k})<s(y,z;\D_{n,-k})]\,,$$
    where $s(\cdot,\cdot;\D)$ is the conformity score function fitted from $\D$.
    \item Let 
    $$
    \hat p_{\rm{cc}}(y,z)=\frac{1}{K}\sum_{k=1}^K\hat p_{k}(y,z)\,.
    $$
    \item Output
    $$\hat C_{{\rm cc}} = \{(y,z):\hat p_{\rm cc}(y,z)\le 1-\beta\}\,.
    $$
\end{enumerate}
In the above description of cross-conformal prediction, we omitted the rounding, ignored the $1/(n_\te+1)$ term in the definition of $\hat p_k(y,z)$, and changed the normalization factor from $1/(n_\te+1)$ to $1/n_\te$. These will only cause negligible differences in the resulting quantities and will disappear in the asymptotic analysis below.  In return, we gain a perfect matching between the cross-conformal $p$-value $\hat p_{\rm cc}$ and a general cross-validation statistic considered in this article.

Due to the dependence caused by the fold rotation, the fold-specific conformal $p$-values $\hat p_k(y,z)$ are often dependent and hence it is hard to claim the finite-sample distribution-free coverage guarantee \eqref{eq:5.5-conf-coverage} for the cross-conformal prediction set $\hat C_{\rm cc}$.  However, in practice it is often observed that cross-conformal prediction sets offer accurate coverage with favorable efficiency (measured as the size of the prediction set) compared to the split conformal prediction set, thanks to a more efficient use of the data.

Here we use the tools developed in \Cref{sec:bousquet} and \Cref{sec:4-clt} to prove an asymptotic coverage guarantee for the cross-conformal prediction under two mild assumptions.  In the rest of this subsection, we will use $x$ to denote $(y,z)$, and $X_i=(Y_i,Z_i)$.
\begin{assumption}
    \label{asm:5.5-conf-score-stab}
    The conformity score function $s(x;\D)$ satisfies the following conditions.
    \begin{enumerate}
        \item [(a)] $s(x;\D_n)$ is symmetric in the entries of $\D_n$ for all values of $n$.
        \item [(b)]  The conditional density of $s(X_0;\D_n)$ given $\D_n$ is upper bounded by a finite constant $c$ almost surely.
    \end{enumerate} 
\end{assumption}

Part (a) of \Cref{asm:5.5-conf-score-stab} is similar to the symmetry condition required for the loss function $\ell(X_0,\D_n)$ in \Cref{sec:bousquet} and \Cref{sec:4-clt}.  The bounded density condition in Part (b) is needed to establish the stability of the indicator function $\mathds{1}[s(X_i;\D_{n,-k})<s(x;\D_{n,-k})]$.  The bounded density condition can be satisfied if $X=(Y,Z)$ and $Y=f(Z)+\epsilon$ where $\epsilon$ is an independent noise with bounded density.

\begin{theorem}
    \label{thm:5.5-cross-conf}
    Under \Cref{asm:5.5-conf-score-stab}, if
       $\|\nabla_1 s(X_0;\D_n)\|_q\le \epsilon_n$ for $q\ge 1$ and a sequence $\epsilon_n$, then
    the cross-conformal $p$-value satisfies
    $$
    d_{W_2}\left[\hat p_{\rm cc}(X_{n+1})\,,~U(0,1)\right]\lesssim n^{-1/2}+\epsilon_n^{q/(2+2q)}+n^{1/2}\epsilon_n^{q/(2+q)}\,,
    $$
    where $d_{W_2}(\cdot,\cdot)$ is the Wasserstein-2 distance between two probability distributions.

    As a consequence, if
$\epsilon_n=o(n^{-\frac{1}{2}-\frac{1}{q}})$ then
$$\hat p_{\rm cc}(Y_{n+1},Z_{n+1})\rightsquigarrow U(0,1)\,.$$ 
\end{theorem}

In practice, it is often the case that $\epsilon_n\gtrsim n^{-1}$, thus the condition $\epsilon_n=o(n^{-1/2-1/q})$ is plausible for $q>2$.
The requirement of $\epsilon_n=o(n^{-1/2-1/q})$ can be further relaxed if we strengthen the stability condition to $\nabla_1 s(X_0;\D_n)$ being $(\epsilon_n,\alpha)$-sub-Weibull. In this case it suffices to require $\epsilon_n=\tilde o(n^{-1/2})$, where $\tilde o(\cdot)$ means ``$o(\cdot)$ up to a polylog factor''.

\begin{proof}[Proof of \Cref{thm:5.5-cross-conf}]
Define
$$Q(x,\D)=\mathbb P\left[s(X_0;\D)<s(x;\D)\bigg| x,\D\right]\,.$$
Intuitively $Q(X_{n+1},\D_{n,-k})$ plays the role of $R(\D_{n,-k})$ as in \Cref{thm:4.1-clt-rand}, by treating $\mathds{1}[s(X_i;\D_{n,-k})<s(X_{n+1};\D_{n,-k})]$ as $\ell(X_i,\D_{n,-k})$ for $i\in I_{n,k}$.  Here we keep in mind that such an
$\ell(\cdot;\cdot)$ depends on an external random quantity $X_{n+1}$.

The proof consists of two steps.
The first step is to show that, in the same spirit of \Cref{thm:4.1-clt-rand},
\begin{equation*}
    \hat p_{\rm cc}(X_{n+1}) = \frac{1}{K}\sum_{k=1}^K Q(X_{n+1},\D_{n,-k})+o_P(1)\,.
\end{equation*}
In fact, letting $\xi_i=\mathds{1}[s(X_i;\D_{n,-k_i})<s(X_{n+1};\D_{n,-k_i})]-Q(X_{n+1};\D_{n,-k_i})$, then we have
\begin{align}
 &    \left\|\hat p_{\rm cc}(X_{n+1})-\frac{1}{K}\sum_{k=1}^K Q(X_{n+1};\D_{n,-k})\right\|_2^2\nonumber\\
 =&\left\|\frac{1}{n}\sum_{i=1}^n \xi_i\right\|_2^2\nonumber\\
 =&\frac{1}{n^2}\sum_{i=1}^n\|\xi_i\|_2^2 +\frac{1}{n^2}\sum_{i\neq j}\mathbb E (\xi_i\xi_j)\nonumber\\
 \le & n^{-1}+\frac{1}{n^2}\sum_{i\neq j}\mathbb E(\xi_i\xi_j)\nonumber\\
 \lesssim & n^{-1}+\epsilon_n^{q/(1+q)}\,.\label{eq:5.5-step1}
\end{align}
To bound $\mathbb E\xi_i\xi_j$ in the final step of the inequality above, observe that $\mathbb E \xi_i\xi_j\neq 0$ only if $k_i\neq k_j$ and in this case using the ``free-nabla'' lemma (\ref{lem:4.4-free-nabla}), we have, using the bounded density assumption,
$$
|\mathbb E \xi_i\xi_j| = |\mathbb E \nabla_j\xi_i\nabla_i\xi_j|\le \|\nabla_i\xi_j\|_2\|\nabla_j\xi_i\|_2=O(\epsilon_n^{q/(1+q)})\,,
$$
where the last equation follows from \Cref{lem:5.5-stab-transfer}.

The next step is to show that 
$$\frac{1}{K}\sum_{k=1}^{K}Q(X_{n+1};\D_{n,-k})\rightsquigarrow U(0,1)\,.$$
By construction, for each individual $k$, $s(X_0;\D_{n,-k})$ and $s(X_{n+1};\D_{n,-k})$ have the same conditional distribution given $\D_{n,-k}$.  As a result we have
\begin{equation}\label{eq:5.5-cond-unif}
\left(Q(X_{n+1};\D_{n,-k})\bigg|\D_{n,-k}\right)\sim U(0,1)\,.
\end{equation}
Let $Q_n(x) = \mathbb E Q(x;\D_n)$. 
We have, using the martingale decomposition and Efron--Stein (\Cref{thm:4.1-Efron--Stein})
\begin{align*}
\|Q_n(X_{n+1};\D_n)-Q_n(X_{n+1})\|_2^2\le \sum_{i=1}^n \|\nabla_i Q(X_{n+1};\D_n)\|_2^2=O(n \epsilon_n^{2q/(2+q)})\,.
\end{align*}
Therefore, for any pair of $k,k'\in[K]$ we have
$$
\|Q(X_{n+1},\D_{n,-k})-Q(X_{n+1};\D_{n,-k'})\|_2=O(n^{1/2}\epsilon_n^{q/(2+q)})\,,
$$
which further implies that
\begin{equation}\label{eq:5.5-step2}
\left\|Q(X_{n+1};\D_{n,-1})-\frac{1}{K}\sum_{k=1}^{K}Q(X_{n+1};\D_{n,-k})\right\|_2=O(n^{1/2}\epsilon_n^{q/(2+q)})\,.    
\end{equation}
The claimed result follows from \eqref{eq:5.5-step1} and \eqref{eq:5.5-step2} because \eqref{eq:5.5-cond-unif} implies that $Q(X_{n+1};\D_{n,-1})\sim U(0,1)$.
\end{proof}

\begin{lemma}
    \label{lem:5.5-stab-transfer}
    Under \Cref{asm:5.5-conf-score-stab}, there exists a constant $c_1$ such that 
    \begin{enumerate}
        \item [(a)] $\|\nabla_1\mathds{1}[s(X_0;\D_n)<s(X_{n+1};\D_n)]\|_2\le c_1 \epsilon_n^{q/(2+2q)}$.
        \item [(b)]     $\|\nabla_1 Q(X_{n+1};\D_n)\|_2\le c_1\epsilon_n^{q/(2+q)}$.
    \end{enumerate}
\end{lemma}
\begin{proof}[Proof of \Cref{lem:5.5-stab-transfer}]
We use shorthand notation to denote $s(\cdot;\D_n)$ and $s(\cdot;\D_n^1)$ by
$s(\cdot)$ and $s'(\cdot)$, respectively.  Let $\Delta=|\nabla_1(s(X_0)-s(X_{n+1}))|$.

By definition
\begin{align}
    &\left|\nabla_1\mathds{1}[s(X_0;\D_n)<s(X_{n+1};\D_n)]\right|\nonumber\\
    =& \left|\mathds{1}[s(X_0)<s(X_{n+1})]-\mathds{1}[s'(X_0)<s'(X_{n+1})]\right|\nonumber\\
    \le & \mathds{1}\left[|s(X_0)-s(X_{n+1})|\le \Delta\right]\,,\label{eq:5.5-hxy}
\end{align}
because the binary difference is nonzero only if the
gap $s(X_0)-s(X_{n+1})$ is dominated by the perturbation, whose absolute value equals $\Delta$.

Now we prove Part (a) as follows.
For any $\epsilon>0$ we have
\begin{align*}
  &  \|\nabla_1\mathds{1}[s(X_0;\D_n)<s(X_{n+1};\D_n)]\|_2^2\\
  \le & \left\|\mathds{1}\left[|s(X_0)-s(X_{n+1})|\le \Delta\right]\right\|_2^2\\
  = & \mathbb P\left[|s(X_0)-s(X_{n+1})|\le \Delta\right]\\
  \le &\mathbb P(|s(X_0)-s(X_{n+1})|\le \epsilon)+\mathbb P(\Delta\ge \epsilon)\\
  \lesssim & \epsilon + \frac{\epsilon_n^q}{\epsilon^q}\\
  \lesssim &\epsilon_n^{q/(1+q)}\,,
\end{align*}
where the first inequality follows from \eqref{eq:5.5-hxy}, the second from union bound, the third from the bounded conditional density of $s(X_0)$ given everything except $X_0$, and the last by choosing $\epsilon=\epsilon_n^{q/(1+q)}$.

For Part (b), the proof is similar except that taking expectation over $X_0$ before taking the $L_2$ norm increases the stability.
For any $\epsilon>0$ we have
\begin{align}
    |\nabla_1 Q(X_{n+1};\D_n)| =&\left|\mathbb E_{X_0}\left\{ \mathds{1}[s(X_0)<s(X_{n+1})]-\mathds{1}[s'(X_0)<s'(X_{n+1})]\right\}\right|\nonumber\\
    \le & \mathbb E_{X_0}\mathds{1}[|s(X_0)-s(X_{n+1})|\le \Delta]\nonumber\\
    \le &\mathbb E_{X_0}\mathds{1}[|s(X_0)-s(X_{n+1})|\le \epsilon]+\mathds{1}(\Delta\ge \epsilon)\nonumber\\
    \le & c\epsilon + \mathds{1}(\Delta\ge \epsilon)\,.\nonumber
\end{align}
Now using triangle inequality we have
$$
\|\nabla_1 Q(X_{n+1};\D_n)\|_2\le c\epsilon + [\mathbb P(\Delta\ge \epsilon)]^{1/2}\le c\epsilon+\frac{\epsilon_n^{q/2}}{\epsilon^{q/2}}\,.
$$
The claimed result follows by choosing $\epsilon=\epsilon_n^{q/(2+q)}$.
\end{proof}


\subsection{Bibliographic Notes}\label{subsec:5.6-bib}

\subsubsection{Model Confidence Sets}
The model confidence set problem appeared in the statistics and econometrics literature in various forms, including the ``Fence Method'' \citep{jiang2008fence}, the penalized regression tuning confidence set \citep{gunes2012confidence}, model confidence sets via backward $F$-tests \citep{hansen2011model,ferrari2015confidence}.
Among these methods, the Fence Method is closest to the model confidence sets developed in \Cref{subsec:5.1-model-conf-set}, as it also computes a quality measure for each candidate model and accounts for the uncertainty of such measures.  The model confidence sets developed in  \Cref{subsec:5.1-model-conf-set} are applicable in high-dimensional settings and have more flexible applications. 
The idea of using normal approximations to assess the uncertainty of CV model selection appeared as early as \cite{yang2006comparing}, who also pointed out necessity of, and challenges in, using the difference instead of individual losses.
The CV model confidence sets based on high-dimensional normal approximations are rigorously studied in \cite{lei2020cross} for single sample splits and in \cite{kissel2022high} for $K$-fold CV under stability conditions.  In particular, \cite{kissel2022high} contains numerical examples demonstrating the advantage of the difference-based model confidence set ($\hat\Theta$ in \eqref{eq:5.1-cvc}) over  the original confidence set based on simultaneous confidence interval of individual models ($\hat\Theta_0$ in \eqref{eq:5.1-cvc-naive}).

The single-model selection estimator $\hat J_{\rm cvc}$, based on the ``most parsimonious element in the confidence set'', was proposed and studied in \cite{lei2020cross}, who provided a consistency proof of $\hat J_{\rm cvc}$ in a fixed-dimensional best subset selection setting for the difference-based CV confidence set.
The model selection estimator $\hat J_{\rm cvc}$ is a rigorously justifiable approach to correct for the overfitting issue of $\hat J_{\cv}$.  See \cite{efron1997improvements,tibshirani2009bias,lim2016estimation} for some examples of existing heuristic approaches.  Another method that improves the $K$-fold CV model selection consistency is the ``profile electoral college CV'' \citep[PEC-CV,][]{zhan2022profile}, which considers the winning frequency of each candidate model under repeated $K$-fold splits and varying values of $K$.  This method exhibits promising empirical performance and can also be used to visualize the level of model selection uncertainty.

As mentioned at the beginning of \Cref{subsec:5.2-cvc-consist}, another potential use of the model confidence set is to provide stopping rules for sequential model building procedures.  An early form of this application appeared in \cite{lei2020cross} in comparing linear vs quadratic models.  It would be possible to apply this idea to similar settings, such as variants of forward stepwise regression \citep{wieczorek2022model,kissel2024forward}.

\subsubsection{Testing Many Means}
The inference of maximum values in a vector considered in \Cref{subsec:5.3-argmin} is connected to many threads in the literature.  The sample-splitting and stabilization idea is partially inspired by the ``adaptive bootstrap'' approach in \cite{mckeague2015adaptive} for detecting marginal correlations. The connection to selective inference is straightforward; a general introduction to selective inference can be found in \cite{taylor2015statistical}. A more directly relevant selective inference work on inference for the max value is \cite{reid2017post}.

It is natural to extend the maximum value inference to high-dimensional mean testing, for which a vast literature exists; see \cite{huang2022overview} for a recent overview of this topic.  A particularly relevant work is \cite{chernozhukov2019inference}, who used high-dimensional Gaussian comparison to achieve asymptotically valid inference. 

The argmin confidence set problem considered in \Cref{subsec:5.4-argmin-conf} also has many connections in the literature and is a subject of active development.  The problem of identifying the optimal population in a finite collection from noisy observations has been studied as early as \cite{gibbons1977selecting,gupta1979multiple}, with further developments such as \cite{futschik1995confidence}, \cite{hansen2011model}, \cite{mogstad2024inference}, \cite{dey2024anytime}, and \cite{kim2025locally}.
Other related problems include discrete MLE \citep{hammersley1950estimating,choirat2012estimation} and rank verification \citep{hall2009using,xie2009confidence,hung2019rank}.

The testing-many-means problem and the related argmin inference problem can also be viewed from a nuisance parameter perspective.  Let  $\mathbb S_{m-1}$ be the $(m-1)$-dimensional simplex. Then for any $w\in\mathbb S_{m-1}$ define $\theta_w=\vartheta^T w\le \max\vartheta$. Then any valid confidence lower bound of $\theta_w$ is also a valid confidence lower bound for $\max\vartheta$.  The method developed in \Cref{subsec:5.3-argmin} uses a data-driven version of a particular $w^*$, the indicator of the maximum entry of $\vartheta$:
$$w^*=\arg\max_{w\in\mathbb S_{m-1}}\vartheta^T w\,.$$
This idea can be taken one step further: the current target $w^*$ aims to maximize $w^T\vartheta$ and does not account for the variance of $w^T X$.  A strictly more efficient target $w^\dag$ is
$$
w^\dag =\arg\max_{w\in\mathbb S_{m-1}}\frac{\vartheta^T w}{\sqrt{w^T\Sigma w}}\,,
$$
where $\Sigma=\mathbb E(X-\vartheta)(X-\vartheta)^T$ is the noise covariance matrix.  The objective function in the definition of $w^\dag$ above is known as the \emph{Sharpe ratio} \citep{palomar2025portfolio}.  It can also be modified to other similar forms such as
$$
w^\dag =\arg\max_{w\in\mathbb S_{m-1}}\vartheta^T w-\lambda w^T\Sigma w\,,
$$
for some penalty parameter $\lambda$.  Then it would be possible to extend the procedure in \Cref{subsec:5.3-argmin} to such weight vectors $w^\dag$, with a potentially improved bias-variance trade-off.

The softmax statistic $Q_i$ in \Cref{subsec:5.3-argmin} is an example of making and using stable algorithms for statistical inference.  The idea of using stable nuisance parameter estimates in semiparametric inference without de-biasing or cross-fitting has been explored in \cite{chen2022debiased}. In particular, the use of soft-min for stabilized optimal policy choice has been studied in \cite{chen2023inference}.   The presented method of data-driven choice of $\lambda$ works by numerically verifying the desired stability level.  An important theoretical and methodological research topic is to estimate and test the stability of a given algorithm.  This is challenging; see \cite{kim2023black,luo2024algorithmic} for some early results.  In the next section we present additional results on stability implied by regularization.

\subsubsection{Conformal Prediction}
For a general  introduction to conformal prediction, we refer the reader to the books \cite{VovkGS05,angelopoulos2024theoretical}.  Some early references on conformal prediction for outlier detection and regression include \cite{LeiRW13,LeiW14,lei2018distribution}.  The split conformal method is also known as \emph{inductive conformal prediction} \citep{papadopoulos2002inductive,LeiRW14}.  The cross-conformal prediction method was introduced by \cite{vovk2015cross} and analyzed in \cite{vovk2018cross}, achieving a coverage level of approximately $1-2\alpha$.   A variant of cross-conformal, called CV+, was introduced and analyzed in \cite{barber2021predictive}, with asymptotic coverage $1-\alpha$ under a similar but different set of stability  and bounded density conditions.  A key component in our proof for cross-conformal is the transfer of stability from the conformity score function to the ranking indicator function (\Cref{lem:5.5-stab-transfer}), where the inequality \eqref{eq:5.5-hxy} is adapted from a similar analysis in \cite{hu2024two}.

\newpage

\section{Stability of Estimators}\label{sec:6-stability}
In this section we examine the stability of several general classes of estimators. To simplify discussion, we will focus on the case where the loss function $\ell(x,f)$ is Lipschitz in $f$, so that the stability of loss function is implied by the stability of $\hat f$ itself.  Although we do not specify the parameter space $\mathcal F$, it would be helpful to think of it as either a finite-dimensional Euclidean space or a Hilbert space so that the gradients and Hessians can be defined in the usual sense.

So far, the least squares regression estimate in \Cref{subsec:5.2-cvc-consist} is the only example in which stability has been checked.  This example can be extended to the general class of asymptotically linear estimators.  Recall that an estimator $\hat f$ is \emph{asymptotically linear} with influence function $h(\cdot)$ if
$$
\hat f(\D_n) = f_0+\frac{1}{n}\sum_{i=1}^n h(X_i)+o_P(n^{-1/2})\,,
$$
where the $o_P(n^{-1/2})$ should be interpreted in terms of the norm of the residual vector.   Then for such asymptotically linear estimators, we have
$$\nabla_1 \hat f(\D_n)=\frac{1}{n}(h(X_1)-h(X_1'))+o_P(n^{-1/2})=o_P(n^{-1/2})\,,$$
provided that $\mathbb P(h(X_1)<\infty)=1$.

However, this is still heuristic, as $o_P(n^{-1/2})$ is still weaker than the $L_2$ and sub-Weibull stability conditions required in the previous sections.  Typically, establishing such stability conditions is nontrivial and may require specific tools depending on the model and estimator. For example, random matrix theory may be useful in establishing $L_2$-norm stability of the least squares estimate.

We will consider the more general and abstract approach of establishing stability of estimators via regularization.  For example, penalized $M$-estimators such as ridge regression and support vector machines are known for their stability due to their strongly convex optimization formulation.  In \Cref{subsec:6.1-stab-sgd-1,subsec:6.2-stab-sgd-2}, we make this more rigorous and general for the stochastic gradient descent algorithm.  In  \Cref{subsec:6.3-stab-diff}, we illustrate the subtlety of the stability of the difference loss functions (see \Cref{subsec:5.1-model-conf-set} for the context) in a prototypical nonparametric regression example, where the stability is achieved by another form of regularization: truncated basis expansion.  In \Cref{subsec:6.4-stab-online} we consider the stability conditions involved in the consistency analysis of online rolling validation in \Cref{subsec:3.3-rv}, focusing on a combination of stochastic gradient descent and nonparametric basis expansion, known as the sieve SGD (introduced in \eqref{eq:3.2-sieve-sgd} in \Cref{subsec:3.2-online-model-selection}).  Other related results in analyzing the stability of M-estimators, as well as alternative ways to construct stable estimators, such as resampling and bootstrap-aggregating, are discussed in \Cref{subsec:6.6-bib}.

\subsection{Stability of Stochastic Gradient Descent}\label{subsec:6.1-stab-sgd-1}
Consider the M-estimation problem
$$\hat f=\arg\min_{f\in\mathcal F}\sum_{i=1}^n \varphi(f;X_i)$$
for some objective function $\varphi:\mathcal F\times\X\mapsto \mathbb R$, and the stochastic gradient descent algorithm (SGD):
\begin{enumerate}
    \item Initialize $\hat f_0=0$
    \item For $t=0,...,n-1$, update $$\hat f_{t+1}=\hat f_t-\alpha_{t+1}\dot\varphi(\hat f_t;X_{t+1})\,,$$
  where $(\alpha_t:1\le t\le n)$ is a sequence of learning rates, and $\dot\varphi(\cdot;x)$ is the gradient of $\varphi(\cdot;x)$ viewed as a function of $ f$.
   \item The final estimate $\hat f(\D_n)=\hat f_n$.
\end{enumerate}
One may also consider the averaged final estimate $\hat f(\D_n)=n^{-1}\sum_{t=1}^n \hat f_t$. The stability results derived below can be carried over to the averaged case.  The SGD estimate may appear to violate the symmetry condition in \Cref{asm:symmetric-f}. This issue can be resolved by randomizing the order of sample points in the SGD update sequence.

The motivation of analyzing (approximate) M-estimators obtained by SGD is quite natural. First, M-estimators cover a wide range of parametric and nonparametric estimators such as the maximum likelihood estimator, the least squares estimator, local polynomial regression, their robust counterparts, as well as machine learning tools such as kernel regression, support vector machine, and neural networks.  Second, the SGD algorithm is arguably the most popular method for large-scale M-estimation.  Under convexity and smoothness conditions on the objective function $\varphi$, the simple form of SGD allows precise characterization of the algorithmic behavior of the estimated parameter and hence facilitates a general and abstract stability analysis. 

As seen in the ridge regression, the strong convexity of the objective function $\varphi(\cdot,x)$ plays an important role in the stability of the resulting estimator. 
\begin{definition}
    [Strong convexity]\label{def:strong_convex} Let $\gamma$ be a positive number. A function $h:\mathcal F\mapsto \mathbb R$ is $\gamma$-strongly convex if
    $h(u+v)-h(u)\ge \langle\dot h(u),~v\rangle+\frac{\gamma}{2}\|v\|^2$, where $\|\cdot\|$ denotes the Euclidean or Hilbert space norm.
\end{definition}

\begin{assumption}
    [Smoothness and convexity of objective function]\label{asm:6.1-smooth_convex}
    There exists positive constants $(\gamma, C_0, \beta)$ such that for all $x\in\X$
    \begin{enumerate}
        \item [(a)] $\varphi(\cdot;x)$ is $\gamma$-strongly convex.
        \item [(b)] $\varphi(\cdot;x)$ is $C_0$-Lipschitz
        \item [(c)] $\dot\varphi(\cdot;x)$ is $\beta$-Lipschitz        \end{enumerate}
\end{assumption}
Part (c) of \Cref{asm:6.1-smooth_convex} is also known as  ``$\beta$-smoothness''.  These are among the standard assumptions in the analysis of the convergence and stability of SGD.  In deep neural networks, it is unrealistic to assume convexity of $\varphi(\cdot;x)$, but we can assume that these conditions hold within a local convergence basin.  Parts (a) and (c) require that $\gamma\le \beta$. When $\gamma=\beta$, $\varphi(\cdot;x)$ is close to a quadratic function. The gap between $\gamma/\beta$ and $1$ reflects how close $\varphi(\cdot;x)$ is to quadratic; the conditions are more stringent when a large value of $\gamma/\beta$ is required.

\begin{example}[Ridge Regression]\label{exa:6.1-ridge} 
Consider $x=(y,z)$ and $\X\subseteq \mathbb B(r_x)$, $\mathcal F=\mathbb B(r_f)$ ($\mathbb B(r)$ denotes a centered ball with radius $r$.)
$$\varphi(f;x)=\frac{1}{2}(y-z^T f)^2+\frac{\lambda}{2}\|f\|^2\,.$$
Then $\varphi(\cdot;x)$ satisfies \Cref{asm:6.1-smooth_convex} with $\gamma=\lambda$, $C_0=r_x^2(1+r_f)+\lambda r_f$, and $\beta=r_x^2+\lambda$. \end{example}

In practice, the objective function $\varphi(\cdot,;x)$ usually is not strongly convex without a regularizing penalty term such as the ridge penalty.  It is still possible to establish similar stability results for SGD if $\mathbb E_X\varphi(\cdot;X)$ is strongly convex.  We provide a detailed example in \Cref{subsec:6.4-stab-online} for the sieve SGD estimator introduced in \Cref{subsec:3.2-online-model-selection}.  

The main consequence of \Cref{asm:6.1-smooth_convex} is the following contractivity result of SGD.
\begin{proposition}[Contractivity of SGD]
    \label{pro:6.1-contractive-sgd}
    Assume $\varphi$ satisfies parts (a), (c) of \Cref{asm:6.1-smooth_convex}, and $\alpha\in (0,2/(\beta+\gamma))$. 
    The SGD update mapping $G_{\varphi,\alpha,x}:\mathcal F\mapsto\mathcal F$
    $$
    G_{\varphi,\alpha,x}(f) = f -\alpha\dot\varphi(f;x)
    $$
    satisfies
    \begin{equation}\label{eq:6.1-sgd-contractivity}
    \|G_{\varphi,\alpha,x}(u)-G_{\varphi,\alpha,x}(v)\|\le \left(1-\frac{\beta\gamma}{\beta+\gamma}\alpha\right)\|u-v\|\,.
    \end{equation}
\end{proposition}
The proof of \Cref{pro:6.1-contractive-sgd}, though elementary, is non-trivial.  We will not provide a proof here.  See the bibliographic notes in \Cref{subsec:6.6-bib} for further references.  However, a one-line proof can be obtained, with a different contraction factor, if we assume that $\varphi(\cdot,x)$ is twice differentiable: Let $\ddot{\varphi}(\cdot,x)$ be the Hessian matrix, which is lower bounded by $\gamma I$ because of strong convexity. Then, by the mean value theorem we have
$G_{\varphi,\alpha,x}(u)-G_{\varphi,\alpha,x}(v) = (I-\alpha\ddot\varphi(w,x))(u-v)$
for some $w\in[u,v]$, where $[u,v]$ denotes the line segment joining $u$ and $v$.

\begin{theorem}[First-order stability of SGD]\label{thm:6.1-first-order-stab-sgd}
    Under \Cref{asm:6.1-smooth_convex}, if the learning rate satisfies $\alpha_t=\beta^{-1}t^{-a}$ for an $a\in(0,1)$ and
    $$\frac{\gamma}{\beta+\gamma}\ge \frac{a(1-a)}{1-2^{-(1-a)}}\frac{\log n}{n^{1-a}}\,,$$
    then for any $i\in [n]$ we have $$\left\|\nabla_i \hat f_n\right\|\le \frac{2^{1+a}C_0}{\beta}n^{-a}\,.$$
\end{theorem}

The intuition behind \Cref{thm:6.1-first-order-stab-sgd} is straightforward.
If $X_i$ is perturbed to $X_i'$, the SGD algorithm will remain the same up to time $i$, and future updates apply to the perturbed solution at time $i$. If $i$ is small, $X_i$ appears early in the SGD iteration, and the perturbation on $\hat f_i$ will shrink exponentially in later steps due to the regularity conditions in \Cref{asm:6.1-smooth_convex}.  If $i$ is large, the learning rate is already very small, limiting the perturbation from changing $X_i$ to $X_i'$.

In classical theoretical convex stochastic approximation, the most popular learning rates are $a\in(1/2,1]$, ensuring $\sum_t\alpha_t= \infty$ and $\sum_t \alpha_t^2<\infty$ \citep{nesterov2013introductory,bubeck2015convex}.  The range of $a$ allowed by \Cref{thm:6.1-first-order-stab-sgd} covers $\alpha\in(1/2,1)$ but not $\alpha=1$. This is due to the much weaker contraction upper bound, since $\sum_{i=1}^t i^{-a}$ grows polynomially in $t$ when $a<1$ while only logarithmically when $a=1$. A potential sharper stability result for $a=1$, or showing its absence, remains an open problem.

As mentioned in the intuition above, the proof of \Cref{thm:6.1-first-order-stab-sgd}
is a direct consequence of the following lemma, which establishes the multiple-step contractivity of SGD updates under \Cref{asm:6.1-smooth_convex}.
\begin{lemma}\label{lem:6.1-first-order-error}
Under \Cref{asm:6.1-smooth_convex} and the conditions in \Cref{thm:6.1-first-order-stab-sgd}, we have for all $t\ge i$
$$
\left\|\nabla_i\hat f_t\right\|\le\frac{2C_0}{\beta}i^{-a}\exp\left[-\frac{\gamma}{(1-a)(\beta+\gamma)}\left\{(t+1)^{1-a}-(i+1)^{1-a}\right\}\right]\,.
$$
\end{lemma}

\begin{proof}
    [Proof of \Cref{lem:6.1-first-order-error}]
    By construction,
$\|\nabla_i\hat f_i\|= \alpha_i\|\dot\varphi(\hat f_{i-1};X_i)-\dot\varphi(\hat f_{i-1};X_i')\|\le 2C_0\alpha_i = \frac{2C_0}{\beta}i^{-a}$.

By contractivity of SGD (\Cref{pro:6.1-contractive-sgd}), for all $t > i$,
\begin{align*}
    \|\nabla_i\hat f_t\|\le & \frac{2C_0}{\beta}i^{-a}\prod_{j=i+1}^t\left(1-\frac{\beta\gamma}{\beta+\gamma}\alpha_j\right)\\
    \le &\frac{2C_0}{\beta}i^{-a} \exp\left(-\frac{\beta\gamma}{\beta+\gamma}\sum_{j=i+1}^t \alpha_j\right)\\
    = &\frac{2C_0}{\beta}i^{-a} \exp\left(-\frac{\gamma}{\beta+\gamma}\sum_{j=i+1}^t j^{-a}\right)\\
    \le &\frac{2C_0}{\beta}i^{-a} \exp\left(-\frac{\gamma}{\beta+\gamma}\int_{x=i+1}^{t+1} x^{-a}dx\right)\\
    = & \frac{2C_0}{\beta}i^{-a}\exp\left[-\frac{\gamma}{(1-a)(\beta+\gamma)}\left\{(t+1)^{1-a}-(i+1)^{1-a}\right\}\right]\,.\qedhere
\end{align*}

\end{proof}
    \begin{proof}[Proof of \Cref{thm:6.1-first-order-stab-sgd}]
    If $i>n/2$, $$\|\nabla_i \hat f_n\|\le \|\nabla_i\hat f_{i}\|\le \frac{(n/2)^{-a}}{\beta}2C_0\,.$$
    If $i\le n/2$, according to \Cref{lem:6.1-first-order-error} with $t=n$ we have
    \begin{align*}
    \|\nabla_i\hat f_n\|\le & \frac{2C_0}{\beta}\exp\left[-\frac{\gamma}{(1-a)(\beta+\gamma)}\left\{(n+1)^{1-a}-(i+1)^{1-a}\right\}\right]\\
    \le &\frac{2C_0}{\beta}\exp\left[-\frac{\gamma}{(1-a)(\beta+\gamma)}(n+1)^{1-a}(1-2^{-(1-a)})\right]\\
    \le & \frac{2C_0}{\beta}\exp(-a\log n)\,.
    \qedhere
    \end{align*}
\end{proof}

\begin{remark}[Extension to Multi-epoch SGD]
    \label{rem:6.1-multi-epoch}
    In practice, SGD is more commonly implemented in multiple epochs.  The results and proof technique of single-epoch SGD presented above may be useful in establishing the stability of multi-epoch SGD algorithms.  For a simple example, consider a second epoch using the same learning rate sequence. Let $\hat f_t^{(e)}$ denote the estimate $\hat f$ at step $t$ in epoch $e$ with $e\in\{1,2\}$.
    Let $\hat f_t^{(e),i}$ be the corresponding estimate with $X_i$ replaced by $X_i'$ in both epochs, and $\tilde f_t^{(e),i}$ the estimate with $X_i$ in the first epoch and $X_i'$ in the second epoch.
    Then
    $\|\nabla_i \hat f_n^{(2)}\|\le \|\hat f_n^{(2)}-\tilde f_n^{(2),i}\|+\|\tilde f_n^{(2),i}-\hat f_n^{(2),i}\|$.
    \Cref{thm:6.1-first-order-stab-sgd} implies that $\|\hat f_n^{(2)}-\tilde f_n^{(2),i}\|\lesssim n^{-\alpha}$, and contractivity of SGD implies that $\|\tilde f_n^{(2),i}-\hat f_n^{(2),i}\|\le \|\nabla_i \hat f_n^{(1)}\|\lesssim n^{-a}$. Thus we conclude $\nabla_i \hat f_n^{(2)}\lesssim n^{-a}$.  The same argument extends to any finite number of epochs.  A more refined stability analysis of multi-epoch SGD under general learning-rate schemes is an interesting direction for future work.
\end{remark}


\subsection{Second-Order Stability}\label{subsec:6.2-stab-sgd-2}
The second-order stability concerns the quantity $\nabla_i\nabla_j\ell(X_0,\D_n)$ for $i\neq j\in[n]$, and is used in the deterministic centering CLT in \Cref{subsec:4.2-cv-clt-det}.  Although the discussion in \Cref{subsec:4.2-cv-clt-det} emphasizes that it is more important to compare $\nabla_i\nabla_j R(\D_n)$ with $\nabla_i R(\D_n)$,  an upper bound on $\nabla_i\nabla_j\ell(X_0,\D_n)$ would still be useful, as it provides an upper bound on $\nabla_i\nabla_j R(\D_n)$.
Now we control $\nabla_i\nabla_j\hat f$, which, by construction of SGD, equals $\nabla_j\nabla_i\hat f$. 

We first show how to control the second-order loss stability through the second-order parameter stability.
\begin{proposition}
    \label{pro:6.2-2nd-order-parameter-to-loss}
    If the loss function $\ell(x,f):\X\times\mathcal F\mapsto \mathbb R$ 
    is $c_0$-Lipschitz, and $\dot\ell(x,\cdot)$ is $c_1$-Lipschitz for all $x\in\X$,
    then
    $$|\nabla_i\nabla_j\ell(x,\hat f_n)|\le c_0\|\nabla_i\nabla_j\hat f_n\|+2c_1\epsilon_{n,i,j}^2\,,$$
    where $\epsilon_{n,i,j}=\max\{\|\nabla_j\hat f_n\|\,,~\|\nabla_i\hat f_n^j\|\,,~\|\nabla_j\hat f_n^i\|\}$, and
     $\hat f_n^i$ ($\hat f_n^j$) denotes the estimate obtained with $X_i$ ($X_j$) replaced by its i.i.d. copy.
\end{proposition}
One example of such $\epsilon_{n,i,j}$ is given in \Cref{lem:6.1-first-order-error}, which implies that $\epsilon_{n,i,j}^2$ has order $n^{-2a}$.  The main result in this subsection is to show that $\|\nabla_i\nabla_j\hat f_n\|$ can also be bounded by $n^{-2a}$ up to a logarithmic factor.

\begin{proof}
    [Proof of \Cref{pro:6.2-2nd-order-parameter-to-loss}]
    Let $\hat f_n^{ij}$ be the estimate obtained with $(X_i,X_j)$ replaced by their i.i.d. copies.
    By the mean value theorem, there exists $\tilde f\in[\hat f_n,~\hat f_n^i]$ and $\tilde f'\in[\hat f_n^j,~\hat f_n^{ij}]$ such that
    \begin{align*}
    |\nabla_i\nabla_j\ell(x,\hat f_n)| = & \left|\langle\hat f_n-\hat f_n^i\,,~\dot\ell(x,\tilde f)\rangle-\langle \hat f_n^j-\hat f_n^{ij}\,,~\dot\ell(x,\tilde f')\rangle\right|\\
    =& \left|\langle\nabla_i\nabla_j \hat f_n\,,~\dot\ell(x,\tilde f)\rangle +
    \langle \nabla_i \hat f_n^j\,,~ \dot\ell(x,\tilde f)-\dot\ell(x,\tilde f')\rangle\right|\\
    \le & c_0\|\nabla_i\nabla_j\hat f_n\| + c_1\|\nabla_i\hat f_n^j\| \cdot\|\tilde f-\tilde f'\|\,.
    \end{align*}
    The claimed result follows from the fact that
    \begin{align*}
        \|\tilde f-\tilde f'\|\le & \max\left\{\|\hat f_n-\hat f_n^j\|\,,~\|\hat f_n-\hat f_n^{ij}\|\,,~\|\hat f_n^i-\hat f_n^j\|\,,\|\hat f_n^i-\hat f_n^{ij}\|\right\}\le  2\epsilon_{n,i,j}\,. \qedhere
    \end{align*}
\end{proof}

Our main additional condition for the second-order stability of SGD estimates is the smoothness of Hessian of $\varphi(\cdot;x)$.
\begin{assumption}\label{asm:6.2-smooth-hessian}
 The Hessian matrix of $\varphi(\cdot;x)$, denoted $\ddot\varphi(\cdot;x)$, is $C_2$-Lipschitz for a constant $C_2$ and all $x$:
 $$
 \|\ddot\varphi(f;x)-\ddot\varphi(f';x)\|_{\rm op}\le C_2\|f-f'\|
 $$
 for all $f\,,f'\in \mathcal F$ and $\|\cdot\|_{\rm op}$ denotes the operator norm.
\end{assumption}
It may be possible to avoid such a smooth-Hessian condition using a similar co-coerciveness argument as in the proof of \Cref{pro:6.1-contractive-sgd}.  Here we work under \Cref{asm:6.2-smooth-hessian} for simplicity and transparency.

\begin{example}\label{exa:6.2-hessian} Ridge logistic regression: $x=(y,z)\in \mathbb B(r_x)$, and $ f\in \mathbb B(r_f)$.
$$\varphi( f;x)=-y (z^T f)+\log(1+e^{z^T f})+\frac{1}{2}\lambda\| f\|^2\,.$$
Then $\varphi$ satisfies \Cref{asm:6.1-smooth_convex} and \Cref{asm:6.2-smooth-hessian} with $\gamma=\lambda$, $C_0=r_x+\lambda r_ f$, $\beta=r_x^2/4+\lambda$, and $C_2=r_x^3$.
\end{example}

\begin{theorem}[Second-Order Stability of SGD]\label{thm:6.2-second-order-stab-sgd}
    Under the same assumptions as in \Cref{thm:6.1-first-order-stab-sgd}, assume in addition that $\varphi$ satisfies \Cref{asm:6.2-smooth-hessian} such that $\alpha_t\gamma\le 1$ for all $t$, then
    $$\|\nabla_i\nabla_j\hat f\|\le cn^{-2a}\log n\,,$$
    for a constant $c$ depending on $(a,\beta,\gamma,C_0,C_2)$.
\end{theorem}
The proof of \Cref{thm:6.2-second-order-stab-sgd} is lengthy and deferred to \Cref{subsec:6.5-proof}. Similar to the proof of first-order stability,  the key idea is to establish a contractivity result analogous to \eqref{eq:6.1-sgd-contractivity} for the second-order difference. This is enabled by the smoothness of the Hessian (\Cref{asm:6.2-smooth-hessian}).  A technical challenge is handling the remainder term of the first-order Taylor expansion of $\dot\varphi$, which must be tracked across all SGD iterations.

\subsection{Difference Loss Stability}\label{subsec:6.3-stab-diff}
According to \Cref{rem:5.1-difference_loss_standardize}, a technical condition needed for the validity of the difference-based CV model confidence set involves stability conditions on the standardized loss differences, which is used to establish high-dimensional Gaussian comparison for the cross-validated risk differences. This often becomes tricky because the loss differences usually have vanishing variances, and standardization means dividing the difference loss function by a vanishing standard deviation, which could make the stability condition less likely to hold.  However, a more optimistic argument is that this condition retains its original intuition: taking a $\nabla$ operator reduces the magnitude of the difference loss by more than a factor of $1/\sqrt{n}$, even though the difference loss itself may be vanishing.  In this subsection we illustrate this intuition through a simplified nonparametric regression problem using truncated basis expansion estimates.

Ignoring constant factors, we use the symbol ``$\lesssim$'' to denote ``upper bound up to a constant factor'' and ``$a\asymp b$'' for ``$a\lesssim b$ and $b\lesssim a$''.

Let $(X_i:1\le i\le n)$ be i.i.d. vectors with $X_i=(Y_i,Z_i)\in\mathbb R^1\times\mathbb R^{\infty}$, such that
\begin{equation}\label{eq:6.3-regression_example}
Y_i=\sum_{j=1}^\infty f_j Z_{ij} +\epsilon_i\,.
\end{equation}
The assumptions on the data distribution are the following.
\begin{assumption}
    \label{asm:6.3-reg-model-cond}
\begin{enumerate}
    \item $\epsilon_i$ is independent of $Z_i$, has mean zero, finite variance $\sigma_\epsilon^2$, and is uniformly bounded by a constant;
    \item  $Z_{ij}$ is mean zero, has unit variance, is uniformly bounded by a constant, and the components are mutually independent;
    \item  The regression coefficients $f_j\asymp j^{-\frac{1+a}{2}}$ for some $a>0$ for all $j\ge 1$.
\end{enumerate}
\end{assumption}
The conditions in \Cref{asm:6.3-reg-model-cond} are stronger than technically necessary and are made for a simpler presentation.
The uniform bound on $\epsilon$ and $Z$ can be relaxed to sub-Weibull tail conditions, and the independence between the entries of $Z$ can be relaxed to $\mathbb E(Z_{ij}|Z_{i1},...,Z_{i,j-1})=0$ for all $j$. The independence between $\epsilon$ and $Z$ can be relaxed in a similar fashion.
The assumed exact rate of $f_j$ in Part 3 is more stringent than the usual Sobolev ellipsoid condition which only assumes upper bound on $|f_j|$.  We need the lower bound on $|f_j|$ to track the bias and variance precisely.  More discussion can be found in \Cref{subsec:6.6-bib}. 

For each $j$, we can naturally estimate $f_j$ by
$$
\hat f_j=\frac{1}{n}\sum_{i=1}^n Y_i Z_{ij}\,.
$$
We consider a truncated estimate of $f=(f_j:j\ge 1)$. Let $J$ be a positive integer,
$$\hat f^{(J)}=(\hat f_j^{(J)}:j\ge 1)$$
where
\begin{equation}\label{eq:6.3-trunc-est}
\hat f_j^{(J)}\coloneqq \hat f_j \mathds{1}(j\le J)\,,~~\forall~j\ge 1\,.
\end{equation}
The tuning parameter $J$ controls the number of basis functions used in the truncated estimation.

The optimal choice of $J$ depends on the sample size.
For each candidate fitting procedure $r$, let $J_r$ be a positive integer varying with the sample size $n$ such that $\lim_{n\rightarrow\infty}J_{r}=\infty$ at a speed specified by the estimator $r$.  For example $J_{r}=n^{\tau_r}$, where $\{\tau_r: r\in[m]\}$ is a collection of positive numbers corresponding to different growth rates of $J_{r}$ as the sample size $n$ increases. 
We will use $\hat f^{(r)}$ to denote $\hat f^{(J_r)}$ for convenience and the estimators will be evaluated under the squared loss: $\ell((y,z),f)=(y-f(z))^2$.


Now we are ready to quantify the stability conditions for the difference loss $\ell^{(r,s)}(X_0,\D_n)$ and the difference risk $R^{(r,s)}(\D_n)$.  

\begin{proposition}[Difference Loss Stability]\label{prop:6.3-lossstab}
Assume the regression model \eqref{eq:6.3-regression_example} and the conditions for $\epsilon$, $Z$, $f$ in \Cref{asm:6.3-reg-model-cond}. If $J_r<J_s$ and $J_r/J_s$ either converges to $0$ or stay bounded away from $0$ as $n$ increases, then there is a positive constant $\kappa$ depending only on the model such that
\begin{itemize}
    \item[(a)] \begin{equation}\label{eq:6.3-first-order-verbose}
    \sqrt{n}\frac{\nabla_i\ell^{(r,s)}(X_0,\D_n)}{\sigma_n^{(r,s)}} \text{ is }
    \left[\left(\sqrt{\frac{J_sJ_r^a}{n}}+\frac{J_sJ_r^{a/2}}{n}\right)\kappa\,,~2\right]\text{-SW}\,;
\end{equation}
    \item[(b)] \begin{equation}
    \label{eq:6.3-risk-stab-verbose}
    n\frac{\nabla_i \left[R^{(r)}(\D_n)-R^{(s)}(\D_n)\right]}{\sigma_n^{(r,s)}} \text{ is }
    \left(\sqrt{\frac{J_sJ_r^a}{n}}\kappa\,,~2\right)\text{-SW}\,.
\end{equation}
\end{itemize}
\end{proposition}

The proof of \Cref{prop:6.3-lossstab} is deferred to \Cref{subsec:6.5-proof}.  The idealized analytical form of the model and the estimator enables us to track bias and variance terms precisely, as well as their perturb-one stability.

The optimal choice of $J$ with minimum risk is $J^*\asymp n^{1/(1+a)}$. Thus, for the scaling parameter on the RHS of \eqref{eq:6.3-first-order-verbose} and \eqref{eq:6.3-risk-stab-verbose} to be $o(1)$, one needs $J_sJ_r^{a}\ll n= (J^*)^{1+a}$.  This is possible if both $J_s$ and $J_r$ are no larger than $J^*$, or $J_s$ is slightly larger than $J^*$ but $J_r$ is sufficiently small. In other words, the stability condition is more likely to hold when comparing over-smoothing estimators.  This condition is violated if $J_r\ge J^*$, suggesting that the stability condition required by the CLT may not hold for under-smoothing estimators.  Such subtleties in the stability condition of the difference loss resemble the well-known behavior of the regular CV: it can eliminate inferior estimators with large bias (over-smoothing) but struggles to distinguish slightly overfitting estimators.

When $J_r/J_s$ stays bounded away from $0$, the sufficient stability condition required by the CLT becomes $J_s\ll J^*$, meaning that if the two estimators under comparison are too similar, the stability conditions can only hold for over-smoothing estimators.


\subsection{Stability of Online Algorithms}\label{subsec:6.4-stab-online}

In this subsection, we provide examples of the stability of online algorithms, specifically addressing \eqref{eq:3.3-wrv-stability} in \Cref{asm:3.3-rolling_stability}.
We focus on a simplified setting where the basis functions are given by the individual coordinates of the covariate vector $Z$, that is, $\phi_j(Z_i) = Z_{ij}$. Under this simplification, the basis expansion estimator in \eqref{eq:3.2-sieve-sgd} reduces to the setting considered in \Cref{subsec:6.3-stab-diff}.

\textbf{Notation Declaration:} In this subsection we will explicitly track the evolution of each coordinate of $\hat f$ over time.  Thus, we need to use the following notation.
\begin{itemize}
    \item $f=(f_j:j\ge 1)$ denotes the true vector of coefficients.
    \item $\hat f_i = (\hat f_{i,j}:j\ge 1)$ denotes the estimated vector of coefficients at the end of step $i$, whose $j$th coordinate is $\hat f_{i,j}$. Note that this notation $\hat f_i$ has a different meaning from the notation ``$\hat f_j$'' in \Cref{subsec:6.3-stab-diff} (where $\hat f_j$ denotes the $j$th coordinate of the estimated coefficient $\hat f$.)  Generally speaking, in this subsection the symbol ``$i$'' (and similarly ``$t$'' and ``$n$'') is used to denote a time point or sample size, whereas the symbol ``$j$'' is used to denote a coordinate of $Z$ or $f$.
    \item $\hat f_i'$ and $\hat f_{i,j}'$ correspond to the perturb-one version of $\hat f_i$ and $\hat f_{i,j}$ respectively, with the perturbed sample point clear from context.
    \item $(J_n:n\ge 1)$ is a sequence of positive numbers indicating the number of basis functions used by sieve SGD at sample size $n$. 
\end{itemize}

\paragraph{Batch truncated series estimate}
To warm up, 
first consider the truncated estimator $\hat f^{(J_n)}$ given in \eqref{eq:6.3-trunc-est}. The derivation from the previous subsection implies that 
$\mathbb E\left\{[\nabla_i\hat f_t(X_{t+1})]^2|\D_i\right\}\lesssim J_t t^{-2}$. Thus if $J_t\asymp t^{\tau}$ then \eqref{eq:3.3-wrv-stability} holds with $b=1-\tau/2$. If the true coefficient $f$ belongs to a Sobolev ellipsoid:
\begin{equation}\label{eq:6.4-sobo}
W(\varpi,L)\coloneqq\left\{f:~\sum_{j\ge 1}j^{2\varpi}f_j^2 \le L\right\}    
\end{equation}
for a pair of positive constants $(\varpi,L)$, then the minimax optimal error rate is achieved by $J_n\asymp n^{1/(1+2\varpi)}$.  Now consider two estimators with basis cardinality sequences $J_{n,1}\asymp n^{\tau_1}$ and $J_{n,2}\asymp n^{\tau_2}$ respectively, with $J_{n,1}$ being the better choice.   The sufficient stability condition for wRV to consistently select the better sequence $J_{n,1}$ is quite similar to that in \Cref{subsec:6.3-stab-diff}: $\tau_2$ must be either (i) smaller than $\tau^*=1/(1+2\varpi)$ while $\tau_1$ is equal to or slightly larger than $\tau^*$, 
or (ii) slightly larger than $\tau^*$ while $\tau_1$ is sufficiently small.  In other words, wRV, like CV, is theoretically more reliable when eliminating highly biased estimators (small $\tau_2$ versus moderate $\tau_1$), or when comparing two very different estimators (moderately large $\tau_2$ and small $\tau_1$).


\paragraph{Sieve SGD}
To illustrate a more challenging and less explored regime of online estimator stability, we 
consider the sieve SGD estimator defined in \eqref{eq:3.2-sieve-sgd}.  Due to its relatively recent introduction, theoretical analyses of convergence rates for general sieve SGD algorithms with non-optimal tuning values are not yet available. Thus, the stability analysis of sieve SGD is still an ongoing line of research.  Here we provide some initial results in this direction.

In the
simplified setting with data generated from \eqref{eq:6.3-regression_example} and basis $\phi_j(Z_i)=Z_{ij}$, the sieve SGD update \eqref{eq:3.2-sieve-sgd} reduces to
\begin{align}
    \hat f_0 = &0\,,\nonumber\\
    \hat f_{i} = & \hat f_{i-1} + \alpha_i (Y_i-\langle \hat f_{i-1},~Z_i\rangle) \Lambda_i Z_i\,,\quad i\ge 1\,,\label{eq:6.4-sieve-sgd-simple}
\end{align}
where $\alpha_i$ is the learning rate, $\Lambda_i$ is a diagonal truncation/shrinkage matrix whose $j$th diagonal entry equals $j^{-2w}\mathds{1}(j\le J_i)$, and $J_i$ is the number of basis functions used at sample size $i$.  As discussed in \Cref{subsec:3.2-online-model-selection}, the tuning parameter of sieve SGD is the entire sequence $(J_i:i\ge 1)$, which typically takes the form $J_i\asymp i^\tau$.   The multiplier $j^{-2w}$ provides coordinate-wise shrinkage, which was introduced in the original sieve SGD algorithm.

\begin{proposition}
    \label{pro:6.4-stab-sieve-sgd}
    Assume i.i.d. data from \eqref{eq:6.3-regression_example} satisfying \Cref{asm:6.3-reg-model-cond}.  Let $\{\hat f_i:i\ge 1\}$ be the sieve SGD estimate \eqref{eq:6.4-sieve-sgd-simple} with basis size $J_i=i^{\tau}$, learning rate $\alpha_i=ci^{-a}$, and shrinkage exponent $w>1/2$, such that $a+2w\tau<1$ and $c<\sum_{j\ge 1}j^{-2w}$.  Suppose $|Y_i-\langle \hat f_{i-1},~Z_i\rangle|$ is uniformly bounded by a constant almost surely for all $i\ge 1$ then we have, for any $t\ge i\ge 1$
    $$
    \mathbb E\left[\|\nabla_i \hat f_t(Z_{t+1})\|^2\bigg|\D_i\right]\lesssim t^{-2a}\,,~~\text{almost surely}\,.
    $$
\end{proposition}

\Cref{pro:6.4-stab-sieve-sgd} is an example of stability without strong convexity.  It should be compared to \Cref{thm:6.1-first-order-stab-sgd}, which shows a similar rate under strong convexity of the objective function. The argument relies on taking the conditional expectation of the SGD objective function given the previous estimate, and exploiting the strong convexity of this conditional expectation.  The proof of \Cref{pro:6.4-stab-sieve-sgd}, along with discussion of potential improvements, is presented in \Cref{subsec:6.5-proof}.

\begin{remark}
    \label{rem:6.4-bounded-residual}
  \Cref{pro:6.4-stab-sieve-sgd} involves several technical conditions, mainly for simplicity of presentation.  The most restrictive assumption is the uniform boundedness of the residual $Y_i-\langle \hat f_{i-1},~Z_i\rangle$. This can be ensured if we further restrict the estimates $\hat f_i$ to a bounded set such as a Sobolev ellipsoid, realizable by a further projection step at the end of each iteration.  Alternatively, one may obtain weaker versions of \Cref{pro:6.4-stab-sieve-sgd}, such as moment or sub-Weibull bounds instead of an $L_\infty$-norm bound.
\end{remark}

\begin{remark}
    \label{rem:6.4-implication}
In order to illustrate the relevance of \Cref{pro:6.4-stab-sieve-sgd} in the context of \Cref{thm:3.3-wrv-consistency}, we need to characterize the risk convergence rate of sieve SGD estimates obtained from general combinations of tuning parameters.  Such results are not yet directly available in the current literature due to the recent appearance of sieve SGD.  However, there is evidence that sieve SGD may have the same convergence rate as the batch mode truncated basis expansion estimate using $J_n$ basis functions at sample size $n$.  Proceeding with this hypothesis and assuming the true coefficient $f$ belongs to the Sobolev ellipsoid $W(\varpi,L)$ in \eqref{eq:6.4-sobo}, consider two sieve SGD estimators with $J_i = i^{\tau_r}$, $\alpha_i = c i^{-a}$ for $r=1,2$. Suppose $\tau_1=\tau^*=1/(1+2\varpi)$ is a better choice than $\tau_2$.  Using the batch mode error rate and stability bound in \Cref{pro:6.4-stab-sieve-sgd}, the stability requirements in \Cref{thm:3.3-wrv-consistency} reduce to
\begin{enumerate}
    \item $a>1/2+\varpi\tau_2$ when $\tau_2<\tau_1$ (under-fitting);
    \item $a>1-\tau_2/2$ when $\tau_2>\tau_1$ (over-fitting).
\end{enumerate}
A caveat here is that such choices of $(a, \tau_1, \tau_2)$ are incompatible with the requirement of $a+2w\tau <1$ in \Cref{pro:6.4-stab-sieve-sgd} unless we allow for $w<1/4$.   We further conjecture that this requirement is an artifact of the current proof and can be improved using a more refined argument.  Intuitively, a larger $w$ implies stronger shrinkage, hence greater stability.  Evidence supporting this conjecture is that one can establish consistency of sieve SGD with $w=0$. See \Cref{subsec:6.6-bib} for relevant literature.
\end{remark}

\subsection{Technical Proofs}\label{subsec:6.5-proof}
\begin{proof}[Proof of \Cref{thm:6.2-second-order-stab-sgd}]
     Assume $i<j$. Let $\hat f_t^{i}$ be the estimate at time $t$ with $X_i'$ as the $i$th sample point. Define $\hat f_t^{ij}$ correspondingly.

     Then
\begin{align}
\left\|\nabla_i\nabla_j\hat f_{j}\right\|= & \left\|\hat f_{j} - \hat f^i_{j} - \hat f^j_{j}+\hat f^{ij}_{j}\right\|\nonumber\\
=&\big\|\hat f_{j-1}-\alpha_{j}\dot\varphi(\hat f_{j-1};X_j) -
\hat f^i_{j-1}+\alpha_{j}\dot\varphi(\hat f^i_{j-1};X_j) \nonumber\\
&-\hat f_{j-1}+\alpha_{j}\dot\varphi(\hat f_{j-1};X_j') +
\hat f^i_{j-1}-\alpha_{j}\dot\varphi(\hat f^i_{j-1};X_j')\big\| \nonumber\\
=&\alpha_{j}\left\|\dot\varphi(\hat f_{j-1};X_j)-\dot\varphi(\hat f^i_{j-1};X_j)-\dot\varphi(\hat f_{j-1};X_j')+\dot\varphi(\hat f^i_{j-1};X_j')\right\|\nonumber\\
\le & 2\beta\alpha_{j}  \|\nabla_i\hat f_{j-1}\|\,,\label{eq:6.2-eta_j}\,,
\end{align}
where the last inequality follows from the Lipschitz property of $\dot\varphi(\cdot;x)$.

For $t>j$,
\begin{align}
  \nabla_i\nabla_j\hat f_{t}= & \hat f_{t} - \hat f^i_{t} - \hat f^j_{t}+\hat f^{ij}_{t}\nonumber\\
=&\hat f_{t-1}-\alpha_{t}\dot\varphi(\hat f_{t-1};X_t) -
\hat f^i_{t-1}+\alpha_{t}\dot\varphi(\hat f^i_{t-1};X_t) \nonumber\\
&-\hat f^j_{t-1}+\alpha_{t}\dot\varphi(\hat f^j_{t-1};X_t) +
\hat f^{ij}_{t-1}-\alpha_{t}\dot\varphi(\hat f^{ij}_{t-1};X_t) \nonumber\\
=&\nabla_i\nabla_j \hat f_{t-1} - \alpha_t\left[\dot\varphi(\hat f_{t-1};X_t)-\dot\varphi(\hat f_{t-1}^i;X_t)-\dot\varphi(\hat f_{t-1}^j;X_t)+\dot\varphi(\hat f_{t-1}^{ij};X_t)\right]\,.\label{eq:6.2-nabla2-f-eq1}
\end{align}
Using the mean value theorem, there exist $\tilde f_{t-1}\in[\hat f_{t-1},\hat f_{t-1}^i]$ and $\tilde f_{t-1}'\in[\hat f_{t-1}^j,\hat f_{t-1}^{ij}]$ such that
\begin{align*}
&  \dot\varphi(\hat f_{t-1};X_t)-\dot\varphi(\hat f_{t-1}^i;X_t)-\dot\varphi(\hat f_{t-1}^j;X_t)+\dot\varphi(\hat f_{t-1}^{ij};X_t)\\
=& \ddot\varphi(\tilde  f_{t-1};X_t)(\hat f_{t-1}-\hat f_{t-1}^i)-\ddot\varphi(\tilde f_{t-1}';X_t)(\hat f_{t-1}^j-\hat f_{t-1}^{ij})\\
=&\ddot\varphi(\tilde  f_{t-1};X_t) \nabla_i\nabla_j\hat f_{t-1}+\left[\ddot\varphi(\tilde f_{t-1};X_t)-\ddot\varphi(\tilde f_{t-1}';X_t)\right](\hat f_{t-1}^j-\hat f_{t-1}^{ij})\,.
\end{align*}

Plugging this into the RHS of \eqref{eq:6.2-nabla2-f-eq1}, we obtain
\begin{align*}
  \nabla_i\nabla_j\hat f_t = \left[I-\alpha_t \ddot\varphi(\tilde  f_{t-1};X_t)\right] \nabla_i\nabla_j\hat f_{t-1}+\alpha_t\left[\ddot\varphi(\tilde f_{t-1};X_t)-\ddot\varphi(\tilde f_{t-1}';X_t)\right](\hat f_{t-1}^j-\hat f_{t-1}^{ij})\,.
\end{align*}

Let $\eta_t=\|\nabla_i\nabla_j\hat f_t\|$, we get
\begin{equation}
\eta_t\le (1-\alpha_t\gamma)\eta_{t-1}+2C_2 \alpha_t\|\tilde  f_{t-1}-\tilde f_{t-1}'\|\cdot\|\hat f_{t-1}^j-\hat f_{t-1}^{ij}\|\,,\label{eq:6.2-eta-rec-1}
\end{equation}
where the second term on the RHS uses Lipschitz property of $\ddot\varphi$ (with Lipschitz constant $C_2$).

To control $\tilde f_{t-1}-\tilde f_{t-1}'$, apply \Cref{lem:6.1-first-order-error} directly to $\hat f_{t-1}-\hat f_{t-1}^j$ and $\hat f_{t-1}^i-\hat f_{t-1}^{ij}$, and then use triangle inequality on $\hat f_{t-1}-\hat f_{t-1}^{ij}=(\hat f_{t-1}-\hat f_{t-1}^j)+(\hat f_{t-1}^j-\hat f_{t-1}^{ij})$ and $\hat f^i_{t-1}-\hat f_{t-1}^{j} = (\hat f_{t-1}^i-\hat f_{t-1})+(\hat f_{t-1}-\hat f_{t-1}^j)$, then we have,
\begin{align*}
 & \|\tilde f_{t-1}-\tilde f_{t-1}'\|\\
 \le & \max\left\{\|\hat f_{t-1}-\hat f_{t-1}^j\|\,,~\|\hat f_{t-1}-\hat f_{t-1}^{ij}\|\,,~\|\hat f_{t-1}^i-\hat f_{t-1}^j\|\,,~\|\hat f_{t-1}^i-\hat f_{t-1}^{ij}\|\right\}\\
\le 
& \frac{2C_0}{\beta}\left(\delta_{t-1,i}(c)i^{-a}+\delta_{t-1,j}(c)j^{-a}\right)\,,
\end{align*}
where $c=\frac{\gamma}{(1-a)(\beta+\gamma)}$ and
$$
\delta_{t,j}(c) = \exp\left[-c((t+1)^{1-a}-(j+1)^{1-a})\right]\,.
$$

For $t\ge j+1$, define 
\begin{align*}
\varrho_{i,j,t}=&\frac{8C_2C_0^2}{\beta^3}t^{-a}\left[\delta_{t-1,i}(c)i^{-a}+\delta_{t-1,j}(c)j^{-a}\right]\delta_{t-1,i}(c)i^{-a}\,,
\end{align*}
then \eqref{eq:6.2-eta-rec-1} implies, for $t\ge j+1$,
\begin{equation}\label{eq:6.2-eta-rec-2}
\eta_{t}\le (1-\alpha_{t}\gamma)\eta_{t-1}+\varrho_{i,j,t}\,.
\end{equation}

We now seek to upper-bound $\eta_n$ using the recursion \eqref{eq:6.2-eta-rec-1}. To do so, let $c'=\gamma/[(1-a)\beta]\ge c$. Thus
\begin{align}
\eta_n \le & (1-\alpha_{n}\gamma)\eta_{n-1}+\varrho_{i,j,n}\nonumber\\
\le & (1-\alpha_{n}\gamma)\left[(1-\alpha_{n-1}\gamma)\eta_{n-2}+\varrho_{i,j,n-1}\right]+\varrho_{i,j,n}\nonumber\\
\le & \left[\prod_{k=j+1}^n(1-\alpha_k\gamma)\right] \eta_j + \sum_{k=j+1}^{n} \left[\prod_{l=k+1}^n(1-\alpha_l\gamma)\right] \varrho_{i,j,k}\nonumber\\
\le & \delta_{n,j}(c')\eta_j+
\sum_{k=j+1}^{n} \delta_{n,k}(c')\varrho_{i,j,k}\nonumber\\
=& \underbrace{\delta_{n,j}(c')\eta_j}_{(III)}+
\underbrace{\frac{8C_2C_0^2}{\beta^3}\sum_{k=j+1}^{n} \delta_{n,k}(c')\delta_{k-1,i}^2(c) k^{-a}i^{-2a}}_{(II)}\nonumber\\
&~~~+\underbrace{\frac{8C_2C_0^2}{\beta^3}\sum_{k=j+1}^{n} \delta_{n,k}(c')\delta_{k-1,i}(c)\delta_{k-1,j}(c) k^{-a}i^{-a}j^{-a}}_{(I)}\,,\label{eq:6.2-eta-rec-3}
\end{align}
where the last inequality appeals to the fact that for positive integers $j<t$, $\prod_{k=j+1}^t (1-\alpha_k \gamma)\le \exp\left[-\frac{\gamma}{(1-a)\beta}\left((t+1)^{1-a}-(j+1)^{1-a}\right)\right]$.

\textbf{Term (I):} Using the fact that $\delta_{k-1,i}(c)\ge \delta_{k-1,i}(c')$, $\delta_{n,k}(c)\delta_{k-1,i}(c)\le \delta_{n,i}(c)e^{c(2^{1-a}-1)}$, and $\delta_{k-1,j}(c)\le 1$, we have
\begin{align}
    (I)\le & \frac{8C_2C_0^2}{\beta^3}e^{c(2^{1-a}-1)}i^{-a}j^{-a}\delta_{n,i}(c)\sum_{k=j+1}^n k^{-a}\nonumber\\
    \le & \frac{8C_2C_0^2}{(1-a)\beta^3}e^{c(2^{1-a}-1)}i^{-a}j^{-a}\delta_{n,i}(c)(n^{1-a}-j^{1-a})\,.\label{eq:6.2-term-I-init}
\end{align}

Let $c_2=(3-a)/[c(1-a)]$.
When $i\le n-c_2 (n+1)^a \log n$, we have
\begin{align*}
  (n+1)^{1-a}-(i+1)^{1-a} = & (n+1)^{1-a}\left[1-\left(1-\frac{n-i}{n+1}\right)^{1-a}\right]\\
  \ge & (n+1)^{1-a}(1-a)\frac{n-i}{n+1}\\
  = & c_2(1-a) \log n\,,
\end{align*}
where the second lines uses $(1-x)^b\le 1-bx$ for $x\in[0,1)$ and $b\in(0,1)$. Thus 
\eqref{eq:6.2-term-I-init} implies
\begin{align}
   (I)\le & \frac{8 C_2C_0^2 }{(1-a)\beta^3}e^{c(2^{1-a}-1)}i^{-a}j^{-a}n^{-(3-a)}(n^{1-a}-j^{1-a})\nonumber \\
    \le & \frac{8 C_2C_0^2 }{(1-a)\beta^3}e^{c(2^{1-a}-1)} n^{-2}\,.  \label{eq:6.2-term-I-small-i}
\end{align}

When $i\ge n-c_2 (n+1)^a\log n$.  Suppose $n$ is large enough that $n-c_2(n+1)^a\log n\ge n/2$. Use the following upper bounds in the RHS of \eqref{eq:6.2-term-I-init}: 
$k^{-a}\le j^{-a}\le i^{-a}\le 2^{a} n^{-a}$, $\delta_{n,i}(c)\le 1$, we have
\begin{align}
   (I)\le  & \frac{8C_2 C_0^2 }{(1-a)\beta^3}e^{c(2^{1-a}-1)}2^{3a}n^{-3a}(n-j)\nonumber\\
    \le & \frac{8C_2 C_0^2 }{(1-a)\beta^3}e^{c(2^{1-a}-1)}2^{3a}n^{-3a}(n-i)\nonumber\\
    \le & \frac{8C_2 C_0^2 }{(1-a)\beta^3}e^{c(2^{1-a}-1)}2^{3a}n^{-3a}c_2(n+1)^a\log n\nonumber\\
    \le & \frac{8C_2 C_0^2 }{(1-a)\beta^3}e^{c(2^{1-a}-1)}2^{3a+1}c_2 n^{-2a}\log n\,.\label{eq:6.2-term-I-i-large}
\end{align}
Combining \eqref{eq:6.2-term-I-small-i} and \eqref{eq:6.2-term-I-i-large}, we conclude that
\begin{equation}
    \label{eq:6.2-term-I-final}
    (I)\le c(a,\beta,\gamma,C_0,C_2) n^{-2a}\log n\,,
\end{equation}
for $n$ large enough.

\textbf{Term (II):}
The second term in  \eqref{eq:6.2-eta-rec-3} is upper bounded by, for large enough $n$,
\begin{align}
  (II)\le &\frac{8 C_2 C_0^2 }{\beta^3}i^{-2a} \sum_{k=j+1}^n k^{-a} \delta_{n,k}(c)
 \delta_{k-1,i}(c)\nonumber\\
 \le & \frac{8 C_0^2 C_2}{(1-a)\beta^3} e^{c(2^{1-a}-1)}i^{-2a} \delta_{n,i}(c)(n^{1-a}-j^{1-a}) \nonumber\\
 \le & c(a,\beta,\gamma,C_0,C_2) n^{-2a}\log n\,,\label{eq:6.2-term-II-final}
\end{align}
where the first inequality follows from $\delta_{k-1,i}(c)\le 1$, the second and third follow from the same argument as in term (I).

\textbf{Term (III):}
\begin{align}
(III)=  &\delta_{n,j}(c')\eta_j\nonumber\\
\le & 2\beta\alpha_j \|\nabla_i\hat f_{j-1}\|\delta_{n,j}(c')\nonumber\\
\le & \frac{4C_0}{\beta}j^{-a}i^{-a}\delta_{j-1,i}(c')\delta_{n,j}(c')\nonumber\\
\le & \frac{4C_0}{\beta}e^{c(2^{1-a}-1)}j^{-a}i^{-a}\delta_{n,i}(c)\nonumber\\
\le & \frac{4C_0}{\beta}e^{c(2^{1-a}-1)}i^{-2a}\delta_{n,i}(c)\,,\label{eq:6.2-term-III-init}
\end{align}
where the first inequality follows from \eqref{eq:6.2-eta_j}, the second from \Cref{lem:6.1-first-order-error}, the third from $\delta_{j-1,i}(c)\ge \delta_{j-1,i}(c')$ and the same argument used in the treatment of term (I), and the fourth one from the fact that $j>i$. 

Finally, a similar argument partitioning $i$ according to $i\le n/2$ and $i\ge n/2$ as in the proof of \Cref{thm:6.1-first-order-stab-sgd} leads to the conclusion
\begin{equation}\label{eq:6.2-term-III-final}
	(III)\le c(a,\beta,\gamma,C_0,C_2) n^{-2a}\,.
\end{equation}
The proof concludes by plugging \eqref{eq:6.2-term-I-final}, \eqref{eq:6.2-term-II-final}, and \eqref{eq:6.2-term-III-final} into \eqref{eq:6.2-eta-rec-3}.
\end{proof}

\begin{proof}[Proof of \Cref{prop:6.3-lossstab}]
Denote
$$
\delta_j =\frac{1}{n}\sum_{i=1}^n Y_i Z_{ij}- f_j = \frac{1}{n}\sum_{i=1}^n\left[\epsilon_iZ_{ij}+f(Z_i)Z_{ij}- f_j\right]\,,
$$
which satisfies
$\mathbb E\delta_j=0$ and ${\rm Var}(\delta_j)=\sigma_j^2/n$ with $\sigma_j^2\in [\sigma_\epsilon^2,\sigma_\epsilon^2+c]$ for a constant $c$ depending only on $f$.

For  positive integers $u<v$, define $$
\| f_{u,v}\|^2=\sum_{j=u+1}^v  f_j^2\,.
$$
The condition $ f_j\asymp j^{-(1+a)/2}$ implies that $$\sqrt{u-v}\cdot v^{-(1+a)/2}\lesssim\| f_{u,v}\|\lesssim \sqrt{u-v}\cdot u^{-(1+a)/2}\,.$$

We use a shorthand notation for the difference of two loss functions $\ell^{(r)}(X_0,\D_n)$ and $\ell^{(s)}(X_0,\D_n)$:
\begin{align*}
T\coloneqq &(Y_0-\hat f^{(r)}(Z_0))^2-(Y_0-\hat f^{(s)}(Z_0))^2\\
 =& 2\epsilon_0(\hat f^{(s)}(Z_0)-\hat f^{(r)}(Z_0)) +\left[2f(Z_0)-\hat f^{(r)}(Z_0)-\hat f^{(s)}(Z_0)\right]\left[\hat f^{(s)}(Z_0)-\hat f^{(r)}(Z_0)\right]\,.
\end{align*}

We first provide a lower bound of $\sigma_n^{(r,s)}$, which, by definition, equals
$\sqrt{\var[\mathbb E(T|X_0)]}$.  Recall  that $\D_n=(X_i:1\le i\le n)$.
A straightforward algebra leads to
$$
\mathbb E(T|X_0) = 2\epsilon_0\sum_{j=J_r+1}^{J_s} f_j Z_{0j} + h(Z_0)
$$
where $h(Z_0)$ is a deterministic function of $Z_0$. As a result
\begin{align*}
  {\rm Var}\left[\mathbb E(T|X_0)\right]\ge& \var\left[2\epsilon_0\sum_{j=J_r+1}^{J_s} f_j Z_{0j}\right]=4\sigma_\epsilon^2\sum_{j=J_r+1}^{J_s}f_j^2\,,
\end{align*}
implying that
\begin{equation}\label{eq:6.3-sigma_ss_lower}
\sigma_{n}^{(r,s)} \ge 2\sigma_\epsilon\|f_{J_r,J_s}\|\,.
\end{equation}

Now take the $\nabla_i$ operator. To derive a simple expression of $\nabla_i T$, we first expand $T$:
\begin{align*}
  T = & 2\epsilon_0\sum_{j=J_r+1}^{J_s}\hat f_j Z_{0j}\\
  &+\left[-2\sum_{j=1}^{J_r}\delta_j Z_{0j}+\sum_{j=J_r+1}^{J_s}( f_j-\delta_j) Z_{0j}+2\sum_{j=J_s+1}^\infty  f_j Z_{0j}\right]\sum_{j=J_r+1}^{J_s}\hat f_j Z_{0j}\,,\end{align*}
  so that \begin{align*}
  \nabla_i T = & 2\epsilon_0\sum_{j=J_r+1}^{J_s}(\nabla_i \delta_j) Z_{0j}\\
 &+\left[-2\sum_{j=1}^{J_r}(\nabla_i\delta_j) Z_{0j}-\sum_{j=J_r+1}^{J_s}(\nabla_i\delta_j) Z_{0j}\right]\sum_{j=J_r+1}^{J_s}\hat f_j Z_{0j}\\
 &+\left[-2\sum_{j=1}^{J_r}\delta_j Z_{0j}+\sum_{j=J_r+1}^{J_s}( f_j-\delta_j) Z_{0j}+2\sum_{j=J_s+1}^\infty  f_j Z_{0j}\right]\sum_{j=J_r+1}^{J_s}(\nabla_i\delta_j) Z_{0j}\\
\eqqcolon& A_1+A_2+A_3\,,
\end{align*}
where $\nabla_i\delta_j=n^{-1}(Y_i' Z'_{ij}-Y_i Z_{ij})$.

Through a simple exercise using \Cref{thm:2.5-sw-mcdiarmid} and \Cref{thm:rio}, one can show that $Y_i$ is $(\kappa,1/2)$-SW, and $\delta_j$ is $(\kappa/\sqrt{n},1)$-SW for a constant $\kappa$ depending on the model only.

Then \Cref{pro:2.5-sum-prod-sw} and \Cref{thm:2.5-sw-mcdiarmid} further imply the following (with a potentially different value of constant $\kappa$):
\begin{align} A_1 & \text{ is } \left(\frac{\sqrt{J_s-J_r}}{n}\kappa,~1\right)\text{-SW}\,,\label{eq:6.3-A1}\\
   A_2 & \text{ is } \left[\frac{\sqrt{J_s}}{n}\left(\| f_{J_r,J_s}\|+\frac{\sqrt{J_s-J_r}}{\sqrt{n}}\right)\kappa\,,~2\right]\text{-SW}\,,\label{eq:6.3-A2}\\
   A_3 & \text{ is } \left[\left(\frac{\sqrt{J_s}}{\sqrt{n}}+\| f_{J_r,\infty}\|\right)\frac{\sqrt{J_s-J_r}}{n}\kappa\,,~2\right]\text{-SW}\,.\label{eq:6.3-A3}
\end{align}



In order to get concrete bounds for $\|f_{J_r,J_s}\|$, we first consider the scenario $J_r/J_s\rightarrow 0$ as $n\rightarrow\infty$. 

In the case $J_r/J_s\rightarrow 0$ we have 
\begin{align*}
\| f_{J_r,J_s}\|\asymp & J_r^{-a/2}\,.
\end{align*}
Plug this into \eqref{eq:6.3-sigma_ss_lower}, \eqref{eq:6.3-A1}, \eqref{eq:6.3-A2}, \eqref{eq:6.3-A3} we obtain
\eqref{eq:6.3-first-order-verbose}.

Now we work on the risk stability.  Under our model, for any truncation value $J$ we have
$$R(\D_n)=\sum_{j=1}^J \delta_j^2+\sum_{j=J+1}^\infty f_j^2\,,$$
and
$$
R^{(r)}(\D_n)-R^{(s)}(\D_n) = -\sum_{j=J_r+1}^{J_s} \delta_j^2+\sum_{j=J_r+1}^{J_s}f_j^2\,,
$$
hence
$$\nabla_i \left[R^{(r)}(\D_n)-R^{(s)}(\D_n)\right]=-2\sum_{j=J_r+1}^{J_s}\nabla_i(\delta_j^2)\,.$$

By definition
$$
\nabla_i(\delta_j^2) = \delta_j(\nabla_i\delta_j)+\delta_j^i(\nabla_i\delta_j)
$$
which is $(n^{-3/2}\kappa,~3/2)$-SW.  Now apply \Cref{thm:2.5-loocv-concentration-sw} to the sequence $(\sum_{j=J_r+1}^J\nabla_i(\delta_j^2):J_r<J\le J_s)$ with respect to the filtration $\{\D_n\cup\{X_{i1},...,X_{i,J}\}:J_r<J\le J_s\}$ to conclude that
$$
\nabla_i \left[R^{(r)}(\D_n)-R^{(s)}(\D_n)\right] \text{ is } \left(n^{-3/2}J_s^{1/2}\kappa\,,~2\right)\text{-SW}
$$
and hence \eqref{eq:6.3-risk-stab-verbose} follows.

In the other case, $J_r\ge  c J_s$ for a constant $c>0$. Now 
$\|f_{J_r,J_s}\|\asymp J_s^{-(1+a)/2}\sqrt{J_s-J_r}$ and the remaining proof is the same as in the  case of $J_r/J_s\rightarrow 0$.
\end{proof}

\begin{proof}[Proof of \Cref{pro:6.4-stab-sieve-sgd}]
Let $\eta_t=\nabla_i \hat f_t$.
For $t>i$, we have, by construction
\begin{equation}
    \label{eq:6.4-eta-recursion}
\eta_t = \eta_{t-1}-\alpha_t \Lambda_t Z_t Z_t^T\eta_{t-1}\,.
\end{equation}
Let $\tilde\eta_t=\Lambda_t^{-1/2}\eta_t$ and $\tilde Z_t = \Lambda_t^{1/2} Z_t$ (with $0/0=0$), then \eqref{eq:6.4-eta-recursion} becomes
\begin{equation}
    \label{eq:6.4-tilde-recursion}
    \tilde\eta_t=\tilde\eta_{t-1}-\alpha_t \tilde Z_t\tilde Z_t^T\tilde\eta_{t-1}\,.
\end{equation}

We first control $\|\tilde\eta_i\|$ in a similar fashion as in \Cref{lem:6.1-first-order-error}:
\begin{equation}
    \|\tilde\eta_i\| =  \left\|\alpha_i(Y_i-\langle \hat f_{i-1},~Z_i\rangle)\tilde Z_i-\alpha_i(Y_i'-\langle \hat f_{i-1},~Z_i'\rangle)\tilde Z_i'\right\|
    \lesssim  \alpha_i\,,    \label{eq:6.4-eta_i-bound}
\end{equation}
because, by assumption, $\|\tilde Z_i\|\le C$ for $w>1/2$ and $Y_i-\langle \hat f_{i-1},~Z_i\rangle$ is uniformly bounded.

Taking squared norm on both sides of \eqref{eq:6.4-tilde-recursion}, we have
\begin{align*}
    \|\tilde \eta_t\|^2 = & \|\tilde\eta_{t-1}\|^2 - (2\alpha_t-\alpha_t^2\|\tilde Z_t\|^2)\tilde\eta_{t-1}^T\tilde Z_t\tilde Z_t^T\tilde\eta_{t-1}^T\nonumber\\
    \le &\|\tilde\eta_{t-1}\|^2 - \alpha_t\tilde\eta_{t-1}^T\tilde Z_t\tilde Z_t^T\tilde\eta_{t-1}^T\,,
\end{align*}
because $\alpha_t\|\tilde Z_t\|^2\le 1$ by the choice of constant $c$.

Dividing both sides by $\|\tilde\eta_{t-1}\|^2$ and take conditional expectation given $\D_{t-1}$, we have
\begin{equation}
    \label{eq:6.4-square-norm-rec}
    \mathbb E\left[\frac{\|\tilde\eta_t\|^2}{\|\tilde\eta_{t-1}\|^2}\bigg|\D_{t-1}\right]\le 1-\alpha_t\frac{\tilde\eta_{t-1}^T\Lambda_t\tilde\eta_{t-1}}{\|\tilde\eta_{t-1}\|^2}\le 1-\alpha_t J_t^{-2w}\le e^{-\alpha_t J_t^{-2w}}\,,
\end{equation}
where the second inequality follows from that the minimum effective eigenvalue (those multiplied to the non-zero entries of $\tilde \eta_{t-1}$) is at least $J_t^{-2w}$.

Now using the telescoping version of \eqref{eq:6.4-square-norm-rec} we obtain
\begin{align*}
    \mathbb E\left[\|\tilde\eta_t\|^2\big|\D_i\right]\le &\|\tilde\eta_i\|^2\exp\left(-\sum_{l=i+1}^t\alpha_l J_l^{-2w}\right)\nonumber\\
    =&\|\tilde\eta_i\|^2\exp\left(-c\sum_{l=i+1}^t l^{-a-2w\tau}\right)\nonumber\\
    \lesssim &i^{-2a}\exp\left[-c'\left((t+1)^{1-a-2w\tau}-(i+1)^{1-a-2w\tau}\right)\right]\,.
\end{align*}
The claim follows from the same argument as in \Cref{lem:6.1-first-order-error} by considering $i\le t/2$ and $i\ge t/2$.
\end{proof}

\begin{remark}[Potential improvements]
Inspecting the proof of \Cref{pro:6.4-stab-sieve-sgd}, the main source of potential (and likely) looseness responsible for the counterintuitive dependence on $w$ seems to be bounding 
$\tilde\eta_{t-1}^T\Lambda_t\tilde\eta_{t-1}\ge J_t^{-2w}\|\tilde\eta_{t-1}\|_2^2$
in \eqref{eq:6.4-square-norm-rec}.  The diagonal entries of $\Lambda_t$ have very different magnitudes. On the other hand, the shrinkage mechanism in the sieve SGD update suggests that $|\hat f_{t,j}|$ decays at a rate of $j^{-2w}$ as $j$ increases.  Thus $\tilde\eta_{t-1,j}=j^w(\hat f_{t-1,j}-\hat f_{t-1,j}')$ should carry a scaling of $j^{-w}$.  So intuitively, in the product $\tilde\eta_{t-1}^T\Lambda_t\tilde\eta_{t-1}$, large entries in $\Lambda_t$ are aligned with large entries of $\tilde\eta_{t-1}$ and it is overly conservative to use a lower bound based on the smallest positive eigenvalue of $\Lambda_t$.
\end{remark}
\subsection{Bibliographic Notes}\label{subsec:6.6-bib}
Stability of regularized M-estimators that are suitable for the conditions in \Cref{sec:bousquet} and \Cref{sec:4-clt} has been established in \cite{bousquet2002stability} and \cite{austern2020}, respectively.  These results are concerned with the global optimum of the M-estimation problem, without taking into account the optimization error. First-order stability of SGD under strong convexity and smooth gradients has been established in \cite{hardt2016train}, which is covered as a special case of our \Cref{thm:6.1-first-order-stab-sgd}.  Our \Cref{pro:6.1-contractive-sgd} corresponds to Lemma 3.7.3 of \cite{hardt2016train}, which provides an elegant proof. The second-order stability of SGD presented in \Cref{subsec:6.2-stab-sgd-2} was first developed in \cite{kissel2022high}.

The stability of (the standardized) difference loss function considered in \Cref{subsec:6.3-stab-diff} is adapted from \cite{kissel2022high}, which contains numerical examples of the first and second-order stability of the difference loss function.  The stability of the difference risk function in part (b) of \Cref{prop:6.3-lossstab} is new and will lead to a deterministic centering CLT for the CV estimate of expected difference risks.

As an example illustrating the necessity of part 3 of \Cref{asm:6.3-reg-model-cond},
\cite{zhang2023targeted} considered a similar nonparametric regression setting, where the coefficients $f_j$ alternate between zero and $j^{-2}$ according to $\lfloor\log_3[\log_2(j)]\rfloor$ being even or odd. In that paper, the purpose was to study the ability of cross-validation to pick the correct model from two candidates $J_r$ and $J_s$ with $J_r=J_s-1$.  If we use this sequence $(f_j:j\ge 1)$ in the setting of \Cref{prop:6.3-lossstab} and assume $J_s\ll n^{1/5}$, then the proof can be adapted to show that  $\nabla_1\ell^{(r,s)}(X_0,\D_n)/\sigma_n^{(r,s)} =o_P(1)$ if and only if $f_{J_s}\neq 0$.
Such a dichotomy originates from the variance of the loss difference $\ell^{(r,s)}(X_0,\D_n)$. When $f_{J_s}=0$, the difference between the two true models is zero and any variance in the loss difference will be from the random fluctuation in the estimated coefficient $\hat f_{J_s}$.  On the other hand, when $f_{J_s}\neq 0$, there is a bias term contributing to the loss difference.  We do not expect Gaussian approximation to hold for the former case due to the severe degeneracy of the loss difference.

The original idea of sieve SGD first appeared in \cite{dieuleveut2016nonparametric}.
The form of sieve SGD considered in this article was proposed in \cite{zhang2022sieve}, which establishes its near optimality in terms of error rates and storage cost (see also \cite{quan2024optimal}).  
A version of sieve SGD implemented without shrinkage ($w=0$) has been shown to achieve the optimal convergence rate in \cite{chen2024stochastic}, providing evidence for potential improved stability analysis with better dependence on the shrinkage parameter.  
The stability property in \Cref{pro:6.4-stab-sieve-sgd} is adapted from \cite{zhang2023online}.

Another important general method for designing stable algorithms is resampling.  In \cite{buhlmann2002analyzing}, it has been shown that the bootstrap-aggregating (bagging) method can turn unstable algorithms such as decision trees into stable ones.  Similar bagging ideas are also used in stability-based variable selection \citep{meinshausen2010stability,liu2010stability}. A more systematic study of bagging and resampling methods for algorithmic stability in the context of predictive inference has been initiated in \cite{soloff2024bagging,soloff2024stability}.

\newpage

\section{Miscellaneous}\label{sec:7-misc}
\subsection{Repeated CV}\label{subsec:7.1-rep-cv}
A popular and useful variant of \(K\)-fold CV is \emph{repeated} \(K\)-fold CV, in which the fold partition is regenerated multiple times and the model selection decision is based on the average of the CV risk estimates across repetitions. Early versions of this idea include ``repeated learning–testing'' \citep{burman1989comparative} (see also \cite{Geisser75}) and ``multifold cross-validation'' \citep{Zhang93}. It is straightforward to verify that averaging the CV risks over multiple random fold splits still yields the same unbiased estimator of \(\mu_{n_{\mathrm{tr}}}\), while averaging can reduce the variance of the risk estimate. Quantifying this variance reduction formally remains an interesting open problem. The main technical challenge is to account for the additional correlation of fitting and testing samples across different rounds of fold splits.

\cite{zhang2015cross} proposed another use of repeated CV: instead of averaging CV risk estimates, consider each candidate model’s \emph{selection frequency} (the proportion of fold splits in which it wins). This voting scheme has been shown to achieve risk consistency in several nonstandard and challenging settings, including choosing between AIC and BIC when the underlying model class can be parametric or nonparametric, and in high-dimensional regimes where the quality gap between two models shrinks as the sample size increases; see also \cite{zhan2022profile} for extensions to multiple choices of \(K\).

A corresponding extension of the CV model confidence sets in \Cref{subsec:5.1-model-conf-set} would compute a separate confidence set for each split. Aggregating multiple confidence sets in a principled and computationally efficient way is nontrivial; see \cite{gasparin2024merging} for a recent attempt.

\subsection{Choice of $K$}\label{subsec:7.2-choose-k}
A practical issue in using cross-validation is the choice of \(K\): which value of \(K\) yields the best model selection? At first glance, the question seems somewhat circular: we plan to use \(K\)-fold CV to select the best model (or, more generally, a fitting procedure), yet doing so appears to require first selecting \(K\)---itself a modeling choice. In other words, to solve the model selection problem, we would first need to solve an even more complicated model selection problem. This observation suggests that there is no single, definitive answer to the choice of \(K\) for cross-validation. The aim of this subsection is therefore to collect a few simple facts, known concerns, and practical rules of thumb about selecting \(K\).

For model selection consistency, it is often recommended to use \emph{reversed} \(K\)-fold CV (see \Cref{rem:3.1-arbitrary-split}) with a diverging \(K\), so that most of the sample is used for evaluation \citep{yang2007consistency,zhang2015cross,arlot2010survey}. However, this scheme is effective only in ``strong-signal'' settings, where the candidate models are well separated (e.g., in the sense of \Cref{asm:3.1-stochastic_domination}). The caveats of this training--validation split are clear: (i) it does not handle the important case in which the difference between models is small; (ii) the quality of models fitted on a much smaller training sample may not reflect their quality at the original sample size. This challenge is described succinctly in \cite{zhang2015cross}:
\begin{quote}
    Simply speaking, in terms of comparing the predictive performances of two modeling procedures, a large enough evaluation set is preferred to account for the randomness in the prediction assessment, but at the same time we must make sure that the relative performance of the two model selection procedures at the reduced sample size resembles that at the full sample size. This typically forces the training size to be not too small. Therefore, the choice of splitting ratio needs to balance the above two conflicting directions.
\end{quote}

How should we choose \(K\) in practice? We suggest the following rule of thumb:
\begin{itemize}
\item Use leave-one-out CV (LOO-CV) if computationally feasible;
 \item  Otherwise, use \(10\)-fold CV if feasible.
  \item Otherwise, use \(5\)-fold CV.%
\end{itemize}
The rationale is simple: we aim to keep the training sample size as close as possible to the full sample size. In many tuning-parameter selection problems, the optimal choice scales as a monomial of the sample size with an exponent between \(-1\) and \(1\). Thus the optimal choices based on the full sample size \(n\) and the training size \(n_{\mathrm{tr}}\) differ by a factor between \(1-1/K\) and \(K/(K-1)\), which plateaus quickly once \(K\) is moderately large. \cite{arlot2016choice} provide a more thorough theoretical investigation of the choice of \(K\) in least-squares density estimation, reaching a similar conclusion: the variance of a quantity central to model selection depends on \(K\) through the factor \(1 + 4/(K-1)\). Hence \(K=5\) is a good balance between computational cost and statistical accuracy. 

Meanwhile, the lack of model selection consistency for these choices of \(K\) (\(K=n,10,5\)) can be addressed via the model confidence set approach in \Cref{sec:5-applications}, or by alternatives such as repeated-CV voting \citep{zhang2015cross,zhan2022profile}.

\subsection{The Rashomon Effect and Occam Dilemma}\label{subsec:7.3-rashomon}
The uncertainty of CV risk estimation and the associated model selection confidence set are closely related to the \emph{Rashomon effect} and the \emph{Occam dilemma}---two of the three major changes in modern statistical learning highlighted in Breiman’s landscape paper \cite{breiman2001statistical}\footnote{The third change, \emph{Bellman}, refers to the blessing (or curse) of dimensionality.}:
\begin{itemize}
    \item [-] the \emph{Rashomon effect} refers to the existence of many models with similar, nearly optimal predictive accuracy;
    \item [-] the \emph{Occam dilemma} refers to the tension between model simplicity and predictive accuracy.
\end{itemize}
See \cite{rudin2025leo} for a more detailed discussion and a comprehensive literature review. Here we focus on the connection to the present article: cross-validation and stability.

\paragraph{The Rashomon Effect}  
The presence of multiple nearly optimal models has been documented by many researchers over the past decade. For example, \cite{nevo2017identifying,fisher2019all} estimate the \emph{Rashomon set}---the collection of candidate models that are nearly optimal---by including all models whose empirical risk is within a small tolerance of the empirical risk minimizer, and then use this set to study variable importance. These methods require a user-specified tolerance, whose choice can substantially affect the output. The CV model confidence set developed in \Cref{subsec:5.1-model-conf-set}, especially the version based on the difference loss, serves a similar purpose but with a more interpretable tuning parameter---the miscoverage level \(\beta\).

The Rashomon set can also be used to construct more stable predictors with comparable near-optimal accuracy. For instance, one may average the models in the Rashomon set: because each constituent is near-optimal, predictive performance is preserved, and stability often improves through averaging. This idea can be refined via weighted averages, choosing weights to optimize stability or variance, or adapting them to covariates when models perform differently across regions of the feature space.

\paragraph{The Occam Dilemma}
The potential conflict between simplicity and predictive accuracy is of practical importance, since interpretability and accuracy are both core objectives in statistical learning and model selection. When the Rashomon set contains a simple model, the decision is straightforward: choose the simple model(s) in the set---an instance of Occam’s razor, and consistent with the perspective in \Cref{subsec:5.2-cvc-consist}. \cite{semenova2022existence,semenova2023path} develop methods to assess whether a simple model exists within the Rashomon set and to find one via a sequence of Rashomon sets induced by progressively reduced hypothesis spaces.

The Occam dilemma is more of a concern when none of the models in the Rashomon set is simple, indicating that the underlying relationship between the response and the covariates cannot be captured by the simple candidate models under consideration.
In such cases, it is reasonable to \emph{prioritize predictive accuracy over model simplicity and interpretability}.
Even when no simple model is available, we can still aim for the selected complex model to be \emph{stable}. One reason to prefer simplicity is its estimation stability, which reduces variance in both estimation and prediction. A complex model that is stable can likewise enjoy a favorable bias-variance trade-off. Random forests and other bagging methods are prominent examples. Thus, stability in a complex model can partially offset the loss of simplicity. A stable and accurate (albeit complex) model is worth preserving: it represents our best approximation to future observations and may even guide the discovery of a simpler, equally accurate model. Hence, stability offers a principled bridge between predictive
accuracy and interpretability.

  \bibliographystyle{plainnat}
  \bibliography{cv-stab}
\end{document}